\documentclass[a4paper,oneside,12pt] {amsart}

\reversemarginpar  

 \usepackage[english]{babel}

\usepackage{amscd}
\usepackage{amssymb}
\usepackage{color}
\usepackage{amsmath}

 \date{\today}

\newtheorem{theorem}{Theorem}[section]
\newtheorem{lemma}[theorem]{Lemma}

\newtheorem{proposition}[theorem]{Proposition}
\newtheorem*{theorem_}{Theorem}

\newtheorem*{definition}{Definition}
\newtheorem*{definitions}{Definitions}
\newtheorem*{remark}{Remark}
\newtheorem*{remarks}{Remarks}

\theoremstyle{definition}

\numberwithin{equation}{section}

\newcommand{\Ker}{\operatorname{Ker}}
\newcommand{\Ran}{\operatorname{Ran}}
\newcommand{\rank}{\operatorname{rank}}
\newcommand{\supp}{\operatorname{supp}}
\newcommand{\dist}{\operatorname{dist}}
\newcommand{\sign}{\operatorname{sign}}

\newcommand{\ima}{\operatorname{Im}}
\newcommand{\const}{\operatorname{const}}
\newcommand{\clos}{\operatorname{clos}}

\newcommand{\tk}{t_{n_k}}
\newcommand{\tj}{t_{m_j}}

\newcommand{\A}{\mathcal{A}}
\newcommand{\LL}{\mathcal{L}}

\newcommand{\F}{A}
\newcommand{\G}{B}

\newcommand\defin {\overset {\text {\rm def} }{=}}
\newcommand\precd {\overset {d}{\prec}}

\newcommand\beqn{\begin{equation}}
\newcommand\neqn{\end{equation}}

\newcommand\nuu{{\beta}} 
\newcommand{\taauu}{\tau}

\newcommand{\he}{\mathcal{H}(E)}

\newcommand{\al}{\alpha}   
\newcommand{\be}{\beta}   
\newcommand{\de}{\delta}   
\newcommand{\ga}{\gamma}   
\newcommand{\deab}{\varkappa}   
\newcommand{\om}{\omega}   
\newcommand{\la}{\lambda}
\newcommand{\eps}{\varepsilon}
\newcommand{\si}{\sigma}   
\newcommand{\wt}{\widetilde } 
\newcommand{\wh}{\widehat } 
\newcommand{\cO}{\mathcal{O} } 
\newcommand{\cD}{\mathcal{D}}
\newcommand{\M}{\mathcal{M}}
\newcommand{\cH}{\mathcal{H}}
\newcommand{\BR}{\mathbb{R} } 
\newcommand{\BN}{\mathbb{N} }

\newcommand{\ta}{\widetilde{\mathcal{A}}}
\newcommand{\tl}{\widetilde{\mathcal{L}}}
\newcommand{\taa}{\tilde a}
\newcommand{\tb}{\tilde b}
\newcommand{\BZ}{\mathbb{Z}}
\newcommand{\Gr}{G}    

\newcommand{\I}{I}

\newcommand{\RR}{\mathbb{R}}
\newcommand{\BC}{\mathbb{C}}
\newcommand{\NN}{\mathbb{N}}
\newcommand{\sm}{\setminus}
\renewcommand{\phi}{\varphi}

\renewcommand{\Im}{\operatorname{Im}}
\renewcommand{\kappa}{\varkappa}

\newcommand{\cDA}{\cD(\A)}
\newcommand{\cDL}{\cD(\LL)}
\newcommand{\cDLst}{\cD(\LL_*)}

\newcommand{\admone}{{\rm A}}
\newcommand{\admonest}{{\rm A}^*}
\newcommand{\admn}{{\rm A}_n}
\newcommand{\admnst}{{\rm A}^*_n}

\newcommand{\Ellone}{L^1}
\newcommand{\Elltwo}{L^2}

\begin{document}
\sloppy

\title[Completeness of nonselfadjoint perturbations]
{ 
Completeness and spectral synthesis \\
of nonselfadjoint one-dimensional perturbations
\\ of selfadjoint operators
}


\author{Anton D. Baranov}
\address{
Department of Mathematics and Mechanics,
Saint Petersburg State University,
28, Universitetski pr., St. Petersburg, 198504, Russia}
\email{a.d.baranov@spbu.ru}

\author{Dmitry V. Yakubovich}
\address{Departamento de Matem\'{a}ticas,
Universidad Autonoma de Madrid, Cantoblanco 28049 (Madrid) Spain
\newline
\phantom{r} and
\newline
Instituto de Ciencias
Matem\'{a}ticas (CSIC - UAM - UC3M - UCM)}
\email{dmitry.yakubovich@uam.es}
\thanks{The results of Sections 2-5 were obtained as a part of the
Project MTM2015-66157-C2-1-P and by the ICMAT Severo Ochoa project
SEV-2015-0554 of the Ministry of Economy and Competition of Spain.
Theorems 1.3-1.6 and the material of Sections 6-8 were obtained
with the support of the Russian Science Foundation grant
14-21-00035.}

\begin{abstract}
We study spectral properties of nonselfadjoint rank one perturbations
of compact selfadjoint operators. The problems under consideration
include completeness of eigenvectors,
relations between completeness of the perturbed
operator and its adjoint, and the
spectral synthesis problem.
We obtain new criteria for completeness and spectral synthesis
in this class as well as
a series of counterexamples which show that the spectral structure
of rank one perturbations is, in general, unexpectedly rich and complicated.

A parallel spectral theory is developed for one-dimensional singular
perturbations of unbounded selfadjoint operators.
Our approach is based on a functional model for this class which translates
the properties of operators to completeness problems
for systems of reproducing kernels and their biorthogonals in some spaces
of analytic (entire) functions.
\medskip

\noindent {\bf M.S.C.(2010):}  Primary: 34L10 
47B32, 
47A55;  
Secondary:
47A45. 


\end{abstract}
\maketitle

{}
\vskip-2cm

%
%
%
%
%
%
%
%
%
%
%
%
%
%



\section{Introduction and main results}

A Banach or Hilbert space linear operator will be called \textit{complete}
if it has a complete set of eigenvectors and root vectors.
Despite a large effort devoted to the study of
criteria of completeness of operators, our understanding still is
far from perfect, even for the case of ordinary
differential operators.

\subsection{Theorems of Keldy\v s and Macaev}
Most general abstract completeness results are due to
Keldy\v s~\cite{Keldysh}, \cite{Keldysh71}
and Macaev~\cite{Mats61} (see, also,~\cite[Chapter V]{Gohb_Krein}).
We recall that an operator $S$ on a Hilbert space belongs
to the Macaev ideal  $\mathfrak {S}_\om$ if it is
compact and its singular numbers $s_k$ satisfy the relation
$\sum_{k\ge 1} k^{-1} s_k <\infty$.

\begin{theorem_}[Keldy\v s, 1951]
Suppose $\A$ is a selfadjoint Hilbert space operator that belongs to a Schatten ideal
$\mathfrak {S}_p$, $0<p<\infty$ and satisfies
$\ker \A=0$. Let $\LL =\A(I+S)$, where $S$ is compact and $\ker(I+S)=0$.
Then the operators $\LL $ and $\LL^*$ are complete.
\end{theorem_}

\begin{theorem_}
[Macaev, 1961]
If $\LL =\A(I+S)$, where $\A$, $S$ are compact operators on a Hilbert space,
$\A$ is selfadjoint, $S\in \mathfrak {S}_\om$
and $\ker \A = \ker(I+S)=0$, then $\LL $ and $\LL^*$ are complete.
\end{theorem_}

A proof of this theorem and its generalizations to operator
pencils can be found in~\cite{MatsMog71}.
As Macaev proved in~\cite[Theorem 3]{Mats64},
in the above result $\mathfrak {S}_\om$ can be replaced by
a wider class, depending on $\A$.
Perturbations of a selfadjoint compact operator $\A$ that have
the form $\A(I+S)$ or $(I+S)\A$, where $S$ is compact,  are called {\it weak perturbations}.
In~\cite{MatsMogul76}, Macaev and Mogul'skii give an explicit condition on the spectrum of $\A$, equivalent to the property
that all weak perturbations of $\A$ are complete; see also~\cite{MatsMogul72}.

\subsection{Statement of the problems}
In this article, we consider
rank $n$ perturbations of compact selfadjoint operators, which are
neither weak nor dissipative. Our main results are concerned with rank one
(nondissipative) perturbations of compact selfadjoint operators.
We study the following spectral properties of the operators from this class:
\smallskip

\begin{itemize}
\item Completeness of eigenvectors or root vectors;
\smallskip

\item Relations between completeness of the operator and its adjoint;
\smallskip

\item Spectral synthesis problem for rank one perturbations.
\end{itemize}
\smallskip

We obtain new criteria for completeness of rank one perturbations,
for joint completeness of the operator and its adjoint and for the
spectral synthesis. At the same time we construct a series
of subtle (counter)examples which show that the spectral structure
of rank one perturbations becomes unexpectedly rich and complicated
as soon as we leave the classes covered by classical theories
(i.e., weak perturbations or dissipative operators).
In particular, we will show the sharpness of the
conditions of Macaev's theorem even for this class.

We also introduce
what we call {\it finite rank singular perturbations} of unbounded selfadjoint operators with discrete
spectrum. These are essentially unbounded inverses to finite rank perturbations
of compact selfadjoint operators with
trivial kernels.
Ordinary differential operators with
nonselfadjoint boundary conditions belong to this class.
We obtain parallel completeness results for the case of rank one singular perturbations.
The spectral theory of these operators is the second subject of this paper.

%
%

\subsection{Our methods: Functional model}
Our main tool for the study of rank one perturbations is a functional model
which translates each singular rank one perturbation of a singular selfadjoint
operator to some "model"$ $ operator on a de Branges space of entire functions
or, more generally, on some model (backward shift invariant)
subspace of the Hardy space in the upper half-plane.
The model operator is essentially a rank one perturbation
of the operator of multiplication by the independent variable
in the corresponding space of analytic functions.
This model relates completeness problems for rank one perturbations
to completeness of systems of reproducing kernels in de Branges spaces and
of their biorthogonal systems. New approaches to the study
of such problems were recently introduced
by Makarov and Poltoratski in
\cite{mak-polt}, \cite{mak-polt1} and by Baranov, Belov and Borichev
in \cite{bar-belov}, \cite{bbb}, \cite{bbby}, \cite{bbb1};
the results of the present paper both rely on and complement the results
of these papers.


\subsection{Main results}
If $f,g$ are vectors in a Hilbert space $H$, we denote by
$g^*$ the linear functional $g^*x\defin \langle x,g \rangle$ on $H$
and by $fg^*$ the rank one linear operator on $H$, given by
$(f g^*)x\defin \langle x,g \rangle f$.

Let $\A$ be a compact selfadjoint operator in a Hilbert space
(throughout the paper all Hilbert spaces are supposed to be separable)
with simple spectrum (i.e., $\A$ is cyclic) and trivial kernel. Thus,
we may assume that $\A$ is the operator of multiplication by
the independent variable $(\A f)(x) = xf(x)$ in some space
$\Elltwo(\mu)$, where $\mu = \sum_n \mu_n \delta_{s_n}$,
$s_n \ne 0$ and $s_n\to 0$, $|n|\to \infty$. In what follows we identify the
elements of $\Elltwo(\mu)$ and sequences: $a=(a_n)$, where $a_n = a(s_n)$.

For $a=(a_n), b=(b_n) \in \Elltwo(\mu)$, consider the corresponding
rank one perturbation of $\A$,
\beqn
\label{nom0}
\LL =\A +ab^*, \qquad \LL f = \A f + \langle f, b\rangle a.
\neqn

Now we formulate the main results of the paper.


\subsubsection{Results on completeness} We start with a (slight) extension of
the Macaev theorem for the case of rank one perturbations
(for a similar result for finite rank perturbations see
Proposition~\ref{gen_weak_Mats}).

\begin{theorem}
\label{genweak0}
Under the above conditions on $\A$ and $\LL$, assume that
\beqn
\label{smooth33}
\sum_n  \frac{|a_n b_n|\mu_n}{|s_n|} <\infty
\qquad\text{and} \qquad
\sum_n \frac{a_n\bar b_n \mu_n}{s_n} \ne -1.
\neqn
Then both $\LL$ and $\LL^*$ are complete.
\end{theorem}

First condition in \eqref{smooth33} means that
$a b\in x \Ellone(\mu)$ and so
the functions $a$ and $b$ have additional "smoothness"$ $ near 0.
If $ a\in x \Elltwo(\mu)$ or $b \in x \Elltwo(\mu)$, we are in the case
of weak perturbations, while, for $a\in x \Elltwo(\mu)$,
the second condition in \eqref{smooth33}
coincides with the nondegeneracy condition $\ker(I+(\A^{-1}a) b^*) =0$.
A more general situation when the first condition in \eqref{smooth33})
is satisfied is when there is $\alpha\in [0,1]$ such that
\beqn
\label{smooth}
|x|^{- \alpha} a\in  \Elltwo(\mu), \qquad |x|^{-1+\alpha} b\in  \Elltwo(\mu),
\neqn
which means that
$$
\sum_n |a_n|^2 |s_n|^{-2\alpha} \mu_n<\infty,
\qquad
\sum_n |b_n|^2 |s_n|^{2\alpha-2} \mu_n<\infty.
$$

Our second result shows that the property of being a generalized
weak perturbation may be replaced by a positivity condition.

\begin{theorem}
\label{positive0}
Suppose that $a_n\overline b_n\ge 0$ for all but possibly
a finite number of values of $n$, $b_n\ne 0$ for any $n$,
and $\sum_n |s_n|^{-1}|a_nb_n|\mu_n = \infty$.
Then $\LL $ and $\LL^*$ are complete.
\end{theorem}

In Theorem~\ref{3to1}, conditions sufficient for the joint
completeness of $\LL $ and $\LL^*$ are given in terms of the
generating function $\phi$, the parameter of the functional model.


\subsubsection{Relations between completeness of $\LL$ and $\LL^*$}
Note that if a
compact bounded operator $T$ is complete, a trivial obstacle
for completeness of $T^*$ is that $T$ may have a nontrivial kernel, while
$\ker T^* = 0$. The first (highly nontrivial) examples
of the situation where $T$ is complete and $\ker T=0$, while $T^*$
is not complete, were constructed by Hamburger~\cite{hamb}.
In~\cite{DeckFoPea} Deckard, Foia\c{s} and Pearcy gave a simpler
construction. However, in these examples one cannot conclude that
the corresponding operator is a small (in some sense)
perturbation of a selfadjoint operator.
Surprisingly, one can find such examples
among rank one perturbations of
a compact selfadjoint operator with an arbitrary spectrum.

\begin{theorem}
\label{adjoint0}
For any compact selfadjoint operator
$\A$ with simple point spectrum and trivial kernel there
exists a rank one perturbation $\LL$ of $\A$
with real spectrum such that
$\ker \LL = \ker \LL^*= 0$ and $\LL$ is complete, but
$\LL^*$ is not complete and, moreover,
the orthogonal complement to the span of root vectors of $\LL^*$
is infinite-dimensional.
\end{theorem}

It was pointed out to the authors by Mark Malamud that a concrete example
of a rank one perturbation $L$ of a compact normal operator
such that $\ker \LL = \ker \LL^*= 0$, $\LL$ is complete, but
$\LL^*$ is not, can be extracted from the results
by Lunyov and Malamud \cite[Section 4]{LunMal15} combined with 
\cite{MalamOrid2012}. The operator in this example is realized as
the inverse to a two-dimensional
first order differential operator with specially chosen boundary conditions.
A version of this example is presented in Appendix 1
with kind permission of M. Malamud.
Note however that for this construction is is essential that
the unperturbed normal operator is {\it nonselfadjoint}.
It would be interesting to find realizations of the examples
from Theorem \ref{adjoint0} as differential operators.

The following theorem shows that for sufficiently ''large'' or ''regular''
rank one perturbations the completeness of $\LL$
implies the  completeness of its adjoint.

\begin{theorem}
\label{frombb0}
Let the perturbation $\LL =\A +ab^*$ of $\A$
be complete and let $a, b \notin x \Elltwo(\mu)$.
Then its adjoint $\LL^*$ is also complete
if any of the following conditions is fulfilled:
\smallskip

$(i)$ $|a_n|^2 \mu_n  \ge C |s_n|^N>0$ for some $N>0$\textup;
\smallskip

$(ii)$ $|b_n a_n^{-1}| \le C |s_n|^{-N}$ for some $N>0$.
\end{theorem}


\subsubsection{Spectral synthesis for rank one perturbations}
Recall that a bounded linear operator $T$ in a Hilbert space is said to
{\it admit spectral synthesis} if
any $T$-invariant subspace $\M$ coincides with the closed
linear span of the eigenvectors and
root vectors which belong to $\M$ (we denote this "spectral"$ $ part of
$\M$ by $\mathcal{E}(\M)$).
Equivalently, this means that the restriction of $T$ to
any its invariant subspace $\M$ is complete.
This notion goes back to Wermer \cite{wer} who showed, in particular, that
the spectral synthesis holds for all normal compact operators.
The first example of a compact operator which does not admit spectral synthesis
also goes back (implicitly) to Hamburger's paper \cite{hamb}. Further examples and
generalizations were obtained by Nikolski~\cite{nk-volterra}
and Markus~\cite{markus70}, who showed also that any weak perturbation $\LL$
of a compact selfadjoint operator with $\ker \LL = 0$ admits the spectral
synthesis.
Recently, Lunyov and Malamud \cite{LunMal14} showed that for a class of
dissipative realizations of Dirac-type differential operators, completeness property
is equivalent to the spectral synthesis property.

It turns out however that the spectral synthesis may fail even for rank
one perturbations of compact selfadjoint operators.
Making use of the recent results from \cite{bbb1} we are
able to prove the following theorem.

\begin{theorem}
\label{synthesis}
For any compact selfadjoint operator $\A$ with
simple point spectrum and trivial kernel there
exists a bounded rank one perturbation $\LL$ of $\A$
with real spectrum such that $\ker \LL=\ker \LL^*=0$ and
both $\LL$ and $\LL^*$ have complete sets of
eigenvectors, but $\LL$ does not admit spectral synthesis.

Moreover, $\LL$ can be chosen in such way that  the spectral
synthesis will fail with an infinite defect, that is,
${\rm dim}\, \big(\M \ominus\mathcal{E}(\M)\big) = \infty$
for some $\LL$-invariant subspace $\M$.
\end{theorem}

The proof of Theorem~\ref{synthesis}
combines a spectral synthesis criterion due to Markus
with recent results on completeness of systems of reproducing kernels
\cite{bbb1}. Let $T$ be a compact operator in a Hilbert space $H$
with a trivial kernel and simple spectrum
such that the sequence of its eigenvectors
$\{x_n\}_{n\in \I}$ is complete in $H$ and minimal.
Denote by $\{x_n'\}_{n\in \I}$ the system biorthogonal to $\{x_n\}$
(i.e., $\langle x_k, x'_n\rangle = \delta_{kn}$).
By a theorem of Markus~\cite[Theorem~4.1]{markus70},
$T$ admits the spectral synthesis if and only if
the system $\{x_n\}_{n\in \I}$ is {\it strongly} or {\it hereditarily complete},
which means that for any partition $\I = \I_1 \cup \I_2$,
$\I_1 \cap \I_2 =\emptyset$,
the mixed system $\{x_n\}_{n\in \I_1} \cup \{x'_n\}_{n\in \I_2}$
is complete in $H$. Such systems are also known as
{\it strong Markushevich bases}. For further properties and examples
of hereditarily and nonhereditarily complete systems see~\cite{dns}.

At the same time we show that under certain
restrictions on the spectrum and on the perturbation
the spectral synthesis holds up to a finite-dimensional defect.

\begin{theorem}
\label{sp_sint}
Let $\A$ be a compact selfadjoint operator with
simple spectrum $\{s_n\}_{n\in I}$, $s_n \ne 0$. Assume that\textup:
\smallskip

$(a)$ $\{s_n\}_{n\in I}$ is ordered so that $s_n>0$ and $s_n$ decrease for
$n\ge 0$, and $s_n<0$ and increase for $n<0$ and
$$
|s_{n+1}-s_n| \ge C_1 |s_n|^{N_1}
$$
for some $C_1,N_1>0$\textup;
\smallskip

$(b)$ $\LL = \A + ab^*$ is a bounded rank one perturbation of $\A$
such that $a, b  \notin x \Elltwo(\mu)$ and $a$ satisfies condition $(i)$
from Theorem \ref{frombb0}, that is, $|a_n|^2 \mu_n \ge C |s_n|^{N}$
for some $C,N>0$\textup;
\smallskip

$(c)$ Operator $\LL$ is complete and all its eigenvalues
are simple and non-zero.
\smallskip
\\
Then for any $\LL$-invariant subspace $\M$
we have
$$
{\rm dim}\, \big(\M \ominus\mathcal{E}(\M)\big) < \infty,
$$
where the upper bound for the dimension depends only on $N$ and $N_1$.
\end{theorem}

In conditions of Theorem  \ref{sp_sint}
the defect of the spectral synthesis may be nonzero even for very regular
sequences $\{s_n\}$ and $\{a_n\}$. Let $s_n = a_n = \big(n+\frac{1}{2}\big)^{-1}$
and $\mu_n \equiv 1$, $n\in\mathbb{Z}$. Then the spectral synthesis
for $\A$ translates to the spectral synthesis (or hereditary completeness)
problem for systems of reproducing kernels in the classical Paley--Wiener
space $PW_\pi$ or, equivalently, for exponential systems in $\Elltwo(-\pi, \pi)$.
This problem was solved in \cite{bbb}, where exponential
systems were constructed for which the spectral synthesis fails. However, the defect
of the synthesis (the dimension of the orthogonal complement
$\M \ominus\mathcal{E}(\M)$) is always at most 1.

Using the recent results of \cite{bbb1}, it is possible to describe
completely the data $\mu$ and $a$ such that any complete rank one perturbation
$\A+ab^*$ (i.e., for any $b$) admits the spectral synthesis (see Theorem \ref{newt}).

We also remark that a sufficient condition for the spectral
synthesis in terms of the ''characteristic function'' of a rank
one perturbation was found by Gubreev and Tarasenko \cite[Theorem
2.5]{Gubr-Tar2010} (see Subsection \ref{pssr} for details).


\subsubsection{Sharpness of Macaev theorem}
Finally, using our techniques it is easy to show that even for
rank one perturbations, when we relax slightly the generalized
weakness property \eqref{smooth33}, the resulting perturbation may
become a {\it Volterra operator} with trivial
kernel (we recall that an operator is called Volterra if it is
compact and its spectrum equals $\{0\}$).

\begin{theorem}
\label{sharp}
There exists a compact selfadjoint operator $\A$
\textup(multiplication by $x$ in $\Elltwo(\mu)$\textup)
which is in $\mathfrak {S}_p$ for any $p>1/2$
with the following property\textup:
for any $\alpha_1, \alpha_2 \ge 0$ with
$\alpha_1+ \alpha_2 <1 $
there exist $a \in |x|^{\alpha_1} \Elltwo(\mu)$ and $b \in
|x|^{\alpha_2} \Elltwo(\mu)$
such that the perturbed operator
$\LL= \A_0 + a b^* $ is a
Volterra operator satisfying $\ker \LL=\ker\LL^*=0$.
\end{theorem}

It is shown in~\cite{MatsMogul72}, \cite{MatsMogul76}
that if $\A $ is a positive
operator, whose spectrum is sufficiently dense, then there is a weak perturbation
of $\A $, which is a Volterra operator (but in this case, of course, $\A$
does not belong to any Schatten class).

In our paper \cite{by}, we study in more detail the following question:
\textit{For which compact selfadjoint operator $\A$ does there exist
%
%
its rank one perturbation $\LL $, which is a Volterra operator}?


\subsection{Organization of the paper}

In {\it Section 2}, we develop the basic theory of rank $n$ singular perturbations
of selfadjoint operators. In particular, we relate
them with rank $n$ (bounded) perturbations of selfadjoint compact operators
(Proposition \ref{L inverse}). In {\it Section 3} (Propositions \ref{gen_weak_Mats} and
\ref{prp-Mats-type-sing}), we give some completeness results for
finite rank perturbations (both usual and singular) by applying
Macaev's theorem and a simple additional argument.

{\it Sections 4 and 5} deal with the functional model for singular rank one perturbations.
The functional model is presented in Theorem 4.4. Its proof is given in {\it Section 5}
as well as a detailed analysis of the parameters of the model.
To make the paper reasonably self-contained, the necessary background on model spaces
and their Clark measures is given in {\it Section 5}.
{\it Subsection \ref{herm--biehl}} contains some basic information
on de Branges spaces of entire functions.
In {\it Subsection \ref{reform}} the counterparts of main results
for the case of singular rank one perturbations are stated
(Theorems \ref{positive} -- \ref{frombb}).
Let us note here that the strategy of the proofs of the main results is to reduce
them to equivalent statements about singular rank one perturbations,
which are proved using the functional model and entire function theory.

In {\it Section 6} several completeness results are proved for singular rank
one perturbations (the main of them is Theorem~\ref{3to1}) and the proofs of theorems
\ref{positive} and \ref{frombb} are given. {\it Section~7} is devoted to the proof
of the most technically involved result of the paper --
Theorem \ref{noncompleteness2} (unbounded version of Theorem~\ref{adjoint0}).
Finally, in {\it Section 8} the proofs of our main results for rank one
perturbations of compact operators (Theorems \ref{genweak0} -- \ref{sharp}) are given.

Two appendices contain a brief survey of completeness results for
linear operators ({\it Appendix 1}) and the proofs of several
technical propositions from Section~\ref{model} ({\it Appendix~2}).
\bigskip
\\
\textbf{Acknowledgements.} The authors are grateful to
Yurii Belov, Vladimir Macaev, Arkadi Minkin, Nikolai Nikolski, and
Roman Romanov for many helpful comments and illuminating discussions.
\bigskip



\section{Singular perturbations of unbounded operators}
\label{model}

\subsection{Finite rank singular perturbations}
We begin with the following definition, motivated by
~\cite{Aziz_Behr_etl}, where a more general case of closed linear
relations was treated.
We refer to \cite{DerkMal91}, \cite{Albev-Kur99}, \cite{Posilicano},
\cite{Br_Marl_Nab_W08} and to books
\cite{GorbGorb91}, \cite{Lyants-Stor-book},
\cite{Kuzhel} for alternative treatments of
singular perturbations in more general settings.

\begin{definition}
Let $\A$, $\LL $ be unbounded closed linear operators on a Banach
space $H$. We say that $\LL $ is a \emph{a finite rank singular
perturbation} of $\A$ if their graphs $G(\A), G(\LL )$ in $H\oplus H$
differ in a finite dimensional space. If, moreover,
\beqn
\label{eq-dims}
\dim \big( G(\A) / (G(\A)\cap G(\LL ) ) =
\dim \big( G(\LL ) / (G(\A)\cap G(\LL ) )  \big) =n <\infty,
\neqn
then we will say that \emph{$\LL $ is a balanced rank $n$ singular perturbation} of $\A$.
\end{definition}

If $\A$ and $\LL $ are differential operators, defined by the same
regular differential expression of order $n$ and different
boundary relations ($n$ independent relations in both cases),
then $\LL $ is a balanced singular perturbation of $\A$ of rank less
or equal to $n$.

The main questions we study in this paper are the following:
\medskip

\emph{Given a singular rank one perturbation $\LL $
of a cyclic selfadjoint operator $\A$, when are operators $\LL $
and $\LL^*$ complete?  Under which assumptions
completeness of $\LL $ implies completeness of $\LL^*$?}
\medskip

\noindent
In Section \ref{d-subord+Mats},
we will relate this question with the completeness of
bounded rank $n$ perturbations of a compact selfadjoint operator and therefore with
theorems by Keldy\v s and Macaev. In fact, Proposition~\ref{L inverse} below implies that, roughly speaking,
rank $n$ singular perturbations of a selfadjoint operator $\A$ with discrete spectrum
are inverses of rank $n$ (bounded) perturbations of $\A^{-1}$.

Suppose $\A$ is a closed, densely defined linear operator on a Hilbert
space $H$. We denote by $\si(\A)$ the spectrum of $\A$
and by $\rho(\A)=\BC\sm \si(\A)$ its resolvent set. In what follows
we will always assume that $0\in \rho(\A)$. An obvious modification
of our construction below works if $\rho(\A)\ne\emptyset$.

We define the Hilbert space $\A H$ as the set of formal
expressions $\A  x$, where $x$ ranges over \textit{the whole space}
$H$. Put $\|\A x\|_{\A H}= \|x\|_H$ for all $x\in H$. The formula
$x=\A (\A^{-1}x)$ allows one to interpret $H$ as a linear submanifold
of $\A H$. We consider the scale of spaces
$$
\cDA \subseteq H \subseteq \A H.
$$
Under our assumptions, $\A^*$ is well defined, and there are
natural identifications
$\cD(\A^*)=(\A H)^*$, $\cDA =(\A^*H)^*$.


\subsection{Singular balanced perturbations and their $n$-data}
The rank $n$ singular balanced perturbations of $\A $
can be described in terms of what we will call $n$-data for $\A $.
Let $n\in \NN$. By \emph{$n$-data} we mean a triple
$
(a,b,\deab),
$
where
\beqn
\label{abkap}
a:\BC^n\to \A H,\quad b:\BC^n\to \A^*H, \qquad \deab :\BC^n\to \BC^n
\neqn
are linear and bounded, $\rank a=\rank b=n$, and for any
$c\in \BC^n$,
$$
\quad \qquad\qquad
\A^{-1}ac\in\cDA, \ \ \deab  c =
b^*(\A^{-1}ac) \enspace \Longrightarrow \enspace c=0.
\quad \qquad\qquad  ({\rm A }_n)
$$
Notice that $b^*$ is an operator $b^*: \cDA\to\BC^n$.

For any $n$-data $(a,b,\deab )$, we define a linear operator $\LL =\LL (\A ,a,b,\deab )$ in the following way:
\beqn
\begin{aligned}
\label{2app}
\cD(\LL )&\defin \big\{
y=y_0+\A^{-1}ac: \\
& \qquad \qquad c\in \BC^n,\, y_0\in \cDA,\, \deab  c+b^*y_0=0
\big\};                     \\
\LL (\A ,a,b,\deab )\, y& \defin \A y_0, \quad y\in \cDL.
\end{aligned}
\neqn
Condition  $(\admn)$  is
equivalent to the uniqueness of
the decomposition $y=y_0+\A^{-1}ac$ for $y\in \cD\big(\LL (\A ,a,b,\deab ) \big)$ and hence
to the correctness of the definition of $\LL $.

The introduction of this kind of perturbation is justified
by the fact that
any balanced rank $n$ singular perturbation of $\A $ has the form
$\LL (\A ,a,b,\deab)$ for some $n$-data, which, in a sense,
is determined uniquely.

\begin{proposition}
\label{Pr1add}
Suppose that $0\notin\si(\A)$. Then one has:
\begin{enumerate}

\item For each $n$-data $(a,b,\deab)$, the operator
$\LL (\A ,a,b,\deab )$, defined
in \eqref{2app},  is a balanced rank $n$ singular perturbation of $\A $\textup;

\item Any  balanced rank $n$ singular perturbation of $\A $ has the
form $\LL (\A ,a,b,\deab )$, for some $n$-data $(a,b,\deab )$.
\end{enumerate}
\end{proposition}

The proof of Proposition \ref{Pr1add} will be given in Appendix 2.

There are other ways of introducing singular perturbations,
that permit one to deal with infinite rank perturbations, see
~\cite{Lyants-Stor-book},~\cite{Ryzhov} and references therein.


\subsection{Elementary properties of singular perturbations}
In the propositions below, $\A $ is assumed to
be an arbitrary closed unbounded linear operator,
and $(a,b,\deab)$ are $n$-data defined in (\ref{abkap}).
We also assume that $0\notin\si(\A)$.

Suppose $n$-data $(a,b,\deab)$ are fixed.
We recall that the adjoint of a closed operator on a
Hilbert space exists if and only
if this operator is densely defined. To be able to consider $\LL^*$,
we need to introduce a dual condition to  $(\admn)$:
for any $d\in \BC^n$,
$$
\qquad  \quad
(\A^*)^{-1} bd \in \cD(\A^*), \ \  \deab^*d=a^*\big((\A^*)^{-1} bd\big)
\Longrightarrow d=0.
\qquad  \quad
(\admnst)
$$

\begin{proposition}
\label{Pr-adjoints-n}
\begin{enumerate}
\item
Operator $\LL =\LL (\A , a,b,\deab)$ is densely defined if and only if
the data $(a,b,\deab)$ satisfy  $(\admnst)$ .

\item If $(\admnst)$  holds, then  $\LL (\A ,a,b,\deab)^*=\LL (\A^*,b,a, \deab^*)$.
\end{enumerate}
\end{proposition}

For the proof of Proposition \ref{Pr-adjoints-n} see Appendix 2.

The next proposition is a kind of a uniqueness assertion.

\begin{proposition}
\label{uniqnss}
\begin{enumerate}
\item
For any invertible operators
$\tau_1, \tau_2$ on $\BC^n$,
$\LL (\A , a,b, \deab)=\LL (\A , a\tau_1^{-1}, b\tau_2, \tau_2^*\deab \tau_1^{-1})$\textup;

\item
Conversely, if $a,b, a_1, b_1:\BC^n\to \A H$ are rank $n$ operators and
$\LL (\A , a,b, \deab)=\LL (\A , a_1,b_1, \deab_1)$, then
there are invertible operators
$\tau_1, \tau_2$ on $\BC^n$
such that
$a_1=a\tau_1^{-1}$, $b_1=b\tau_2$, $\deab_1=\tau_2^*\deab \tau_1^{-1}$.
\end{enumerate}
\end{proposition}

The following proposition relates singular finite rank perturbations with
inverses of usual finite rank perturbations of bounded selfadjoint operators.

\begin{proposition}
\label{L inverse}
\begin{enumerate}

\item
Given operators $a,b$ of rank $n$
and an invertible $\deab$, the operator
$\A^{-1}-(\A^{-1}a)\deab^{-1}(b^*\A^{-1})$
has a trivial kernel $($and therefore has an inverse in the algebraic
sense, defined on its image$)$
if and only if the triple
$(a,b,\deab)$ satisfies $(\admn)$.

\item
If the triple $(a,b,\deab)$ satisfies $(\admn)$
and $\deab$ is invertible, then
$$
\LL (\A , a,b, \deab)=\big(\A^{-1}-(\A^{-1}a)\deab^{-1}(b^*\A^{-1})\big)^{-1}.
$$
Therefore, in this case $\LL (\A , a,b, \deab)$ is an $($unbounded$)$
algebraical inverse  to a bounded rank $n$ perturbation of the
bounded operator $\A^{-1}$.
\end{enumerate}
\end{proposition}

\begin{proposition}
\label{prop_shift}
For any $\la\in \rho(\A )$,
$$
\LL (\A ,a,b,\deab)-\la I=\LL \big(\A -\la I, a,b, \be(\la)\big),
$$
where $ \be (\la) =\deab+\la b^*(\A -\la)^{-1}\A^{-1} a$.
\end{proposition}

The proofs of Propositions \ref{uniqnss} -- \ref{prop_shift} are straightforward, and we
omit them.
\bigskip



\section{$d$-subordination of operators and Macaev type results on completeness
of rank $n$ perturbations of selfadjoint operators}
\label{d-subord+Mats}

In this section we prove a Macaev-type completeness theorem for
finite rank perturbations of selfadjoint operators (both usual and
singular ones). First we discuss rank $n$ perturbations of a
compact selfadjoint operator.

\begin{definition}
Let $\LL_1\in B(H_1)$, $\LL_2\in B(H_2)$ be two bounded Hilbert
space operators. We say that $\LL_2$ \textit{is $d$-subordinate}
to $\LL _1$ (and write $\LL _1\precd \LL _2$) if there exists a
bounded linear operator $Y:H_1\to H_2$, which intertwines $\LL _1$
with $\LL _2$ and has a dense range:
$$
Y\LL _1=\LL _2Y; \qquad \clos\Ran Y=H_2.
$$
\end{definition}

In this situation, if $k$ is an eigenvector of $\LL _1$, then $Yk$
is an eigenvector of $\LL _2$, and a similar assertion holds for
root vectors. It follows that
\beqn
\label{d-subord}
\LL _1\precd
\LL _2,\; \text{$\LL _1$ is complete} \enspace \implies \enspace
\text{$\LL _2$ is complete.}
\neqn

Let $\A $ be a compact, not necessarily cyclic
selfadjoint operator, acting on a Hilbert space. Without loss of
generality, we put $\A$ to be the multiplication operator $M_x$ on
a direct integral of Hilbert spaces $\cH\defin\int^\oplus H(x)\,
d\mu (x)$, where $\mu$ is a (discrete)
positive measure on $\BR$. We will
assume that $\ker \A = 0$, then $\mu(\{0\})=0$.

Let
$$
\LL =\A +ab^*, \quad \text{where $a,b:\BC^n\to \cH$}
$$
be a rank $n$ perturbation of $\A$.
Let $\{e_j\}_{j=1}^n$ be the
standard basis in $\BC^n$.
Put $a_j(x)=(a e_j)(x)\in H(x)$,
$|a_j|(x)=\|a_j(x)\|_{H(x)}$ and $|a|(x)=\big(\sum_j(|a_j(x)|^2 \big)^{1/2}$.
Define $b_j(x)$, $|b_j|(x)$ and $|b|(x)$ similarly.
Then $|a|$, $|b|$ are functions in $\Elltwo(\mu)$.

We call $\LL$ \textit{a rank  $n$ generalized weak perturbation of
$\A$} if $\ker \LL =0$ and $|a|\cdot|b|\in x\,\Ellone(\mu)$. For these
perturbations, we define a matrix \beqn \label{def-omega} \omega=
\omega(\A ,a,b)= \big( \omega_{jk} \big)_{j,k=1}^n, \quad
\text{where} \enspace \omega_{jk} = \int_{\BR} x^{-1} \langle
a_j(x), b_k(x)\rangle_{H(x)} \, d\mu(x). \neqn The following
statement is an easy consequence of Macaev's theorem and of the
observation \eqref{d-subord}.

\begin{proposition}
\label{gen_weak_Mats} Suppose that
$\A =\A^*$ is compact, $\ker \A =0$ and $\LL $ is a rank $n$
generalized weak perturbation of $\A$. If the $n\times n$ matrix
$I_n+\om$ is invertible, then $\LL $ and $\LL^*$ are complete.
\end{proposition}

\begin{proof}[Proof]
Consider a scalar bounded function
$$
\psi(x)=
\begin{cases}
1, & \quad  \text{if $(|a|\cdot|b|)(x)=0$ or $|x|\cdot |a|(x)/|b|(x)>1$},  \\
\root\of{\vphantom{\big|}|x|\cdot |a|(x)/|b|(x)}, & \quad
\text{if $0<|x|\cdot |a|(x)/|b|(x)   \le   1$},
\end{cases}
$$
and define the functions
$\wt a_j=a_j/\psi$ and $\wt b_j= \psi\cdot b_j/x$. It is easy to check the following facts.

\smallskip

(1) $|\wt a_j|, |\wt b_j| \in \Elltwo(\mu)$,
$j=1,\dots, n$, so that the $n$-tuples of functions
$\{\wt a_j\}_1^n$, $\{\wt b_j\}_1^n$ define
operators $\wt a, \wt b:\BC^n\to\cH$;

(2) The bounded operator $Y=M_\psi$ on $\cH$ commutes with $\A
=M_x$ and has a dense range;

(3) $Y\wt \LL = \LL Y$, where $\wt \LL =(I+\wt a \wt b^*)\A$.

Notice that $\om_{jk}=\langle \wt a_j, \wt b_k \rangle_{\cH}$.
Since $I_n+\om$ is invertible, it follows that
$\ker (I+\wt a \wt b^*)=0$. By Macaev's theorem,
$\wt \LL$ is complete. Since $\wt \LL \precd \LL$, $\LL$ also is complete.
\end{proof}
\medskip

Note that Theorem~\ref{genweak0} is a special case of Proposition
\ref{gen_weak_Mats} with $n=1$.
\medskip

Now we prove a counterpart of Proposition \ref{gen_weak_Mats} for
singular perturbations. We call $\LL $ a \textit{rank $n$
generalized weak singular perturbation of $\A $} if $0\notin
\sigma(\A )$ and $x^{-1} |a|\cdot |b| \in \Ellone(\mu)$. For these
perturbations, we will also use \eqref{def-omega} to define an
$n\times n$ matrix $\om=\om(\A ,a,b)$. We remark that if
$\deab-\om$ is invertible, then conditions  $(\admn)$  and
$(\admnst)$  hold, which implies that $\LL $ and $\LL^*$ are
correctly defined.

\begin{proposition}
\label{prp-Mats-type-sing} Suppose that $\LL =\LL (\A ,a,b,
\deab)$ is a rank $n$ generalized weak singular perturbation of a
selfadjoint operator $\A $ with compact resolvent. If the matrix
$\deab-\om$ is invertible, then $\LL^*$ is correctly defined, and
$\LL $ and $\LL^*$ are complete.
\end{proposition}

\begin{proof} [Proof] 
Let us apply first Proposition~\ref{prop_shift}. The Lebesgue
dominated convergence theorem implies that $\be(iy)\to \deab-\om$
as $y\to  \infty$. Hence for all but a discrete set of $\la$'s in
$\BC^+$, $\be(\la)$ is invertible and therefore the same holds
true for all $\la$'s on the real line, except a discrete set in
$\BR\sm \si(\A )$ (recall that $\si(\A )$ is discrete).

By substituting $\A $ with $\A -\la I$ for some real $\la$, if
necessary, we can assume that $\be(0)=\deab$ is invertible. By
Proposition~\ref{L inverse}, $\LL ^{-1}=\A^{-1}+\wh a \wh b^*$,
where $\wh a = - \A^{-1}a\deab^{-1}$ and $\wh b=\A^{-1}b$. Put
$\tilde\om=\om(\A^{-1}, \wh a, \wh b)$. Then it is easy to check
that $\wh\om =-  \om \deab^{-1}$. Hence the matrix $\wh\om
+I_n=(\deab- \om)\deab^{-1}$ is invertible. By
Proposition~\ref{gen_weak_Mats}, $\LL ^{-1}$ and $(\LL^*)^{-1}$
are complete. Hence $\LL $ and $\LL^*$ are complete.
\end{proof}
\bigskip



\section{Basic Functional Model for singular rank one perturbations}

\subsection{Abstract singular rank one perturbation}
One of the main objects of our study are singular rank one perturbations of an
unbounded selfadjoint operator $\A $, given by
\beqn
\label{defA}
H\defin \Elltwo(\mu), \qquad (\A f)(x)=x f(x), \enspace x\in \Elltwo(\mu).
\neqn
By the spectral theorem, any selfadjoint operator  of spectral
multiplicity $1$ can be represented in this form. We assume that $\mu$
is a \textit{singular measure} and that ${\rm \supp}\, \mu \ne \RR$.
Without loss of generality we may assume then that
$0\notin {\rm \supp}\, \mu$.

Notice that for $\A $ given in this form, $\A H=x \Elltwo(\mu)$.
Hence the $1$-data for $\A$ is just a triple $(a,b,\deab)$,
where
\beqn
\label{cond-a-b}
\frac{a}{x}, \  \frac{b}{x}  \in \Elltwo(\mu); \quad \deab\in \BC
\neqn
\big(however, possibly, $a,b \notin \Elltwo(\mu)$\big)
and the condition
$$
\begin{aligned}
\deab&\ne
\int_\BR x^{-1}a(x)\overline{b(x)} \, d\mu(x) \\
&\qquad\qquad\qquad
\text{ in the case when $a(x)\in \Elltwo(\mu)$}
\end{aligned}
\eqno{(\admone)}
$$
is fulfilled.
By \eqref{2app}, the corresponding singular perturbation $\LL =\LL
(\A , a,b,\deab)$ of $\A $ is defined as follows: \beqn
\begin{aligned}
\label{2}
\cDL&\defin
\big\{
y=y_0+c\cdot \A^{-1}a: \\
& \qquad \qquad
c\in \BC,\,
y_0\in \cDA,\,
\deab\, c+\langle y_0, b\rangle=0
\big\};                     \\
\LL  y& \defin \A y_0, \quad y\in \cDL.
\end{aligned}
\neqn

\begin{definitions}
{\rm We call $\LL $ a {\it real type perturbation} of $\A $ if the function
$\bar a b$ and the number $\deab$ are real.
We call  $\LL $
{\it a strong real type perturbation} if, in addition to the above two requirements,
the spectrum $\sigma (\LL )$ is contained in  $\RR$. }
\end{definitions}

The spectrum of any real type perturbation of $\A $ is symmetric
with respect to the real line.
Recently there have been much interest in the study
of nonselfadjoint operators  with real spectrum in
connection with non-Hermitean Hamiltonians that have a space-time
reflection symmetry, see~\cite{Alb_Fei_Kur},~\cite{Bender} and
references therein.

We recall that for any closed operator $\LL$ on $H$,
the adjoint operator $\LL^*$ is well-defined if and only if $\cDL$ is dense in $H$.
Let $\LL =\LL (\A ,a,b, \deab)$ be any singular rank one perturbation of $\A $.
It follows from Proposition~\ref{Pr-adjoints-n}
that $\LL^*$ exists whenever
$$
\begin{aligned}
\deab&\ne
\int_\BR x^{-1}a(x)\overline{b(x)} \, d\mu(x) \\
&\qquad\qquad\qquad \text{ in the case when $b\in \Elltwo(\mu)$}.
\end{aligned}
\eqno{(\admonest)}
$$
Moreover, if $(\admonest)$ holds, then $\LL^*= \LL (\A , b, a, \deab^*)$.


\subsection{Counterparts of main results for singular rank one perturbations}
\label{reform}
By Proposition \ref{L inverse}, completeness problems
for finite rank one perturbations may be reformulated as the
completeness results for singular perturbations which are their inverses.
In this subsection, we state several theorems for singular rank one perturbations,
which are equivalent to the main results from Introduction.

Most of our results will concern the case when $\mu $ is a discrete measure:
$\mu=\sum_{n\in \mathcal{N}} \mu_n\delta_{t_n}$,
where $\lim_{|n|\to\infty} |t_n|=\infty$.
Then the spectrum of $\A $ is
$$
\sigma(\A ) =  \{t_n: n\in \mathcal{N}\}
$$
and has no accumulation points on $\RR$.
In this situation, we will assume that $\{t_n\}$ is strictly
increasing, and the index set $\mathcal{N}$ can be $\BN$, $-\BN$ or $\BZ$
(for doubly infinite sequences). Then we will say that
\textit{$\A $ has a discrete spectrum} and will use the notation
$$
a_n=a(t_n), \qquad b_n=b(t_n).
$$

If   $\LL =\LL (\A ,a,b,\deab)$
is
a singular rank one perturbation of
a cyclic selfadjoint operator with discrete spectrum $\{t_n\}$,
then $\A^{-1}$ is a compact operator
with spectrum $\{s_n\} = \{t_n^{-1}\}$
and $\LL^{-1}$ is a rank one perturbation
of $\A^{-1}$ whenever $\deab \ne 0$.

Generalized weak singular perturbations were introduced in Proposition
\ref{prp-Mats-type-sing}, thus giving a singular perturbation
version of Theorem~\ref{genweak0}. For another version of this statement
(and a different proof) see Theorem \ref{3to1}. We now state an analog
of Theorem \ref{positive0} for singular perturbations.

\begin{theorem}
\label{positive} Suppose that $\LL (\A ,a,b,\deab)$ is a real type
perturbation, $a_n\overline b_n\ge 0$ for all but possibly a
finite number of values of $n$.
Suppose also that $b_n \ne 0$ for any $n$, and
$\sum_{n\in \mathcal{N}} |t_n|^{-1}|a_nb_n|\mu_n = \infty$. Then
$\LL^*$ is correctly defined, and $\LL $ and $\LL^*$ are complete.
\end{theorem}

The following two theorems are counterparts of Theorems
\ref{adjoint0} and \ref{frombb0} for singular rank one perturbations.

\begin{theorem}
\label{noncompleteness2}
For any cyclic selfadjoint operator $\A$ with
discrete spectrum, there exists a strong real type
singular rank one perturbation $\LL=\LL (\A ,a,b,\deab)$ of $\A$, which is not complete,
while its adjoint $\LL^*$  is correctly defined, has trivial kernel and is
complete.
Moreover, the data $(a,b,\deab)$
can be chosen in such a way that $a,b\notin \Elltwo(\mu)$
and the orthogonal complement to
the space spanned by the
eigenvectors of $\LL$ is infinite-dimensional.
\end{theorem}

\begin{theorem}
\label{frombb}
Assume the data $(a, b, \deab)$ satisfy
$a\notin \Elltwo(\mu)$  $($and, thus, $({\rm A })$ holds$)$.
Assume also condition $({\rm A}^*)$.
Let the perturbation $\LL =\LL (\A , a,b, \deab)$
be complete. Then its adjoint $\LL^*$ is also complete
if any of the following conditions is fulfilled:
\smallskip

$($i$)$ $|a_n|^2 \mu_n  \ge C |t_n|^{-N}>0$ for some $N>0$\textup;
\smallskip

$($ii$)$ $|b_n a_n^{-1}| \le C |t_n|^{N}$ for some $N>0$.
\end{theorem}

Analogously, one can reformulate the statements about the spectral synthesis
(Theorems \ref{synthesis} and \ref {sp_sint})
for the case of singular rank one perturbations.


\subsection{Functional model: statement}
In the below functional model for singular rank one perturbations
we always make the following
\vskip.3cm

\noindent \textbf{Assumption.}
The function $a\in (1+|x|)^{-1} \Elltwo(\mu)$ is nonzero and
$b\in (1+|x|)^{-1} \Elltwo(\mu)$ is a {\it cyclic} vector
for the resolvent of $\A $, that is, $b\ne 0$ $\mu$-a.e.
\vskip.3cm

Let $\BC^\pm=\big\{z\in \BC:  \pm \ima z>0\big\}$
denote the upper and the lower half-planes and set $H^2=H^2(\BC^+)$ to be
the Hardy space in $\BC^+$.
For the basic properties of the space $H^2$ we refer to~\cite{hj, ho}.

Define
\beqn
\label{singul2}
\nuu(z)  = \deab + z b^*(\A -z)^{-1}\A^{-1}a                    
         = \deab + \int \bigg(\frac{1}{x-z}
            -\frac{1}{x}\bigg)\, a(x)\overline{b(x)} \, d\mu (x),
\neqn
\beqn
\label{singul1}
\rho(z)
 = \de + zb^*(\A -z)^{-1}\A^{-1}b  = \de + \int \bigg(\frac{1}{x-z}-\frac{1}{x}\bigg) |b(x)|^2
\, d\mu (x),
\neqn
where $\delta$ is an arbitrary real constant.
(The function $\be$ is the same as in
Proposition~\ref{prop_shift}.)  Since $\mu(\{0\})=0$, $zb^*(\A
-z)^{-1}\A^{-1}a\ne \const$, and therefore $\nuu \not\equiv 0$ in
$\BC\sm\BR$.

We set
\begin{align}
\label{the}
\Theta(z)& =  \frac{i-\rho(z)}{i+\rho(z)}\, , \\
\label{defphi}
\phi(z) & =\frac {\nuu(z)} 2\,  (1+\Theta(z)).
\end{align}
%
%
It is easy to see that $\Theta $ and $\phi$ are analytic in $\BC^+$.
Since $\mu$ is a singular measure on $\BR$,
it follows that $\Im \rho(z)\ge 0$ for $z\in {\BC}^+$ and
$\Im \rho(z)= 0$ a.e. on $\BR$. Therefore
$\Theta$ is an inner function in the upper half-plane $\BC^+$
(that is, a bounded analytic function with $|\Theta| =1$ a.e. on $\RR$
in the sense of nontangential boundary values).
Therefore $\Theta$
generates a {\it backward shift invariant} or {\it model
subspace} $K_\Theta \defin H^2\ominus \Theta H^2$ of the Hardy space $H^2$.
These subspaces, as well as their vector-valued generalizations,
play an outstanding role both in function theory and in operator theory.
For their numerous applications we refer to~\cite{nk12}.

The following statement will be our main tool for studying the
rank one perturbations:

\begin{theorem}[a functional model]
\label{rank-one-model}
Let $\LL =\LL (\A ,a,b, \deab)$ be a singular rank one perturbation of $\A $
with $b\ne 0$ $\mu$-a.e.,
and let $\Theta$ and $\phi$ be defined by the above formulas
\eqref{singul2}--\eqref{defphi}.
Then $\Theta$ is analytic in a neighborhood
of $0$,  $1+\Theta \notin H^2$, $\Theta(0) \ne -1$,
\beqn
\label{main0}
\phi\notin H^2, \qquad
\frac{\phi(z) -\phi(i)}{z-i} \in K_\Theta,
\neqn
and $\LL $ is unitary equivalent to the operator $T=T_{\Theta,\phi}$
which acts on the model space $K_\Theta \defin H^2\ominus \Theta H^2$
by the formulas
$$
\cD(T) \defin \{f=f(z)\in K_\Theta: \text{there exists}\ c=c(f)\in \BC:
zf-c\phi\in K_\Theta\},
$$
$$
Tf \defin zf - c \varphi, \qquad f\in \cD(T).
$$
If, moreover, $\LL $ is a real type
singular rank one perturbation, then
\beqn
\label{main}
\Theta = \frac{\phi}{\bar\phi} \qquad \text{a.e. on $\BR$}.
\neqn

Conversely, any inner function $\Theta$ which is analytic in a neighborhood
of $0$ and satisfies $1+\Theta \notin H^2$, $\Theta(0) \ne -1$,
and any function $\phi$ satisfying \eqref{main0}
correspond to some singular rank one perturbation
$\LL = \LL (\A , a, b, \deab)$ of the operator $\A $ of multiplication by the independent
variable in $\Elltwo(\mu)$, where $\mu$ is some singular measure on $\RR$ and
$x^{-1}  a(x)$, $x^{-1}  b(x) \in \Elltwo(\mu)$.
If, moreover, $\Theta$ and $\phi$ satisfy \eqref{main}, then
the  perturbation $\LL $ is of real type.
\end{theorem}

This model is close to Kapustin's model for rank one perturbations of
singular unitary operators~\cite{Kap}.
In fact, one of the initial goals of this work was to extend the
model of Kapustin to the rank one singular perturbations
of unbounded selfadjoint operators.
We refer to~\cite{Kap1} for a more general construction.

Gubreev and Tarasenko in~\cite{Gubr-Tar2010} constructed a model
for operators that have a discrete spectrum (compact resolvent),
are neither dissipative nor anti-dissipative, whose imaginary part
is two-dimensional, under an additional restriction that their
spectrum does not intersect the real axis (last restriction does
not seem to be essential). If we suppose that $a,b\in \Elltwo(\mu)$,
then $\LL =\LL (\A ,a,b,\deab)$ is a rank one bounded perturbation
of $\A $ (see \S 2 below), and it follows that $\LL $ has
two-dimensional imaginary part and is neither dissipative nor
anti-dissipative (unless $\LL =\LL^*$). It is easy to see that,
conversely, any operator
$\A +i(f_1 f_1^* - f_2 f_2^*)$ on $H$, where $\A =\A^*$,
$f_j\in H$ ($j=1,2$) can be represented as $\A_1+ a b^*$, where
$\A_1=\A_1^*$ and $a, b$ are
linear combinations of $f_1, f_2$. Hence the class of operators we
consider is very close to the class of operators in the
paper~\cite{Gubr-Tar2010}, and in fact, in the case of
discrete spetrum their model is essentially the same as ours.
Paper~\cite{Gubr-Tar2010} also contains certain results
on completeness and on Riesz bases
of eigenvectors of the perturbed operator (later on, we will
comment on the connections between these results and ours).

In~\cite{KisNab06}, Kiselev and Naboko study a general operator
with two-dimensional imaginary part by making use of the Naboko
model. A related model for operators of this class that have real
spectrum was constructed by Zolotarev in~\cite{zolot2002} in terms
of certain generalization of the de Branges spaces. The main point
in the works by Kiselev and Naboko \cite{KisNab06}, \cite{KisNab09} and
others is the study of so-called almost Hermitian non-dissipative
operators; this is a stronger requirement than the assumption that
the spectrum is real. By using~\cite[Theorem~3.1]{KisNab09}, it is
easy to check that in the model given by the above theorem, $\LL
(\A ,a,b, \deab)$ is almost Hermitian if and only if $\phi$ is
{\it outer}. The perturbations in Theorems \ref{noncompleteness2}, \ref{adjoint0} and
\ref{synthesis} can be chosen to be almost Hermitian, as is seen
from the proofs.

For a more general model based on operator-valued characteristic
functions see~\cite{Ryzhov}.


\subsection{Eigenfunctions of the model operator}
In view of Theorem~\ref{rank-one-model}, the completeness of $\LL $ (or $\LL^*$)
translates into the completeness of $T$ or $T^*$, respectively.

In what follows, we use the term ``meromorphic'' in the sense
``meromorphic in $\BC$''.
In the case when $\A $ has a discrete spectrum,
$\Theta$ and $\phi$ are meromorphic.
Then any function in $K_\Theta$ is also meromorphic in $\BC$.
As it will be explained in Section \ref{herm--biehl} below,
this situation reduces to the study of de Branges spaces
of entire functions.

In the case when $\Theta$ is meromorphic, the eigenfunctions
of $T$ are of the form $\frac{\phi}{z-\lambda}$, $\lambda\in Z_\phi$,
where $Z_\phi$ is the zero set for $\phi$, while the eigenfunctions
of $T^*$ are just the reproducing kernels of $K_\Theta$
(see Lemma~\ref{lem_eigs} below for a more accurate statement).
So the completeness of eigenfunctions of $T^*$ means that any function
in $K_\Theta$ vanishing on $Z_\phi$ is zero,
that is, that $Z_\phi$ is a {\it uniqueness set} for $K_\Theta$.

As a corollary of Theorem~\ref{rank-one-model},
we  show that under some mild restrictions
any complete and minimal system of reproducing kernels in a model space and its
biorthogonal system can be realized (up to a unitary equivalence)
as the systems of eigenfunctions of some rank one singular perturbation
and of its adjoint.

\begin{theorem}
\label{mincomp}
Let $\Theta$ be an inner function analytic in a neighborhood of $0$
and such that
$1+\Theta \notin K_\Theta$. Let the system of reproducing kernels
$\{k_\lambda\}_{\lambda\in \Lambda}$, $\Lambda\subset \BC^+$
$($or $\Lambda\subset \clos\BC^+$ in the case when
$\Theta$ is meromorphic$)$, be complete and minimal in $K_\Theta$.
Then there exists a function $\phi$
such that $\phi \notin H^2$,
$\phi$ vanishes exactly on the set $\Lambda$ and
$\frac{\phi(z)}{z-\lambda} \in K_\Theta$
for any $\la\in \Lambda$. Moreover, $\phi$ is determined uniquely
up to a multiplicative constant, and the following statements hold.
\begin{enumerate}

\item $\Theta$ and $\phi$ correspond to some singular rank
one perturbation $\LL $ of the
multiplication operator in $\Elltwo(\mu)$
with $x^{-1} a$, $x^{-1} b \in \Elltwo(\mu)$.

\item If, moreover, $\zeta-\Theta \notin H^2$ for any $\zeta\in\BC$,
$|\zeta|=1$, then the adjoint $\LL^*$ of $\LL $ exists and
the system $\{k_\lambda\}_{\lambda\in \Lambda}$
is the set of eigenfunctions of the corresponding
operator $T^*$.

\item If, moreover, $\Theta$ is a meromorphic inner function
and $\Lambda \subset{\RR}$, then
there is a constant $\xi$, $|\xi|=1$, such that
$\Theta$ and $\xi \phi$ correspond to some almost Hermitian singular rank one perturbation.
\end{enumerate}
\end{theorem}

For the proof of Theorem~\ref{mincomp} see Subsection \ref{tse}.

If $\phi$ is the function, defined as above, we will refer to
it as a \textit{generating function}
for the system $\{k_\lambda\}_{\lambda\in \Lambda}$.

Interestingly, completeness problems for a minimal system
of reproducing kernels in a model (or de Branges) space
and for its biorthogonal are, in general,
not equivalent~\cite{bar-belov}. This is the reason for existence
of the unexpected examples of Theorems~\ref{adjoint0}
and \ref{noncompleteness2}.
However, for certain classes of perturbations
the completeness of $T^*$ implies the completeness of $T$,
see Theorems \ref{frombb} and \ref{3to1}.
\bigskip



\section{Basic Functional Model: Proofs}

Let $\LL =\LL (\A , a,b,\deab)$ be a singular rank one perturbation of a
cyclic selfadjoint operator $\A $, given by \eqref{defA},
where $\mu$ is a singular measure.
Notice first of all that $\LL $ can be expressed as follows:
\beqn
\label{repr-L}
\LL y=(\A +a (b^*)_{a,\deab})y \defin \A y+ ((b^*)_{a,\deab} y) a, \qquad y\in \cDL,
\neqn
where the linear functional
$(b^*)_{a,\deab}: \cDL\to \BC$ is defined by
\beqn
\label{b-L}
(b^*)_{a,\deab}y \defin -c
\neqn
whenever $y=y_0+c\cdot \A^{-1}a$ is a decomposition of a vector $y\in \cDL$
as in \eqref{2}. The summands in the right hand part of
\eqref{repr-L} are elements of $x\Elltwo(\mu)$ and in general
do not belong to $\Elltwo(\mu)$.

If $b\in \Elltwo(\mu)$, then
$(b^*)_{a,\deab}$ is a bounded functional, given by
$$
(b^*)_{a,\deab}y=\frac 1{\deab-\langle \A^{-1}a, b\rangle}\,
\langle y, b\rangle, \qquad y\in \Elltwo(\mu).
$$
If, moreover, $a, b\in \Elltwo(\mu)$, then
$\cDL=\cDA$, and $\LL (\A , a,b,\deab)$ is just a
bounded rank one perturbation of $\A $. In particular,
$$
\LL (\A , a,b,\deab)=\A +a   b^*, \qquad
\text{if} \enspace
\deab=1+\langle \A^{-1}a,b \rangle.
$$
Other values of $\deab$ do not enlarge the scope of perturbations considered.
Therefore in the case when $a, b\in \Elltwo(\mu)$, we may
consider the usual rank one perturbation
$$
\LL (\A , a,b)\defin \A +a   b^*.
$$
%
%
However, most of the time we do not need to distinguish
this case and all the results remain valid in this situation.


\subsection{Model spaces and Clark measures}
\label{clark}

We recall that the functions $\nuu$, $\rho$, $\Theta$ and $\phi$
have been defined above by formulas \eqref{singul2}--\eqref{defphi}.
Since $\mu(\{0\})=0$, $zb^*(\A -z)^{-1}\A^{-1}a\ne \const$, and therefore $\nuu \not\equiv 0$ in
$\BC\sm\BR$.

Here we will discuss the model space $K_\Theta$, the so-called Clark measures,
related to it, and Clark orthogonal bases of reproducing kernels in $K_\Theta$.
If we identify the functions in $H^2$ with their boundary values
on $\RR$, then an equivalent definition of $K_\Theta$ is
$K_\Theta = H^2 \cap \Theta \overline{H^2}$. Thus, we have a criterion
for the inclusion $f\in K_\Theta$ which we will repeatedly use:
\beqn
\label{crit}
f\in K_\Theta \ \Longleftrightarrow \ f\in H^2 \ \ \text{and} \ \
\Theta \overline f\in H^2.
\neqn
In other words, if $f \in K_\Theta$, then the function
$\tilde f(x) = \Theta(x) \overline{f(x)}$, $x\in\BR$,
is also in $K_\Theta$ (in the sense that it coincides with nontangential
boundary values of a function from $K_\Theta$).
If $\Theta$ is meromorphic, then  any
$f\in K_\Theta$ is also meromorphic, and the formula
$\tilde f(z) = \Theta(z) \overline{f(\overline z)}$
holds for all $z\in\BC^+$.

The following statement is an immediate corollary of (\ref{crit}).

\begin{lemma}
\label{lem1}
Let $\Theta$ be inner and let $\frac{\phi}{z+i} \in H^2$.
Assume that $\phi$ satisfies either \eqref{main} or $\frac{\phi(z)-\phi(i)}{z-i}\in K_\Theta$.
Then

\begin{enumerate}

\item If $\phi \in H^2$, then $\phi \in K_\Theta$.

\item Let $\lambda\in \BC^+ \cup \RR$ and let
$\phi$ be analytic in a neighborhood of $\lambda$.
Then the function $g_\lambda(z) = \frac{\phi(z) - \phi(\lambda)}{z-\lambda}$
belongs to $K_\Theta$.
\end{enumerate}
\end{lemma}

We remark that if $\Theta$ be inner and $\frac{\phi}{z+i} \in H^2$, then
\eqref{main} implies that $\frac{\phi(z)-\phi(i)}{z-i}\in K_\Theta$.

\begin{proof}[Proof of Lemma \ref{lem1}]
(1) We have $ \Theta(x) \overline{\phi(x)} = (x+i) \Theta(x)
\overline{\psi(x)} +\overline{\phi(i)} \Theta(x)$
for some $\psi\in K_\Theta$, which gives that $\frac{\Theta\overline\phi}{x+i}
\in H^2$. If, moreover, $\phi \in H^2$, then $\phi\in K_\Theta$
by (\ref{crit}), which proves (1).
%

(2) Obviously, $g_\lambda \in H^2$. Also,
$\Theta (x) \overline{g_\lambda(x)}(x-\overline \lambda)
= \Theta(x) \overline{\phi(x)} -\overline{\phi(\lambda)} \Theta(x)$,
whence $\Theta \overline{g_\lambda} \in H^2$.
\end{proof}

For $\lambda\in \BC^+$ set
$$
k_\lambda(z) = \frac{1-\overline{\Theta(\lambda)}\Theta(z)}
{z-\bar \lambda}, \qquad
\tilde k_\lambda(z) = \frac{\Theta(z) - \Theta(\lambda)}{z-\lambda}.
$$
Note that $\tilde k_\lambda(z) = \Theta(z) \overline{k_\lambda(\overline z)}$.
The definitions of $k_\lambda$ and $\tilde k_\lambda$ can be extended to the
points $\lambda\in \RR$ such that $\Theta$ has an analytic
extension to a neighborhood of $\lambda$, and
$\tilde k_\lambda = - \Theta(\lambda) k_\lambda$ for
these values of $\lambda\in \RR$.

Note that $k_\lambda$ is the {\it reproducing kernel of} $K_\Theta$
corresponding to the point $\lambda$, i.e.,
\beqn
\label{repr-fla}
\langle f, k_\lambda \rangle =2\pi i f(\lambda), \qquad f\in K_\Theta.
\neqn
Analogously,
$\langle f, \tilde k_\lambda \rangle = - 2\pi i \,\overline{\tilde f(\lambda)}$.

Orthogonal bases of reproducing kernels were studied by
L. de Branges~\cite{br} for meromorphic inner functions
and by D.N. Clark~\cite{cl} in the general case.
They may be constructed as follows.
For any $\zeta\in{\BC}$, $|\zeta|=1$, the function
$(\zeta+\Theta)/(\zeta-\Theta)$ has positive real part in the upper half-plane.
Then, by the Herglotz theorem,
there exist $p_\zeta\ge 0$, $q_\zeta\in {\RR}$
and a measure $\sigma_\zeta$ with $\int_\RR (1+t^2)^{-1} d\sigma_\zeta(t)<\infty$
such that
\beqn
\label{clark-meas}
\frac{\zeta+\Theta(z)}{\zeta-\Theta(z)}= -i p_\zeta z +iq_\zeta
+\frac{1}{i}
\int_{\RR} \Big(\frac{1}{t-z} -\frac{t}{t^2+1} \Big) \,
d \sigma_\zeta(t),  \qquad z\in\mathbb{C^+}.
\neqn

We will say that $\sigma_\zeta$ is a \textit{Clark measure of $\Theta$}.
It follows from the results of Clark~\cite{cl}
(for the setting of the unit disc, instead of the upper half-plane)
that in the case when $p_\zeta=0$ the mapping
\beqn
\label{clark-meas1}
(U_\zeta f)
(z)
 = \frac 1 {2\,\sqrt{\pi}} \big(\zeta-\Theta(z)\big) \int_{\RR}\frac{f(t)}{t-z}\,
d\sigma_\zeta(t)
\neqn
is a unitary operator from $\Elltwo(d\sigma_\zeta)$ onto $K_\Theta$ (see
Proposition~\ref{splittings} below).
The inverse operator to $U_\zeta$ has the sense just of
the embedding $K_\Theta \subset \Elltwo(\pi \sigma_\zeta)$;
clearly, it also is unitary. As Poltoratskii has shown in~\cite{polt},
any function $f \in K_\Theta$ has finite nontangential boundary
values $\sigma_\zeta$-a.e.

In particular, if $\sigma_\zeta$ is purely atomic
(that is, $\sigma_\zeta=\sum_n c_n \delta_{t_n}$,
where $\delta_x$ denotes the Dirac measure at the point $x$)
and $p_\zeta=0$, then
$k_{t_n}\in{K^2_{\Theta}}$ and the system
$\{k_{t_n}\}$ of reproducing kernels
is an orthogonal basis in ${K_{\Theta}}$.
Of course, if $\Theta$ is meromorphic, then any Clark measure
$\sigma_\zeta$ is atomic.

If $p_\zeta>0$ in the representation (\ref{clark-meas}),
then the orthogonal complement to the span of $\{k_{t_n}\}$ is one-dimensional
and is generated by the function $\zeta-\Theta$, which in this case
belongs to $K_\Theta$ (this is half-plane version
of a result due to Ahern and Clark~\cite{ahe-cl70}).
Thus,
\beqn
\label{mass}
\zeta-\Theta \in H^2 \
\Longleftrightarrow \ \zeta - \Theta(iy) = O(y^{-1}), \ y\to \infty \
\Longleftrightarrow  \ p_\zeta>0.
\neqn

We will use the notation
\beqn
\label{defnu}
\nu = |b|^2\mu.
\neqn
Returning to our model of singular perturbations, observe that
the representations (\ref{singul1})--(\ref{the}) mean that
$\nu$ is the Clark measure $\sigma_{-1}$ for
$\Theta$. By the results of~\cite{polt},
\beqn
\label{phi-a-b}
\phi = ia/b \qquad \nu\text{-a.e.}
\neqn

\begin{proposition}
\label{param}
Let $a,b$ be functions that satisfy \eqref{cond-a-b} and let $\deab\in \BR$.
Let $\Theta$ and $\phi$ be defined by \eqref{the} and \eqref{defphi}.
Then we have:
\begin{enumerate}

\item $1+\Theta\notin H^2$ and $\Theta(0)\ne -1$\textup;

\item $\frac{\phi(z) - \phi(i)}{z-i} \in K_\Theta$\textup;

\item If $a\notin \Elltwo(\mu)$, then $\phi \notin H^2$\textup;

\item If $a\in \Elltwo(\mu)$, then $\phi \in H^2$ if and only if
$\deab=\int x^{-1}a(x)\overline{b(x)} d\mu(x)$\textup;

\item $\deab =0$ if and only if $\phi(0) =0$.
\end{enumerate}
\end{proposition}

\begin{proof}[Proof]
(1) It follows from \eqref{singul1} that $|\rho(iy)|=o(y)$ as $y\to \infty$.
Since $\frac {1-\Theta}{1+\Theta}=- i \rho$, it follows from
\eqref{mass} that $1 + \Theta \notin H^2$.

(2) It follows from the formula (\ref{defphi}) for $\phi$ that
$$
\frac{\phi(z) -\phi(i)}{z-i}
 = \frac{1+\Theta(z)}{2}\int
\frac{a(x) \overline{b(x)}}{(x-z)(x-i)} \, d\mu(t)
+ \beta(i) \,\frac {\Theta(z)-\Theta(i)}{2\,(z-i)}.
$$
Since $\frac{a}{x-i}\in  \Elltwo(\mu)$, we have
$\frac{a}{(x-i) b}\in  \Elltwo(\nu)$, where $\nu = |b|^2 \mu$.
Since $\nu$ is the Clark measure $\sigma_{-1}$ for
$K_\Theta$, and the Clark operator $U_{-1}$ maps $\Elltwo(\nu)$
onto $K_\Theta$, we have
$\frac{\phi(z) - \phi(i)}{z-i} \in K_\Theta$.

(3) If $\phi\in H^2$, then $\phi\in K_\Theta$.
Hence, $\phi \in \Elltwo(\nu)$, and, since
$\phi = ia/b$ $\nu$-a.e., we have
$\int |a(x)|^2 \,d \mu(x) = \int |\phi(x)|^2 \,d\nu(x) <\infty$.

(4) Now let $a\in \Elltwo(\mu)$. Then we have
$$
 2 \phi(z) = (1+\Theta(z))
 \Big(\deab - \int \frac{a(x) \overline{b(x)}}{x}
 d \mu(x) \Big)
+ (1+\Theta(z)) \int \frac{a(x) \overline{b(x)}}{x-z} d\mu(x).
$$
Since $\int |a|^2 |b|^{-2} d\nu = \int|a|^2 d\mu<\infty$,
the boundedness of the operator $U_{-1}$ implies
that the last term is in $K_\Theta$.
Hence, $\phi \in K_\Theta$ if and only if
$(1+\Theta(z))\big(\deab -\int x^{-1} a(x) \overline{b(x)} d \mu(x) \big)$ is in $K_\Theta$.
Since $1+\Theta \notin\ K_\Theta$, we conclude that $\phi \in K_\Theta$
if and only if the coefficient is zero.

(5) Obviosly, $\deab =0$ if and only if $\beta(0) = 0$.  By (1), we have
$1+\Theta(0) \ne 0$, whence the statement follows.
\end{proof}

\subsection{Proof of Theorem~\ref{rank-one-model} on the model in $K_\Theta$}
\label{pr32}

We define ``diagonalizing'' transforms
\begin{align}
V_0 u(z)=\frac 1 {2\;\root\of\pi}\; b^*(\A -z)^{-1}u, \qquad
z\in \rho(\A ), \,
u\in \Elltwo(\mu), \\
V_\LL  u(z)=\frac 1 {\root\of\pi}\; (b^*)_{a,\deab}(\LL -z)^{-1}u, \qquad
z\in \rho(\LL ), \,
u\in \Elltwo(\mu).
\end{align}

First we need the following proposition.

\begin{proposition}
\label{splittings}
Let $\Theta$ and $\phi$ be defined by (\ref{the}) and (\ref{defphi}). Then

\begin{enumerate}

\item $V_0$ is an isometric isomorphism of $\Elltwo(\mu)$ onto
$\frac 1 {1+\Theta}K_\Theta$\textup;

\item $V_\LL $ is an isometric isomorphism of $\Elltwo(\mu)$ onto
$\frac 1 \phi K_\Theta$.

\end{enumerate}
\end{proposition}

\begin{proof}[Proof]
First let us deduce the splitting formula
\beqn
\label{rel-L0}
V_\LL  u(z)=\nuu(z)^{-1} V_0  u(z), \qquad u\in \Elltwo(\mu).
\neqn
To do that, choose any $u\in \Elltwo(\mu)$ and put $y\defin (\LL-z)^{-1}u\in \cDL$.
A direct calculation shows that the representation $y=y_0+c\A^{-1}a$ as in
\eqref{2} is given by
$$
y_0=(\A -z)^{-1}u + cz(\A -z)^{-1}\A^{-1}a, \qquad
c=
-\nuu(z)^{-1}b^*(\A-z)^{-1} u.
$$
By \eqref{b-L}, this implies \eqref{rel-L0}.

Statement (1) is very close to the results by Clark
~\cite{cl}. To prove it, one can apply the arguments
given in~\cite[Proposition~9.5.4]{CimaMathRoss}.
Namely, for $\xi\in \BR\sm\supp \mu$, put
$$
\eta_\xi(x)=\frac 1 {\rho(\xi)+i} \; \frac {b(x)}{x-\xi}\in \Elltwo(\mu).
$$
Let
$\xi, \tau\notin \supp \mu\cup \bar Z_\Theta$. Direct calculations give that
\begin{gather*}
(1+\Theta(z))V_0\eta_\xi(z)=-\frac {\Theta(z)-\Theta(\xi)}{2\;\root\of\pi(z-\xi)}
=\frac{\Theta(\xi)}{2\;\root\of\pi}\,k_{\xi}(z)\in K_\Theta;                   \\
\langle \eta_\xi, \eta_\tau \rangle_{\Elltwo(\mu)}
=
- \frac 1 {2i}\, \frac {1-\overline{\Theta(\tau)} \Theta(\xi)}{\xi- \tau}
=
\bigg\langle \frac {\Theta(\xi)}{2\;\root\of\pi}\, k_{\xi}, \;
\frac{\Theta(\tau)}{2\;\root\of\pi}\, k_{\tau} \bigg\rangle_{K_\Theta}.
\end{gather*}
Since $\{\eta_\xi\}$ are complete in $\Elltwo(\mu)$ and
$\{\Theta(\xi)k_{\xi}\}$ are complete in $K_\Theta$, the assertion (1) follows.

Statement (2) follows from statement (1) and formula \eqref{rel-L0}.
\end{proof}

\begin{proof}[Proof of Theorem~\ref{rank-one-model}]
Put $V_{\LL ,\phi} u\defin \phi\cdot V_\LL  u$, $u\in \Elltwo(\mu)$.
By Proposition~\ref{splittings}, $V_{\LL ,\phi}$ is an
isometric isomorphism from $\Elltwo(\mu)$ onto $K_\Theta$.
Define an operator $T=V_{\LL ,\phi}\LL  V_{\LL ,\phi}^{-1}$ on $K_\Theta$.
It is unitarily equivalent to $\LL $.  The splitting formula
$$
V_\LL (\LL -\xi)^{-1}u(z)=
\frac
{V_\LL  u(z)-V_\LL  u(\xi)}
{z-\xi}, \qquad \xi\notin \sigma(\LL ),
$$
is immediate. It easily implies the expressions for
$\cD(T)$ and for the action of $T$, given in the Theorem.

The properties of $\Theta$ and $\phi$ follow immediately from
Proposition~\ref{param}. If, moreover,
$\LL $ is a real type perturbation, then $\nuu$ is real a.e. on $\BR$,
and it follows that $\Theta$ and $\phi$ satisfy \eqref{main}.

We turn to the proof of the converse statement
and show that any pair $(\Theta, \phi)$ with the above properties
can be realized in our model. Let $\nu$ be the Clark
measure $\sigma_{-1}$ for $\Theta$.
Note that by the hypothesis $0 \notin {\rm supp}\, \nu$
and also that, by (\ref{mass}), $p_{-1} =0$.
Since $\frac{\phi(z) -\phi(i)}{z-i}$
belongs to the space $K_\Theta$, by
(\ref{clark-meas1}) there exists $u\in \Elltwo(\nu)$ such that
$$
\frac{\phi(z) -\phi(i)}{z-i} =
\big(1+\Theta(z)\big) \int\frac{u(t)}{t-z} \, d\nu(t).
$$
Choose any $b$ so that $|b|>0$ $\nu$-a.e. and put
$\mu = |b|^{-2}\nu$. Then $x^{-1}b \in \Elltwo(\mu)$. We have
$$
\phi(z) = \phi(i) + (1+\Theta(z)) \bigg[\int
\bigg(\frac{1}{x-z}-\frac{1}{x}\bigg) (x-i)u(x)|b(x)|^2 d\mu (x) -
\int\frac{u(x)}{x}d\nu(x) \bigg].
$$
Since, by definition of the Clark measure (\ref{clark-meas}),
$$
(1+\Theta(z))\bigg(r_0+ \frac{1}{\pi i}
\int \bigg(\frac{1}{x-z}-\frac{1}{x}\bigg) d\nu (x) \bigg) = 2
$$
for some constant $r_0$, we can
write $\phi(i)$ as an  analogous integral
and finally obtain
$$
\phi(z) = \frac{1+\Theta(z)}{2}
\bigg(\deab + \int \bigg( \frac{1}{x-z}-\frac{1}{x}\bigg)v(x)|b(x)|^2 d\mu(x)\bigg),
$$
for some constant $\deab$ and $v$ such that $x^{-1}v \in \Elltwo(\nu) = \Elltwo(|b|^2\mu)$.
Put $a(x) = v(x)/b(x)$. Then $x^{-1} a \in \Elltwo(\mu)$
and
$$
\phi(z) = \frac{1+\Theta(z)}{2}
\bigg(\deab + \int \bigg( \frac{1}{x-z}-\frac{1}{x}\bigg)
a(x)\overline b(x) d\mu(x)\bigg).
$$
Since $\phi \notin K_\Theta$, by Proposition~\ref{param}, either $a\notin \Elltwo(\mu)$
or $\deab \ne \int x^{-1}a(x) \overline b(x) d\mu(x)$ and so $({\rm A })$
is satisfied. We conclude that $\Theta$ and $\phi$ correspond
to the singular perturbation associated with the measure $\mu$ and
the data $a$, $b$ and $\deab$.

Finally, note that if $\Theta$ and $\phi$ satisfy (\ref{main}), then
$\nuu$ is real on ${\RR}$ whence $a\overline b \in {\RR}$
and $\deab \in {\RR}$. Thus the constructed perturbation is of real type.
\end{proof}

\begin{remark}
{\rm It follows from this theorem that if two operators
$\LL $, $\LL _1$ as above are real type perturbations and $\phi=\phi_1$, then
$\LL $ and $\LL _1$ are unitarily equivalent. }
\end{remark}


\subsection{The spectrum and eigenfunctions of $T$ and of $T^*$}
\label{tse}
Throughout this section we will assume that
$\Theta$ and $\phi$ are meromorphic. Then it follows from Lemma~\ref{lem1}
that for any $\lambda \in \BC^+\cup\RR$ such that $\phi(\lambda) = 0$,
the function
$$
h_\lambda(z) = \frac{\phi(z)}{z-\lambda}
$$
belongs to $K_\Theta$. Denote by $Z_\phi$ the set of zeros of $\phi$
in $\BC^+\cup\RR$. Recall that if we put $\tilde \phi(z) = \Theta(z)
\overline{\phi(\overline z)}$, then $\tilde \phi$ is analytic in $\BC^+$.
Denote by $Z_{\tilde \phi}$ the zero set of $\tilde \phi$ in $\BC^+ \cup\RR$
(note that the zeros of $\phi$ and $\tilde \phi$ on $\RR$ coincide).

The following lemma describes the spectrum
of the model operator $T$ and of its adjoint.
For a set $Z\subset \BC$ we denote by $\overline Z$ the set $\{\overline z: \, z\in Z\}$.

\begin{lemma}
\label{lem_eigs}

Let meromorphic $\Theta$ and $\phi$ correspond to a
singular rank one perturbation of a cyclic selfadjoint operator
$\A $ with the compact resolvent. Then the following holds:
\begin{enumerate}

\item Operators $\LL $ and $T$ have compact resolvents\textup;

\item $\si(T)=\si_p(T)=Z_\phi \cup \overline Z_{\tilde \phi}$\textup;

\item The eigenspace of $T$ corresponding to an eigenvalue
$\la$, $\la\in Z_\phi \cup \overline Z_{\tilde \phi}$, is spanned by
$h_\la$\textup;

\item Suppose that either $\Im\la \ge 0$ and $\la$ is a zero of
$\phi$ (exactly) of order $k$, or
$\Im\la < 0$ and $\bar \la$ is a zero of $\tilde \phi$ of order $k$. Then
$\dim\ker (T-\la)^\ell=\ell$ for $\ell \le k$, and
$\ker (T-\la)^s=\ker (T-\la)^k$ for $s>k$. Moreover,
for $1\le \ell \le k$, the space $\ker (T-\la)^\ell$ is spanned by
the eigenvectors and root vectors $\phi(z)/(z-\la), \phi(z)/(z-\la)^2, \dots,
\phi(z)/(z-\la)^\ell$ of $T$.

\item Suppose $T^*$ is correctly defined. Then
$\si(T^*)=\overline{Z}_\phi \cup Z_{\tilde \phi}$ and
$\ker(T^*-\overline \la I)$ is spanned by $k_{\la}$ for $\lambda\in Z_\phi$,
while $\ker(T^*- \la I)$ is spanned by $\tilde k_{\la}$ for
$\lambda\in Z_{\tilde \phi}$.

\end{enumerate}
\end{lemma}

\begin{proof}[Proof]
Since $\A $ has a compact resolvent and $\LL $ is its finite rank perturbation in the
sense of~\cite{Aziz_Behr_etl}, $(\LL -\la I)^{-1}$ is compact for any
$\la\notin \si(\LL )$. Hence the resolvent of $T$ also is compact.
This gives statement (1) and also implies that $\si(T)=\si_p(T)$.

Now let us describe the eigenvalues and eigenfunctions of $T$.
If for some $\la \in \BC$
$$
Tf=zf- c\phi  =\lambda f,
$$
then $f = \frac{c\phi}{z-\lambda}$. Hence $\la$ is in $\sigma_p(T)$ if and only if
$\frac{\phi(z)}{z-\lambda}\in K_\Theta$. It follows also that
for any $\la\in \si_p(T)$, the eigenspace $\ker (T-\la)$ is one-dimensional.
If $\Im \la \ge0$, $\frac{\phi(z)}{z-\lambda}\in K_\Theta$
if and only if $\phi(\la)=0$ (see Lemma~\ref{lem1}).
If $\Im \la <0$, then the inclusion $\frac{\phi(z)}{z-\lambda}\in K_\Theta$
is equivalent  to $\frac{\Theta(z)\overline{\phi(\overline z)}}{z-\overline \la}=
\frac{\tilde \phi(z)}{z-\overline \la} \in H^2(\BC^+)$, which
happens if and only if $\tilde \phi(\overline \la)=0$.
This proves statements (2) and (3).

It is easy to check statement (4) by applying induction in
$\ell$; we omit the details.

Now suppose that $T$ has an adjoint $T^*$. The above observations
imply that $\si_p(T^*)=\overline{Z}_\phi \cup Z_{\tilde \phi}$,
and that all eigenspaces of $T^*$ are one-dimensional. Now
let $\la\in Z_\phi$.
Then
$$
\begin{aligned}
\langle f, T^* k_\lambda \rangle
& = \langle  Tf,  k_\lambda \rangle
  =
  2\pi i \big( zf(z) - c \phi(z) \big)\big|_{z=\lambda}
& = 2\pi i \lambda f(\lambda) = \lambda \langle f, k_\lambda
\rangle
\end{aligned}
$$
for any $f\in \cD(T)$. Since $\cD(T)$ is dense in $K_\Theta$, one has
$T^*k_\la=\bar \la k_\la$.
Analogously, using the equality
$\langle f, \tilde k_\lambda \rangle = - 2\pi i \overline{\tilde f(\lambda)}$,
it is easy to show that
$T^* \tilde k_\la= \la \tilde k_\la$ for $\lambda \in Z_{\tilde \phi}$.
This gives statement (5).
\end{proof}

\begin{remarks}
{\rm
1. If $\A $ is cyclic and compact, then
$\LL $ also is compact, and
$\phi$ is meromorphic in $\BC\setminus \{0\}$.
Items (3)--(5) of the above Lemma apply to
any $\la\ne 0$. It follows, in particular, that
in this case, $\LL $ is a Volterra operator (that is,
$\si(\LL )=\{0\}$) if and only if $\phi(\la)\ne 0$
and $\tilde\phi(\la)\ne 0$
for all $\la \ne 0$ with
$\Im \la\ge 0$.

2. It is clear that the system $\{k_\lambda\}_{\lambda \in \si_p(T^*)}$
is (up to normalization) biorthogonal to the system
$\{h_\la\}_{\la \in \si_p(T)}$.

3. A statement analogous to Lemma~\ref{lem_eigs}
holds for general model spaces (not necessarily associated with meromorphic
inner functions). E.g., if $\lambda\in \BC\setminus\RR$
or $\lambda\in \RR$ and $\Theta$ is analytic in a neighborhood of $\lambda$,
we have $\lambda \in \sigma(T)$ if and only if
$\lambda \in \sigma_p(T)$ and if and only if $\phi(\lambda) = 0$.
In this case, $h_\lambda$ is an eigenfunction of $T$
while $k_\lambda$ is an eigenfunction of $T^*$ (if $T^*$ is correctly defined).}
\end{remarks}

We finish this section with the proof of Theorem~\ref{mincomp}.

\begin{proof}[Proof of Theorem~\ref{mincomp}]
(1) Since the system $\{k_\lambda\}_{\lambda\in \Lambda}$ is complete and minimal,
for a fixed
$\lambda_0 \in \Lambda$ there exists a unique (up to a constant factor)
function $g\in K_\Theta$ such that $g (\lambda)=0$, $\lambda\in
\Lambda\setminus\{\lambda_0\}$. Put $\phi = (z-\lambda_0)g$. Then $\phi$
vanishes exactly on the set $\Lambda$. Clearly, $\phi \notin H^2$ and
$\frac{\phi(z)}{z-\lambda} \in K_\Theta$ for any $\lambda\in \Lambda$.
Thus, by the converse statement in Theorem~\ref{rank-one-model},
$\Theta$ and $\phi$ correspond to some singular rank one perturbation.

(2) Suppose that $\zeta-\Theta \notin H^2$ for any $\zeta$ with $|\zeta|=1$.
We conclude that $b\notin \Elltwo(\mu)$
(otherwise, it would follow from \eqref{singul1} and
\eqref{the} that for some $\zeta$ of modulus one,
$y^{-1}(\zeta+\Theta(iy))(\zeta-\Theta(iy))^{-1}\to 0$ as
$y\to \infty$, which by \eqref{mass} gives that $\zeta-\Theta \in H^2$).
Thus, condition (${\rm A}^*$) is satisfied and $\LL^*$ is correctly defined.

(3) First we claim that in this case $\phi$ is outer in $\BC^+$.
Indeed, by the assumption, $\phi$ has no
zeros in $\BC^+$. If there were a representation
$\phi(z) = \psi(z) e^{i\ga z}$, where $\psi\in (z+i)H^2$ and $\ga > 0$, then the function
$\psi(z)\,\frac{e^{i\ga z}-1}{z}$ would belong to $K_\Theta$ and vanish on $\Lambda$, a contradiction.
Finally,
$\tilde\phi(z)=\Theta(z) \overline{\phi(\bar z)}$
also is outer in $\BC^+$, and so $\Theta = \xi^2\phi/\bar\phi$
on $\BR$ for a unimodular constant $\xi$.
Indeed, if $\Theta\bar\phi = I\phi$ for a (meromorphic) inner function $I$,
then $Ig$ also is in $K_\Theta$. If $I$ has a Blaschke factor
$\frac{z-z_0}{z-\overline z_0}$, then the function
$gI \frac{z-\lambda_0}{z-z_0}$ belongs to $K_\Theta$
and vanishes on $\Lambda$, a contradiction. The case when
$I(z) = e^{i\ga z}$ can be excluded as above.
\end{proof}


\subsection{Hermite--Biehler and Cartwright classes}
\label{herm--biehl}

An entire function $E$ is said to be in the {\it Hermite--Biehler class}
(which we denote by $HB$) if
$$
|E(z)|>|E(\overline z)|, \qquad z\in \BC^+.
$$
We also always assume that $E \ne 0$ on $\RR$.
For a detailed study of the Hermite--Biehler class see
~\cite[Chapter VII]{levin}.
Put $E^*(z) = \overline{E(\overline z)}$. If $E\in HB$, then
$\Theta = E^*/E$
is an inner function which is meromorphic in the whole plane $\BC$;
moreover, any meromorphic inner function
can be obtained in this
way for some $E\in HB$ (see, e.g.,~\cite[Lemma~2.1]{hm1}).

Given $E\in HB$, we can always write it as $E=\F-i\G$, where
$$
\F=\frac{E+E^*}{2},\qquad \G=\frac{E^*-E}{2i}.
$$
Then $\F$, $\G$ are real on the real axis
and all their zeros are real and simple.
Moreover, if $\Theta =E^*/E$, then $2\F = (1+\Theta)E$.

Any function $E\in HB$ generates the \textit{de Branges space}
${\mathcal H} (E)$, which consists of all entire functions $f$
such that $f/E$ and $f^*/E$ belong
to the Hardy space $H^2$, and $\|f\|_E = \|f/E\|_{\Elltwo({\RR})}$
(for the de Branges theory see~\cite{br}).
It is easy to see that the mapping $ f\mapsto f/E $
is a unitary operator from ${\mathcal H}(E)$ onto  $K_\Theta$
with $\Theta=E^*/E$ (see, e.g.,~\cite[Theorem~2.10]{hm1}).

The reproducing kernel of the de Branges space ${\mathcal H} (E)$
corresponding to the point $w\in\BC$ is given by
\beqn
\label{rk-deBr}
 K_w(z)=\frac{\overline{E(w)} E(z) - \overline{E^*(w)} E^*(z)}
{2\pi i(\overline w-z)} =
\frac{\overline{\F(w)} \G(z) -\overline{\G(w)}\F(z)}{\pi(z-\overline w)}.
\neqn

An entire function $F$ is said to be of {\it Cartwright class}
if it is of finite exponential type and
$$
\int_{\RR} \frac{\log^+|F(x)|}{1+x^2}dx <\infty.
$$
For the theory of the Cartwright class we refer to~\cite{hj, ko1}.
It is well-known that zeros $z_n$ of a Cartwright class function
$F$ have a certain symmetry: in particular,
\beqn
\label{cartw}
F(z) = K z^m e^{icz}\, v.p. \prod_n \Big(1- \frac{z}{z_n}\Big)
\defin
K z^m e^{icz}
\lim_{R\to \infty} \prod_{|z_n|\le R} \Big(1- \frac{z}{z_n}\Big),
\neqn
where the infinite product converges in the
``principal value'' sense, $c\in \RR$ and $K\in \BC$ are some constants,
$m\in\mathbb{Z}_+$.

A function $f$ analytic in $\BC^+$ is said to be of bounded type,
if $f=g/h$ for some functions $g$, $h\in H^\infty(\BC^+)$. If,
moreover, $h$ can be taken to be outer, we say that $f$ is in
\textit{the Smirnov class in $\BC^+$}. It is well known that if
$f$ is analytic in $\BC^+$ and $\ima f>0$, then $f$ is in the
Smirnov class~\cite[Part 2, Ch. 1, Sect. 5]{hj}. In particular, if
$t_n\in {\RR}$, $u_n>0$ and $\sum_n u_n<\infty$, then  the
function $\sum_n \frac{u_n}{t_n-z}$ is in the Smirnov class in
$\BC^+$. Consequently, $\sum_n \frac{v_n}{t_n-z}$ is in the
Smirnov class in $\BC^+$ for any $\{v_n\}\in \ell^1$.

Given a nonnegative function $m$ on $\BR$ such that
$\log m\in L ^1\big(\text{d}x/(x^2+1)\big)$,
there is a unique outer function $\cO$ of Smirnov class,
whose modulus is equal to $m$ a.e. on $\BR$
and which satisfies $\cO(i)>0$ (see~\cite[Part II, Chapter 3]{hj}).
This function will be denoted by $\cO_m$.

The following theorem due to M.G. Krein (see, e.g.,
~\cite[Part II, Chapter 1]{hj}) will be useful:
{\it If an entire function $F$ is of bounded type both in $\BC^+$
and in $\BC^-$, then $F$ is of Cartwright class.
If, moreover, $F$ is in the Smirnov class
both in $\BC^+$ and in $\BC^-$, then $F$
is a Cartwright class function of zero exponential type.}
We refer to~\cite[Section IV.8]{Gohb_Krein} for applications of these results to
the spectral theory of non-dissipative operators.

Finally, we remark that applications of de Branges spaces to the spectral theory
of discrete selfadjoint operators and their perturbations are numerous and
well-known, see for instance,~\cite{mak-polt} and also
papers~\cite{Martin} and~\cite{Silva-Toloza}, that are
closely related to our approach.
\bigskip



\section{Positive results on completeness of $\LL $ and $\LL^*$ for the case of
singular rank one perturbations. Proofs of Theorems \ref{positive}
and \ref{frombb}}

Let $\LL = \LL (\A , a,b,\deab)$ be a singular rank one perturbation
of the multiplication operator $\A $ defined in Section \ref{model}.
We assume that $\mu=\sum_n \mu_n\delta_{t_n}$ is a discrete measure.
Then Theorem~\ref{rank-one-model} provides a model operator $T$,
given in terms of meromorphic functions $\Theta$  and $\phi$
satisfying condition (\ref{main0}).

We will assume that $\phi$ and $\tilde \phi$ have only simple zeros. In the case
of multiple zeros, the root
vectors, which were calculated in Lemma~\ref{lem_eigs},
should be taken into account.
It can be checked that the same results hold in this case too and that, basically,
the same arguments work.

The spectrum of $\A $ is
$$
\sigma(\A ) =  \{t_n\} = \{x\in \RR: \, \Theta(x) = -1 \}.
$$
We will use the notation $a_n=a(t_n)$, $b_n=b(t_n)$.
We recall that the
measure $\nu$, defined in  \eqref{defnu}, is the Clark measure $\sigma_{-1}$
for $\Theta$. In our case,
$$
\nu = \sum_n  \nu_n \delta_{t_n},
\quad
\text{where }
\nu_n=\mu_n|b_n|^2
$$
(recall that $b_n \ne 0$ for any $n$),
and $\Theta'(t_n)= -2i/\nu_n$.
Finally, by \eqref{phi-a-b}, $\phi(t_n) =i a_n/b_n$.
Note that if $a_n=0$, then $\phi(t_n)=0$.


\subsection{An abstract criterion for completeness of $T^*$}
\label{abstr}

Since $\phi$ is meromorphic and $\frac {\phi}{z+i}$
is in $H^2$, its Nevanlinna factorization in $\BC^+$ has the form
$$
\phi(z)=e^{i\al z}\cdot B(z)\cdot \cO_{|\phi|}(z),
$$
where $\cO_{|\phi|}$ is the outer part of $\phi$, $B$ is a Blaschke product
in $\BC^+$ and $\al=\al(\phi)\ge 0$.
The following proposition gives a criterion for the completeness
of the eigenfunctions of $T^*$. This is a standard tool
for the study of completeness for systems of reproducing kernels
and many results of this type are known.
For the case of the Paley--Wiener spaces it goes back to
Levin~\cite[Appendix III, Theorem~6]{levin}. An explicit statement
for the de Branges spaces is given in~\cite[Theorem~2.2]{Gubr-Tar2010}.
We include a short proof of this statement.

\begin{proposition}
\label{completeness}
Let $\Theta$ and $\phi$ correspond to some rank one singular
perturbation of a cyclic operator $\A $ with discrete spectrum. More precisely, let
$\Theta$ be a meromorphic inner function and  let $\phi \notin H^2$ satisfy
\eqref{main0}.

$(1)$
If $\al(\phi)>0$ or $\al(\tilde \phi)>0$, then the system
$\{k_\lambda\}_{\lambda \in Z_\phi}
\cup \{\tilde k_\lambda\}_{\lambda \in Z_{\tilde \phi}}$
is not complete in $K_\Theta$.

$(2)$ Suppose that $\al(\phi)= \al(\tilde \phi) = 0$
and both $\phi$ and $\tilde \phi$ have only simple zeros. Then
the following two statements are equivalent:

\begin{enumerate}
\item[(i)] the system
$\{k_\lambda\}_{\lambda \in Z_\phi} \cup \{\tilde k_\lambda\}_{\lambda \in Z_{\tilde \phi}}$
is not complete in $K_\Theta$. $($Thus, operator $T^*$, whenever
it is correctly defined, is not complete.$)$

\item[(ii)]
\label{item:ii}
there exists a nonzero entire function $F$
of zero exponential type, whose zeros lie in $\BC_-\cup \RR$, such that $F\phi\in H^2$.
\end{enumerate}
Moreover, the function $F$ in {\rm (ii)} is always of Cartwright class.
\end{proposition}

\begin{proof}[Proof]
(1) Suppose that $\alpha = \al(\phi)>0$. Since $\phi$
is meromorphic, it has the form
$\phi(z)=e^{i \alpha z}\phi_1(z)$, where
$\frac{\phi_1}{z+i}$ is in $H^2$. Then the function
$\phi_1(z)\frac{e^{i\alpha z}-1}{z}$ belongs to $K_\Theta$, is non-zero and
is orthogonal to the system
$\{k_\lambda\}_{\lambda \in Z_\phi} \cup \{\tilde k_\lambda\}_{\lambda \in Z_{\tilde \phi}}$,
since the zeros sets of $\phi$ and $\phi_1$
(respectively, of $\tilde \phi$ and $\tilde \phi_1$) coincide.
Analogously, if $\tilde \al = \al(\tilde \phi)>0$, then
$e^{-i\tilde \alpha z}\frac{\tilde \phi}{z+i} \in H^2$ and, thus,
the function $\phi(z)\frac{e^{i \tilde \alpha z}-1}{z}$ is in $K_\Theta$
and is orthogonal to the system
$\{k_\lambda\}_{\lambda \in Z_\phi} \cup \{\tilde k_\lambda\}_{\lambda \in Z_{\tilde \phi}}$.

(2)  For the rest of the proof, we suppose that $\al(\phi)=\al(\tilde \phi) =0$.
Assume first that the system
$\{k_\lambda\}_{\lambda \in Z_\phi}
\cup \{\tilde k_\lambda\}_{\lambda \in Z_{\tilde \phi}}$
is not complete and so there is a nonzero $g\in K_\Theta$
such that
\beqn
\label{orthog}
\begin{aligned}
2\pi i g(\lambda) & = \langle g, k_\la \rangle =0, \qquad \la \in Z_\phi \\
2\pi i \tilde g(\lambda) & = \langle \tilde k_\la, g \rangle =0
\qquad  \la \in Z_{\tilde \phi}.
\end{aligned}
\neqn
Since $\phi$ has only simple zeros, we may write
$g = \phi\, G$ for some function $G$ which is meromorphic in
${\BC}$, analytic in ${\BC}^+$ and on ${\RR}$, and
of bounded type in ${\BC}^+$. Consider the function
$$
\tilde g(z) = \Theta(z)\, \overline{g(\overline z)} =
\Theta(z)\, \overline{\phi (\overline z)} \,
\overline{G(\overline z)}  =\tilde \phi(z) G^*(z),
$$
where $G^*(z) = \overline{G(\overline z)}$. Since $\tilde g$ vanishes at $Z_{\tilde\phi}$ we conclude that
$G^*$ has no poles in $\BC^+$ and thus $G$ is an entire function. Moreover,
since $G= g/\phi$ and $G^*= \tilde g/\tilde\phi$ in $\BC^+$ and
$\al(\phi)=\al(\tilde \phi) =0$, the functions $G$ and $G^*$ are in Smirnov class in $\BC^+$
and, by Krein's theorem, $G$ is of zero exponential type and of Cartwright class.
We have $\phi G = g\in H^2 $ and $\tilde\phi G^* = \tilde g\in H^2$.

Finally, to obtain from $G$ the function $F$ with zeros in $\BC^-\cup\RR$ note
that the zeros of $G$ in $\BC^+$ satisfy the Blaschke condition. Let $B$ be the Blaschke
product constructed over $Z_G\cap \BC^+$ (counting multiplicities). Then $F=G/B$ is an entire function
with zeros in $\BC^-\cup\RR$ which satisfies all the required properties (note that $\phi F = \phi\, G/B\in H^2$).

To prove the converse, assume that there exists $F$ as in (ii). Put
$g=F\phi$. By the assumption, $g \in H^2$, whereas
$\tilde g(z) = \tilde \phi(z)F^*(z)$. By the conditions on $F$, the ratio $F^*/F$ is  Blaschke product,
while $\tilde \phi/\phi$ is a ratio of two Blaschke products due to the condition
$\al(\phi)=\al(\tilde \phi) =0$. Hence,  $\tilde \phi(z)F^*(z)$ is in $H^2(\BC^+)$
as soon as the function $\phi F$ belongs to this space.
Thus $\tilde g$ is in $H^2$, whence $g\in K_\Theta$.
Since $g$ vanishes on the set $Z_\phi$ and $\tilde g$ on the set $Z_{\tilde \phi}$,
the function $g$ is orthogonal to
$\{k_\lambda\}_{\lambda \in Z_\phi}
\cup \{\tilde k_\lambda\}_{\lambda \in Z_{\tilde \phi}}$ by (\ref{orthog}).
\end{proof}

\begin{remarks}
{\rm 1. It is an obvious, but useful observation
that the function $F$ in the statement~2 of Proposition~\ref{completeness}
can be always chosen to have no zeros in $\BC^+$ (or in $\BC^-$).
Moreover, it is not difficult to show that in the case of real type
perturbations (i.e., $\phi = \tilde \phi$)
the function $F$ may be chosen to have {\it only real} zeros.
\smallskip

2. Condition (ii) in Proposition~\ref{completeness} is equivalent
to the condition that there exists a Cartwright class function $F_1$ of zero
exponential type, whose zeros lie in $\BC_-\cup \RR$,
and such that $F_1 \tilde \phi\in H^2$. Indeed, the inclusion $g=F\phi \in K_\Theta$
is equivalent to the inclusion $\tilde g= F^*\tilde\phi \in K_\Theta$.
Dividing by a Blaschke product we can move all zeros
of $F$ to the upper half-plane and then $F^*$
will have all zeros in $\BC^-\cup\RR$.
\smallskip

3. The above proof can be in an obvious way extended to the case when
$\phi$ or $\tilde \phi$
has multiple zeros. The statement remains unchanged: the root vectors of $T^*$
(which are of the form
$\frac{\partial^j}{\partial\overline\lambda^j}k_\lambda$ or
$\frac{\partial^j}{\partial\lambda^j}\tilde k_\lambda$ for some $j$) are not complete in $K_\Theta$
if and only if there exists a nonzero entire function $F$
of zero exponential type, whose zeros lie in $\BC_-\cup \RR$, such that $F\phi\in H^2$. }
\end{remarks}

It would be interesting to compare the above theorem with~\cite[Theorem~5]{Veselov2}, where
a completeness criterion is given in terms of Naboko's model of a nondissipative
perturbation of a selfadjoint operator.


\subsection{Sufficient conditions for completeness.}
The results stated in the Introduction show
that completeness of $T$ and $T^*$ (equivalently, $\LL $ and $\LL^*$)
are essentially different
things. In this subsection we prove several results which show that, under
some additional restrictions, both $T$ and $T^*$ are complete or
the completeness of eigenfunctions of $T^*$
(reproducing kernels $k_\lambda$) implies the completeness of
eigenfunctions of $T$ (functions $h_\lambda$).

We will need the following lemma.

\begin{lemma}
\label{maincor-1}
Suppose that the meromorphic function $\phi$ associated with
a perturbation $\LL (\A , a, b, \deab)$ satisfies the conditions
\begin{align}
\label{smooth-sing}
& \sum_n \frac{|a_n b_n| \mu_n}{|t_n|}<\infty, \\
\label{import}
& \sum_n \frac{a_n\overline b_n \mu_n}{t_n} \ne \deab.
\end{align}
%
%
Then $|\phi(iy)| \ge c y^{-1}$, $|\tilde \phi(iy)| \ge c y^{-1}$,
$y\to \infty$, for some constant $c>0$.
\end{lemma}

\begin{proof}[Proof]
By (0.11), we have
$$
2\phi(z) = \big(1+\Theta(z)\big)
\Bigg(\deab -\sum_n \frac{a_n\overline b_n \mu_n}{t_n}
  + \sum_n \frac{a_n\overline b_n\mu_n}{t_n-z}\bigg).
$$
Since
$$
\sum_n \frac{a_n\overline b_n\mu_n}{t_n-iy} =
\sum_n \frac{a_n\overline b_n\mu_n}{t_n}\cdot
\frac{t_n}{t_n-iy} \to 0, \quad y\to \infty,
$$
\eqref{import} gives that $|\phi(iy)| \ge c_1 |1+\Theta (iy)|>0$, $y\to\infty$.
It is easy to see that for any
meromorphic
inner function $\Theta$
there exists a constant $c_2>0$ such that
for any $\gamma \in \BC$, $|\gamma|=1$, and $z\in \BC^+$
\begin{equation}
\label{estim}
|\gamma +\Theta(z)| \ge 1- |\Theta(z)| \ge \frac{c_2 \ima z}{|z|^2 +1}, \qquad \Im z>1
\end{equation}
(for the proof note that if $z_0$ is a zero of $\Theta$, then
$|\Theta(z)| \le \big| \frac{z-z_0}{z-\overline z_0}\big|$
and if $\Theta$ has no zeros, then $\Theta(z)=\zeta e^{i\al z}$
for some $\al>0$, $|\zeta|=1$).
Hence, $|\phi(iy)| \ge c y^{-1}$, $y\to\infty$, where $c>0$.
The same is true  for $\tilde \phi$,  because it has the same representation
as $\phi$ with conjugate coefficients.
\end{proof}

Next we give sufficient conditions for the joint completeness
of $\LL $ and $\LL^*$. This theorem largely covers the completeness
results of Gubreev and Tarasenko (see the discussion after the proof).

\begin{theorem}
\label{3to1}
Let meromorphic $\Theta$ and $\phi$ correspond to some
singular rank one perturbation $\LL =\LL (\A , a, b, \deab)$
such that the data $(a, b, \deab)$ satisfy both $(\admone)$ and
$(\admonest)$, and let $\al(\phi)= \al(\tilde \phi) = 0$.
\begin{enumerate}

\item
Assume that at least one of the following two conditions holds:
\begin{equation}
\label{maincor1}
\limsup_{y\to\infty} y^N |\phi(iy)| >0
\end{equation}
for some $N>0$, or
\begin{equation}
\label{maincor10}
\int_{\RR} \frac{dt}{|\phi(t+i\eta)|^\taauu (1+|t|)^{N}}<\infty
\end{equation}
for some $N>0$, $\taauu>0$ and $\eta\ge 0$.
Then both $\LL $ and $\LL^*$ are complete.

\item $($Generalized weak perturbations$)$ If \eqref{smooth-sing}
and \eqref{import} hold, then the operator $\LL^*$ is correctly
defined, and both $\LL $ and $\LL^*$ are complete.
\end{enumerate}
\end{theorem}

Note that Statement (2) is a particular case of
Proposition~\ref{prp-Mats-type-sing}. A combination of
Lemma~\ref{maincor-1} with Statement (1) of Theorem \ref{3to1}
gives a direct proof of (2), which does not use Macaev's theorem.

\begin{proof}[Proof of Theorem~\ref{3to1}]
Statement (2) is a special case of (1), because by Lemma~\ref{maincor-1} in this case (\ref{maincor1}) is satisfied.
The proof of (1) will consist of several steps.
\medskip

{\it Step 1: completeness of $\LL^*$.}
Let $T$ be the model operator corresponding to $\LL $.
If the system of eigenfunctions of $T^*$
is not complete, then, by Proposition~\ref{completeness},
there exists a nonzero Cartwright class entire function $F$
of zero exponential type with
zeros in $\BC^-\cup \RR$ such that $F \phi \in H^2 $. Hence,
$$
|F(iy)\phi(iy)| \le C y^{-1/2}, \qquad y>0.
$$
Thus, if \eqref{maincor1} holds, then
$\liminf_{y\to \infty} y^{-N}|F(iy)| < \infty$. Since
$$
F(z) = K \, v.p. \prod_n \Big(1- \frac{z}{z_n}\Big), \qquad
|F(iy)|^2 = |K|^2 \prod_n \frac{x_n^2 + (y+y_n)^2}{x_n^2+y_n^2},
$$
with $z_n = x_n - iy_n$, $y_n \ge 0$, we conclude that $F$ is a
polynomial. Then $\phi\in H^2$, which contradicts Theorem~\ref{rank-one-model}.

If  \eqref{maincor10} holds, then it follows from
the H\"older inequality that
\begin{equation}
\label{sle4}
\int_\RR |F(t+i\eta)|^\gamma (1+|t|)^{-M}dt <\infty
\end{equation}
for some $\gamma\ge 0$ and $M>0$. Since $F$ is of zero type, we conclude again
that $F$ is a polynomial, a contradiction.
\medskip

{\it Step 2: completeness of $\LL $ in the case $a_n\ne 0$ for any $n$.}
By Proposition~\ref{Pr-adjoints-n},
$\LL^* = \LL _1$, where we set $\LL _1 = \LL (\A , b, a, \overline{\deab})$.
%
%
%
%
%
Then $\LL =\LL _1^*$ is
complete if and only if $T_1^*$ is complete. Note that in the
corresponding pair $(\Theta_1, \phi_1)$ we have $\phi_1 =
(1+\Theta_1) \nuu_1/2$, where $\nuu_1$ is defined by
\eqref{singul2} with the data $(b,a, \overline{\deab})$ in place
of $(a,b, \deab)$, and thus differs from $\nuu$ only by
conjugation of the coefficients. Thus, $\tilde\phi_1 =
(1+\Theta_1)\beta/2$. Now it follows from estimate (\ref{estim})
that if $\phi$ satisfies either \eqref{maincor1}  or
\eqref{maincor10} with $\eta>0$, then
$$
\limsup_{y\to\infty} y^N |\tilde \phi_1(iy)| >0
\qquad\text{or}\qquad
\int_{\RR} \frac{dt}{|\tilde \phi_1(t+i\eta)|^\taauu (1+|t|)^{N+2}}<\infty
$$
for some $\taauu>0$. If the eigenfunctions of $T_1^*$ are
not complete, then, by Proposition~\ref{completeness}
(see the remarks after it), there is
a Cartwright class entire function $F$
of zero exponential type with
zeros in $\BC^-\cup \RR$ such that $F \tilde \phi_1 \in H^2 $
and we conclude, as in Step 1, that $F$ is a polynomial,
which is a contradiction because $\tilde \phi_1\notin H^2$.

The case $\eta=0$ is a bit more tricky. In this case we have
$$
\int_{\RR} \frac{dt}{|\phi(t)|^\taauu (1+|t|)^{N}} =
\int_{\RR} \bigg|\frac{1+\Theta_1(t)}{1+\Theta(t)}\bigg|^\taauu
\cdot \frac{dt}{|\tilde \phi_1(t)|^\taauu (1+|t|)^{N}} <\infty,
$$
and we conclude that
$$
\int_{\RR} \bigg|\frac{1+\Theta(t)}{1+\Theta_1(t)}\bigg|^\gamma
\frac{|F(t)|^\gamma}{(1+|t|)^M} \,dt <\infty
$$
for some $\gamma>0$ and $M>0$. Since $F$ is a Cartwright
class function of zero type with zeros in $\BC^-\cup \RR$, we conclude that $F$
is an outer function in $\BC^+$ and so is
$\frac{1+\Theta}{1+\Theta_1} F$. Hence, we conclude that
$\frac{1}{(z+i)^M}\,\big(\frac{1+\Theta}{1+\Theta_1} \big)^\gamma F^\gamma
\in H^1(\BC^+)$.
Thus, we have also
\begin{equation}
\label{sle1}
\int_{\RR} \bigg|\frac{1+\Theta(t+i)}{1+\Theta_1(t+i)}\bigg|^\gamma
\frac{|F(t+i)|^\gamma}{(1+|t|)^M} \,dt <\infty.
\end{equation}
Applying again (\ref{estim}) we conclude that taking a larger $M$
we may omit the factor $\big|\frac{1+\Theta(t+i)}{1+\Theta_1(t+i)}\big|^\gamma$
in (\ref{sle1})
and get (\ref{sle4}). Thus, $F$ is a polynomial, which once again gives a contradiction.
\medskip

If some of the coefficients $a_n$ are zero (in other words, if $a$ is not cyclic for
$\A^{-1}$), our model does not formally
apply to $\LL_1$. We reduce the problem to the case where
$a$ is cyclic by considering the representation of the function
$\phi$ with respect to a different Clark measure for the same
space $K_\Theta$.
\medskip

{\it Step 3: reduction to the case where $a$ is cyclic.}
We need to show that in conditions of the theorem the system
$\{h_\lambda\}_{\lambda \in Z_\phi\cup \overline Z_{\tilde \phi}}$
is complete in $K_\Theta$. Let $\zeta \in \BC$, $|\zeta| = 1$, $\zeta \ne -1$,
and let $\sigma_\zeta$ be the corresponding Clark measure defined by
(2.5). We may choose $\zeta$ such that $\zeta - \Theta\notin H^2$
and $\Theta(0) \ne \zeta$ (thus, $p(\zeta) = 0$).
Since $\Theta$ is meromorophic, $\sigma_\zeta$ is  a discrete measure,
$\sigma_\zeta = \sum_n \nu_n^\circ \delta_{t^\circ_n}$, where
$\{t^\circ_n\} = \{t: \Theta(t) = \zeta\}$, and we have, for some $q^\circ\in \RR$,
$$
\frac{\zeta+\Theta(z)}{\zeta-\Theta(z)}= iq^\circ
+\frac{1}{i}
\sum_n \Big(\frac{1}{t^\circ_n-z} -\frac{1}{t^\circ_n} \Big) \,
\nu^\circ_n,  \qquad z\in\mathbb{C^+}.
$$
Since the function $\phi$ satisfies \eqref{main0}, we can represent it, as in the
proof of Theorem~\ref{rank-one-model},
$$
2\phi(z) = \big(1 -\zeta \Theta(z)\big)
\Bigg(\deab^\circ + \sum_n \bigg(\frac{1}{t^\circ_n-z} - \frac{1}{t^\circ_n}\bigg)
c^\circ_n \nu^\circ_n\bigg),
$$
for some $\deab^\circ$ and $c_n^\circ$ satisfying
$\sum_n (t^\circ_n)^{-2} |c^\circ_n|^2 \nu^\circ_n <\infty$.
Let us write
$$
\nu^\circ_n = |b^\circ_n|^2 \mu_n, \qquad
c^\circ_n = \frac{a_n^\circ}{b_n^\circ} \mu_n,
$$
such a representation exists and is unique up to the
choice of the arguments of $b^\circ_n$.
Recall  that by the properties of the Clark measures, $\sum_n  (t^\circ_n)^{-2}
\nu^\circ_n<\infty$.
Then if we put $a^\circ=(a^\circ_n)$, $b^\circ = (b^\circ_n)$, then
$a^\circ/x, \, b^\circ/x \in \Elltwo(\mu^\circ)$, where
$\mu^\circ = \sum_n \mu_n \delta_{t^\circ_n}$.

Note that $\phi(t_n^\circ) = i a_n^\circ/b_n^\circ$. We may choose $\zeta$
so that $\phi(t_n^\circ) \ne 0$ for all $n$.
Indeed, it suffices to choose $\zeta$ to be different from
all values of $\Theta$ at the zeros of the meromorphic function $\phi$
(which form a countable set).
Thus, with this choice of $\zeta$ we have $a_n^\circ \ne 0$ for all $n$.

At the same time, since $\phi$ satisfies \eqref{main0},
it follows from Proposition~\ref{param} (applied to the inner function
$\Theta^\circ = -\bar \zeta \Theta$ in place
of $\Theta$) that either $a^\circ \notin \Elltwo(\mu^\circ)$ or
$a^\circ \in \Elltwo(\mu^\circ)$ and
\begin{equation}
\label{admcirc}
\deab^\circ \ne \sum_n\frac{a_n^\circ \overline{b_n^\circ} \mu_n}{t_n^\circ}.
\end{equation}
Thus the data $(a^\circ, b^\circ, \deab^\circ)$ satisfy the condition $(\admone)$
(with $t_n^\circ$ in place of $t_n$), whence the
rank one perturbation $\LL ^\circ = \LL (\A^\circ, a^\circ, b^\circ, \deab^\circ)$
of the operator of the multiplication by $x$ in $\Elltwo(\mu^\circ)$ is well defined.
The pair $(\Theta^\circ, \phi)$ will be the model pair for $\LL ^\circ$.
Assume for the moment that the data $(a^\circ, b^\circ, \deab^\circ)$ also
satisfy the condition $(\admonest)$ and so $(\LL ^\circ)^*$ is well defined
(we prove it in Step 4 below).
Since $\phi$ satisfies (\ref{maincor1})
and (\ref{maincor10}), while now both $a_n^\circ \ne 0$ and $b_n^\circ \ne 0$
for all $n$, it follows from Step 2 that $(\LL ^\circ)^*$ is complete, and so
the system
$\{h_\lambda\}_{\lambda \in Z_\phi\cup \overline Z_{\tilde \phi}}$ is
complete in $K_{\Theta^\circ}$. But for $\Theta^\circ = -\bar \zeta \Theta$
we clearly have
$K_{\Theta^\circ}= K_\Theta$,
and so $\LL^\circ$ is unitarily equivalent to $\LL$.

{\it Step 4: the data  $(a^\circ, b^\circ, \deab^\circ)$  satisfy condition $(\admonest)$.}
To complete the proof, it
 remains to verify $(\admonest)$ for the data $(a^\circ, b^\circ, \deab^\circ)$.
If $b^\circ \notin \Elltwo(\mu^\circ)$, then there is nothing to prove. Assume
that $b^\circ \in \Elltwo(\mu^\circ)$, which means that $\sigma_\zeta(\RR) = \sum_n \nu_n^\circ<\infty$.
It follows that there exists $\xi$, $|\xi| = 1$, such that
$\xi - \Theta(iy) = O(y^{-1})$, $y\to\infty$, whence $\xi - \Theta \in K_\Theta$
by \eqref{mass}.
By the construction, $\xi \ne -1$ and $\xi \ne \zeta$.  Then it follows that
$\sum_n |b_n|^2 \mu_n = \sigma_{-1}(\RR) <\infty$
(indeed, $\|\xi- \Theta\|_{\Elltwo(\sigma_{-1})} = \|\xi+1\|_{\Elltwo(\sigma_{-1})}<\infty$).
Since the data $(a, b, \deab)$ satisfy $(\admonest)$ we have
$$
\deab \ne \sum_n\frac{a_n \overline{b_n} \mu_n}{t_n}.
$$
Comparing two expansions for $2\phi$,
$$
\big(1+\Theta(iy)\big)
\Bigg(\deab -\sum_n \frac{a_n\overline b_n \mu_n}{t_n}
  + \sum_n \frac{a_n\overline b_n\mu_n}{t_n-iy}\bigg) =
\big(1- \bar\zeta \Theta(iy)\big)
\Bigg(\deab^\circ -\sum_n \frac{a_n^\circ\overline{b_n^\circ}\mu_n}{t^\circ_n}
  + \sum_n \frac{a_n^\circ\overline{b_n^\circ}\mu_n}{t_n^\circ-iy}\bigg),
$$
as $y \to\infty$   and using the fact that $\Theta(iy)\to \xi$,
we conclude that (\ref{admcirc}) is satisfied (otherwise, the right-hand side would
tend to 0 as $y\to\infty$, while the left-hand side has a nonzero finite limit).
\end{proof}

\begin{remarks}
{\rm  1. Theorem~\ref{3to1} implies Theorems~2.3 and~2.4 from \cite{Gubr-Tar2010}.
In~\cite[Theorem~2.4]{Gubr-Tar2010}, completeness of $\LL $ was proved
under the following assumptions on the function $\phi$:
there exists a weight $w$ satisfying
the Muckenhoupt $(A_2)$ condition such that
$$
|\phi(t)|\le w(t), \quad t\in \RR, \qquad\text{and}\qquad
\int_\RR \frac{w(t)}{|\phi(t)|^2(1+t^2)}dt <\infty.
$$
Since $\int_\RR \frac{dt}{w(t)(1+t^2)}<\infty$ for any $(A_2)$-weight,
already the second inequality implies the inequality
$\int_\RR \frac{dt}{|\phi(t)|(1+t^2)}<\infty$.
The hypotheses of Theorem~2.3 in~\cite{Gubr-Tar2010} also imply
the same inequality. Thus, condition
(\ref{maincor10}) of our Theorem~\ref{3to1} is much weaker than the hypothesis of
Theorems~2.3 and~2.4 in ~\cite{Gubr-Tar2010}. It is not so surprising, because
the Muckenhoupt condition, which is intrinsic to the unconditional bases problem, seems to be too restrictive in the
completeness problems. Theorem~\ref{3to1} does not formally cover the case
when $(\admonest)$ is not satisfied, but it is anyway obvious that in this case
$\LL $ is not complete (indeed, we have $a\notin \Elltwo(\mu)$
and $\Ker \LL ^{-1} = \{0\}$, while $b\in \Ker (\LL ^{-1})^* \ne \{0\}$).

2. Essentially, the argument of Step 3
reduces to considering perturbations
%
%
%
%
%
%
$\A\mapsto \A+ \tau b\cdot b^* \mapsto \A + a\cdot b^*$ and to
choosing a parameter $\tau \in \RR$ such that $a$ is
cyclic for $\A+\tau bb^*$. To be more precise, one has to speak of
a singular perturbation $\LL(\A, b, b, \tau')$ instead of
$\A+ \tau b\cdot b^*$.
}
\end{remarks}


\subsection{Proofs of Theorems \ref{positive} and \ref{frombb}}
We conclude this section with the proofs of Theorems
\ref{positive} and \ref{frombb}.
The proof of Theorem~\ref{frombb} will be based on
the results of~\cite[Theorems 1.2 and 5.2]{bar-belov}.

\begin{proof}[Proof of Theorem~\ref{positive}]
Note that the condition
$\sum_{n\in \mathcal{N}} |t_n|^{-1}|a_nb_n|\mu_n = \infty$
implies that $a,b\notin \Elltwo(\mu)$,  and so $\LL =\LL(\A, a, b, \deab)$
and $\LL^*$ are well defined. Since $b_n\ne 0$ for any $n$ we can
consider the functional model for $\LL$.
If $\nuu$ is given by (\ref{singul2}), then
$$
\ima \nuu(iy) = \sum_n \frac{ya_n\overline b_n \mu_n}{t_n^2+y^2}.
$$
If $a_n\overline b_n\ge 0$ for all $n$, we have
$$
\lim_{y\to\infty} y\ima \nuu(iy)  = \sum_n a_n\overline b_n\mu_n \in (0, \infty],
$$
and therefore $|\nuu(iy)| \ge Cy^{-1}$, $y\to\infty$.
Hence, by (\ref{estim}), $|\phi(iy)|\ge C_1 y^{-2}$, $y\to\infty$,
and so the system $\{k_{\lambda}\}_{\lambda \in Z_\phi
\cup \overline{Z_\phi}}$ is complete in $K_\Theta$ by
Theorem~\ref{3to1} (note that $\tilde \phi = \phi$).

We show that in the general case $|\phi(iy)|\ge Cy^{-N}$
for some $N$. To simplify the notations put $u_n = a_n\overline b_n\mu_n$.
Since there is only finite number of negative
terms $u_n$ and an infinite number of positive terms,
there exists a nonnegative $k_0\in\BZ$ such that
$$
\sum_n u_n t_n^{2k} = 0, \quad 0\le k<k_0,
\quad \text{while}\enspace \sum_n u_n t_n^{2k_0}\ne 0
$$
(it may happen that $\sum_n u_n t_n^{2k_0}= \infty$).
Then
$$
\frac{\ima  \nuu(iy)}{y} =
\sum_n \frac{u_n}{t_n^2+y^2} = \frac{1}{y^2}\sum_n\frac{u_n}{1+
\frac{t_n^2}{y^2}} =
\frac{(-1)^{k_0}}{y^{2k_0}} \sum_n\frac{u_n t_n^{2k_0}}{ y^2 + t_n^2}
$$
(in the last equality, the formula
$$
\frac{1}{1+q} = \sum_{k=0}^{k_0-1} (-1)^k q^k +\frac{(-1)^{k_0} q^{k_0}}{1+q}
$$
has been applied to $q= t_n^2/y^2$). Since
$$
\lim_{y\to \infty} y^2 \sum_n \frac{u_n t_n^{2k_0}}{ y^2 + t_n^2}
= \sum_n u_n t_n^{2k_0}\in (-\infty, \infty]\sm \{0\},
$$
we have
$|\ima\nuu(iy)|\ge Cy^{-2k_0-1}>0$ for $y>y_0>0$.
By \eqref{estim}, similar estimate holds for $\phi$ and we
can apply Theorem~\ref{3to1}.
\end{proof}

\begin{proof}[Proof of Theorem~\ref{frombb}]
By Proposition~\ref{Pr-adjoints-n}, $\LL^* = \LL (\A , b,a, \overline{\deab})$.
Note that, by $(i)$ or $(ii)$ $a_n\ne 0$ for any $n$, and so we can
consider the model from Theorem~\ref{rank-one-model} for the
perturbation $\LL^*$. The functions $\Theta$ and $\phi$ are then defined by
formulas (\ref{the}) and (\ref{defphi}) where the roles of $a$ and $b$
are interchanged. Completeness of $\LL $ means that the corresponding
operator $T^*$ is complete, and we need to show that completeness of the system
$\{k_\lambda\}_{\lambda\in Z_\phi}$ implies the completeness of the system
$\{h_\lambda\}_{\lambda\in Z_\phi}$.

In the first case the measure $\nu = \sum_n \nu_n\delta_{t_n}$,
$\nu_n = |a_n|^2\mu_n $, is the Clark measure for $K_\Theta$.
Note that $a\notin \Elltwo(\mu)$ if and only if $\nu(\BR)=\infty$.
Now the statement follows immediately from~\cite[Theorem~1.2]{bar-belov}.
Since $\phi(t_n) = b_n/a_n$, the second condition
means that $|\phi(t_n)| \le C |t_n|^{N}$ for some $N>0$.
Application of~\cite[Theorem~5.2]{bar-belov}
completes the proof.
\end{proof}
\bigskip



\section{Noncomplete perturbations. Proof of Theorem~\ref{noncompleteness2}.}

This section is devoted to the proof of
Theorem~\ref{noncompleteness2}, which shows that
for any spectrum there exist singular rank one
perturbations such both $\LL $ and $\LL^*$ are correctly defined
(actually, both $a$ and $b$ are not in $\Elltwo(\mu)$)
and, moreover, $\LL^*$ is complete, while $\LL $ is not.
The proof is based on an extension of an idea
suggested by Yurii Belov in~\cite[Example 1.3]{bar-belov}.

It is much easier to see that
that for any cyclic selfadjoint operator $\A$ with
any spectrum $\{t_n\}$ such that $|t_n| \to \infty$, $|n| \to\infty$,
there exists a real type rank one singular perturbation $\LL$ of $\A$,
which is not complete. Indeed, it suffices to
take a singular perturbation $\LL$ such that
$a$, $b$ are real, $b\in \Elltwo(\mu) $, $a\notin \Elltwo(\mu)$
and $\int x^{-1}a(x)\overline b(x) d\mu(x)=-1$.
Then it is easy to see that the adjoint to the
bounded operator $\LL^{-1}$ has non-trivial kernel, which
implies the non-completeness of $\LL$. In this example,
$(\admonest)$ does not hold, so that
$\LL^*$ is not correctly defined.

We will use the following lemma on entire functions
(a close result can be found in \cite[\S 2]{nk-volterra}).
For an entire function $F$ with zeros $z_n$ (counting multiplicities)
we use the standard notation $n_F(r) = {\rm card}\,\{n:\,  |z_n| \le r\}$.

\begin{lemma}
\label{lem_entire}
For any sequence $\{t_n\}_{n\in\NN} \subset \RR$, $t_n\to\infty$,
there exists a function $n:\, [0,\infty) \to [0, \infty)$,
$n(r) = o(\log r)$ and $n(r)\to \infty$, $r\to \infty$,
with the following property:
there does not exist a non-constant entire function $U$ of order less than $1$
with real zeros such that
\smallskip
\begin{enumerate}
\item
$U(0)\ne 0$\textup; \
$\int_0^R  \frac{n_U(r)}{r}dr =
o\big(\int_0^R  \frac{n(r)}{r}dr\big)$, $R\to \infty$\textup;
\medskip

\item the sequence $\{U(t_n)\}_{n\in\NN}$ is bounded.
\end{enumerate}
\end{lemma}

\begin{proof}[Proof]
Choose a sequence $\{x_k\}_{k\in \NN}$, $x_k>0$, with the
following properties: $x_1=2$, $2x_k < x_{k+1}^{1/2}$ and each
interval $(2x_k, x_{k+1}^{1/2})$ contains at least one point from
the sequence $\{t_n\}$. Then $x_k\ge 2^{2^{k-1}}$.
Take $n$ to be the counting function
of the sequence $\{x_k\}$. Assume now that $U$
satisfies $(1)$, $(2)$ and let us prove that it has to be constant. By (1),
we have $n_U(r) \le n(er)\log
(er) =o(\log^2 r)$, $r\to\infty$. Also, for infinitely many $k$ we
have $[x_k, x_{k+1}] \cap Z_U =\emptyset$. Fix any such $k$ and
let $x\in (2x_k, x^{1/2}_{k+1})$. Denote by $\{u_n\}$ the set of
zeros of $U$. Without loss of generality assume that $U(0)=1$.
Then we have
$$
\log |U(x)| = \sum_n \log\Big|1-\frac{x}{u_n}\Big|.
$$
Let us estimate the summands whose contribution is negative, that is, those with
$u_n>x_{k+1}$ (note that $x>2x_k$ and so $x/u_n >2$ for $0 < u_n < x_k$):
$$
\begin{aligned}
\sum_{u_n> x_{k+1}} \log\Big|1-\frac{x}{u_n}\Big| & =
\int_{x_{k+1}}^\infty \log\Big(1-\frac{x}{r}\Big) \,dn_U(r)  \\
& =  -    \quad
  n_U(x_{k+1}) \log\Big(1-\frac{x}{x_{k+1}}\Big) -
x \int_{x_{k+1}}^\infty \frac{n_U(r)}{r(r-x)}\,dr = O(1),
\end{aligned}
$$
when $k$ is
sufficiently large, since $x/x_{k+1} \le x_{k+1}^{-1/2}$. Thus,
$$
\log |U(x)| = \sum_n \log\Big|1-\frac{x}{u_n}\Big| =
\sum_{n: \; u_n<x_k} \log\Big(\frac{x}{u_n}-1\Big) +O(1) \to + \infty
$$
as $k\to\infty$,
$[x_k, x_{k+1}] \cap Z_U =\emptyset$ and  $x\in
(2x_k, x^{1/2}_{k+1})$. This contradicts condition $(2)$, because
the interval $(2x_k, x^{1/2}_{k+1})$ contains points from the sequence
$\{t_n\}$.
\end{proof}

\begin{proof}[Proof of Theorem~\ref{noncompleteness2}]
In view of Theorem \ref{rank-one-model}, we may prove an analogous
statement in the model space.
Thus, we need to find an inner function $\Theta$ and an outer function $\phi$
such that $\Theta = \phi/\bar\phi$ on $\RR$,
$1+\Theta \notin H^2$,
$\phi \notin H^2$, $\frac{\phi}{z+i} \in H^2$,
$\{t\in \RR: \Theta(t) = -1\} = \{t_n\}$
and the system $\{h_\lambda\}_{\lambda\in Z_\phi}$
of eigenfunctions of $T$ is not complete in $K_\Theta$,
while the system $\{k_\lambda\}_{\lambda\in Z_\phi}$
of eigenfunctions of $T^*$ is complete in $K_\Theta$.
Note also that $\ker \LL$ (or $\ker \LL^*$) is nontrivial if and only
if $\kappa = 0$, which is equivalent to $\phi(0) = 0$ by
Proposition \ref{param}, (5).

The function $\Theta$ in question is meromorphic
and, in view of the discussion in Subsection \ref{herm--biehl}
we may reformulate the problem in terms of de Branges spaces. Thus, we
need to construct a space $\he$, $E=A-iB$, and
an entire function $g$ with the following properties:
\smallskip

(\textbf{i}) $A\notin \he$;

(\textbf{ii}) $g$ is real on $\RR$, has only simple real zeros,
$g(0) \ne 0$ and $Z_g \cap \{t_n\} = \emptyset$,
where $\{t_n\}$ are the zeros of $A$;

(\textbf{iii}) $g\notin \he$, but $\frac{g(z)}{z-\lambda} \in \he
$, $\lambda\in Z_g$;

(\textbf{iv}) the system
$\big\{\frac{g(z)}{z-\lambda}\big\}_{\lambda\in Z_g}$ is not
complete in $\he$ and its orthogonal complement is
infinite-dimensional;

(\textbf{v})
the system $\{K_\lambda\}_{\lambda \in Z_g}$ of
reproducing kernels \eqref{rk-deBr} is complete in $\he$.

\smallskip

If such $E$ and $g$ are constructed, it will remain to put $\Theta
= E^*/E$ and $\phi = g/E$;
they will satisfy the properties listed above.
\bigskip
\\
{\bf Strategy of the proof.}
Without loss of generality we may assume that $\mathbb{N} \subset
\mathcal{N}$ and $t_n\to\infty$, $n\to\infty$. Assume also that
$|t_n|\ge 1$, $n\in \mathcal{N}$.
Let us fix an entire function $A$
which is real on $\RR$ and whose zero set is exactly $\{t_n\}_{n\in
\mathcal{N}}$.
Choose a subsequence $\mathcal{N}_1 = \{n_k\}$ of
$\mathcal{N}$ such that $t_{n_{k+1}} > 2t_{n_k}>0$. By Lemma
\ref{lem_entire}, there exists a function $n_0: [1,\infty)$, such that if $U$ is an
entire function of order less than one
whose zero counting
function $n_U$ satisfies $\int_1^R  \frac{n_U(r)}{r}dr =
o\big(\int_1^R  \frac{n_0(r)}{r}dr\big)$, $R\to \infty$, and
$\{U(t_{n_k})\}$ is bounded, then $U$ is a constant.

The main step of the proof will be a construction of an entire
function $S$ of order zero such that
$\int_0^R \frac{n_S(r)}{r}\,dr = o\big(\int_0^R  \frac{n_0(r)}{r}\, dr\big)$,
$R\to \infty$,
and of an entire function $g$ with simple real zeros distinct from
$\{t_n\}$ that have the following properties:

(a) $g$ has the representation
\begin{equation}
\label{arb0}
\frac{g(z)}{A(z)} = \sum_{n\in \mathcal{N}} \frac{d_n}{z-t_n},
\end{equation}
where $d_n$ are real, nonzero and, for any $N>0$,
\begin{equation}
\label{arb1}
2^{|n|} |d_n| =o(|t_n|^{-N}), \qquad |n|\to\infty,\  n\notin \mathcal{N}_1,
\end{equation}
\begin{equation}
\label{arb2}
2^k |d_{n_k}| = o(\tk^{-N}), \qquad k\to\infty
\end{equation}
(thus, the sum in the definition of $g$ converges absolutely and uniformly on compact sets);

(b) $g$ admits the estimate
\begin{equation}
\label{arb3}
\frac{C_1}{|S(z)|}\cdot \bigg(\frac{|\ima z|}{|z|^2+1}\bigg)^2
\le \bigg|\frac{g(z)}{A(z)}\bigg| \le \frac{C_2}{|S(z)|}\cdot \bigg(\frac{|z|^2+1}{|\ima z|}\bigg)^2,
\qquad z \notin\RR,
\end{equation}
for some constants $C_1,\, C_2>0$.

Once $g$ has been constructed, we will define a de Branges space $\he$, where $E=A-iB$, $A\notin \he$
and will
prove that $g$ and $E$ satisfy the above properties (i)--(v).
%
\bigskip
\\
{\bf Construction of $S$ and $g$.}
Define entire functions $A_0$ and $B_0$ by the formulas
\begin{equation}
\label{arb7}
A_0(z) = \prod_k \bigg(1-\frac{z}{t_{n_k}}\bigg), \qquad  \frac{B_0(z)}{A_0(z)} = \sum_k \frac{v_k}{t_{n_k} - z},
\end{equation}
where $v_k \in (1,2)$ are chosen so that $B_0(t_n) \ne 0$
for $n\in \mathcal{N} \setminus \mathcal{N}_1$ and $B_0(0) \ne 0$.
Thus $B_0/A_0$ is a Herglotz function and each interval $(t_{n_k}, t_{n_{k+1}})$ contains exactly one zero
of $B_0$, which we denote by $s_k$. Choose a subsequence $\{s_{k_j}\}$
of $\{s_k\}$ so sparse that the entire function
$$
S(z) = \prod_j \bigg(1-\frac{z}{s_{k_j}}\bigg)
$$
satisfies $\int_0^R  \frac{n_S(r)}{r}dr = o\big(\int_0^R  \frac{n_0(r)}{r}dr\big)$ as $R\to \infty$.

Since $t_{n_{k+1}} > 2t_{n_k}$, an easy estimate of the sum in the second equation
in \eqref{arb7} shows that there exists $\delta>0$ such that $\dist (s_l, \{t_{n_k}\}) \ge \delta$.
Therefore, for any $N>0$, we have $t_{n_k}^N = O(|S(\tk)|)$, $k\to \infty$, and, in particular, $\sum_k |S(\tk)|^{-1} <\infty$.
Hence, we have
\beqn
\label{p_k}
\frac{B_0(z)}{S(z)A_0(z)} = \sum_k \frac{v_k}{S(t_{n_k}) (t_{n_k} - z)}
=\sum_k \frac {p_k} {t_{n_k} - z},
\neqn
where $p_k=\frac{v_k}{S(t_{n_k})}$.
%
%
Indeed, the residues in the right- and left-hand sides coincide and it is easily seen that their difference is an entire function
of zero type (apply Krein's theorem) which tends to $0$ along $i\RR$, and thus is identically zero.
Next, put
\begin{equation}
\label{arb8}
\gamma(z)  = 1 +\sum_{n\in \mathcal{N} \setminus \mathcal{N}_1}\frac{q_n}{t_n - z},
\end{equation}
where $q_n>0$ and $\sum q_n < 1$ (thus, $\gamma(0) \ne 0$).
The size of $q_n$ will be specified later.

Now we define the entire function $g$ by
\beqn
\label{def-g}
\frac{g(z)}{A(z)} = \frac{B_0(z)}{S(z)A_0(z)} \,\gamma (z).
\neqn
%
%
It is real on $\BR$, has only real zeros and $g(0) \ne 0$,
since $B_0(0)\gamma(0) \ne 0$.
Therefore \textbf{(ii)} holds.
Note that in \eqref{p_k}, $|2^kp_k| = o(\tk^{-N})$ for any $N>0$. We have
$$
\begin{aligned}
\frac{g(z)}{A(z)} & =
\bigg(\sum_k \frac{p_k}{\tk - z}\bigg)\bigg(
1 +\sum_{n\in \mathcal{N} \setminus \mathcal{N}_1} \frac{q_n}{t_n - z}\bigg) \\
& = \sum_k \bigg( 1 - \sum_{n\in \mathcal{N} \setminus \mathcal{N}_1}  \frac{q_n}{t_{n_k} - t_n}\bigg)
\frac{p_k}{t_{n_k}-z} + \sum_{n\in \mathcal{N} \setminus \mathcal{N}_1} \bigg(\sum_k \frac{p_k}{t_{n_k} - t_n}\bigg) \frac{q_n}{t_n-z}.
\end{aligned}
$$
Now we specify the choice of $q_n$ justifying the convergence in the last equation
and the validity of the interchange of sums. Namely, we put
$$
q_n = (2|t_n|)^{-|n|} \min_k |t_n - \tk|, \qquad n\in \mathcal{N}
\setminus \mathcal{N}_1.
$$
Then it is clear that the function $g/A$ has the representation
\eqref{arb0} with the coefficients
$$
\begin{aligned}
d_{n_k} & =  - p_k \Big(
1  - \sum_{n\in \mathcal{N} \setminus \mathcal{N}_1} \frac{q_n}{t_{n_k} - t_n}\Big), \\
d_n & = - q_n \sum_k  \frac{p_k}{t_{n_k} - t_n},
\qquad
n\in \mathcal{N} \setminus \mathcal{N}_1,
\end{aligned}
$$
which have the properties \eqref{arb1}--\eqref{arb2}.
Note also that the coefficients at $(z-t_n)^{-1}$ are non-zero
for any $n\notin \mathcal{N}_1$ by the smallness of $q_n$,
while $d_{n_k} \ne 0$ by the choice of $v_k$.

By an easy estimate of the Herglotz
integral, for any function $f$ with $\ima f> 0$ in $\BC^+$ we have
\begin{equation}
\label{arb9}
\frac{C|\ima z|}{|z|^2+1} \le |f(z)| \le \frac{C'(|z|^2+1)}{|\ima z|},
\qquad z\notin \RR,
\end{equation}
for some $C, C' > 0$.
Since both $B_0/A_0$ and $\gamma$ are Herglotz functions,  \eqref{def-g}
implies the estimate \eqref{arb3}.
%
%
%
%
\bigskip
\\
\textbf{Choice of the weights and construction of the de Branges space $\he$.}
Choose another subsequence of indices $\mathcal{N}_2 = \{m_j\}\subset \BN$ so that
$\mathcal{N}_1\cap \mathcal{N}_2 = \emptyset$,
$t_{m_{j+1}}>2\tj$, which is so sparse that for the entire function
$$
T(z) = \prod_j \bigg(1-\frac{z}{\tj}\bigg)
$$
we have $|yT(iy)| = o(|S(iy)|)$, $|y| \to \infty$ and
\begin{equation}
\label{arb4}
\sum_{n\notin \mathcal{N}_1} 2^{|n|} d_n^2 |T(t_n +i )|^2 <\infty, \qquad
\sum_k 2^k d_{n_k}^2 |T(\tk+i)|^2 <\infty.
\end{equation}
This is possible by the properties \eqref{arb1}--\eqref{arb2} of the coefficients $d_n$.

Now we define the positive weights $\nu_n$ as follows:
$$
\nu_{n_k} = d_{n_k}^2, \qquad n_k \in \mathcal{N}_1,
$$
$$
\nu_{m_j} = 1, \qquad m_j \in \mathcal{N}_2,
$$
$$
\nu_n = 2^{|n|} d_n^2, \qquad n\notin \mathcal{N}_1\cup \mathcal{N}_2.
$$
By construction $\sum_{n\in \mathcal{N}} \nu_n = \infty$, but
$\sum_{n\in \mathcal{N}} t_n^{-1}\nu_n <\infty$. Indeed,
$\sum_{n\notin \mathcal{N}_2}\nu_n <\infty$, while  $\sum_{j}
t_{m_j}^{-1} < \infty$. Therefore, we may define an entire
function $B$ by
\beqn
\label{frac_B_A}
\frac{B(z)}{A(z)} = \sum_n
\frac{\nu_n}{t_n - z}.
\neqn
Since $\ima \frac{B}{A} > 0 $ in
$\BC^+$, we get that $E=A-iB$ is a function in the
Hermite--Biehler class. Let $\Theta = E^*/E$ be the corresponding
meromorphic inner function. Then we have $A = E(1+\Theta)$ and
$$
\frac{1-\Theta(z)}{1+\Theta(z)} = \frac{1}{i}\sum_n \frac{\nu_n}{t_n - z}.
$$
Thus, by (\ref{mass}) $1+\Theta\notin H^2$ (equivalently, $A\notin \he$).
This gives \textbf{(i)}.
\bigskip
\\
{\bf Proof of (iii).} First we notice that $g\notin\he$. Indeed, if $g$ were
an element of $\he$, then
from the fact that $\nu = \sum_n \nu_n\delta_n$
is a Clark measure $\sigma_{-1}$ for $K_\Theta$, we would get
$$
\sum_n \frac{|g(t_n)|^2}{|E(t_n)|^2} \,\nu_n =
\sum_n d_n^2 \frac{|A'(t_n)|^2}{|B(t_n)|^2} \,\nu_n <\infty.
$$
By \eqref{frac_B_A}, $B(t_n)/A'(t_n) = -\nu_n$. Hence
the latter sum equals to
$\sum_n \nu_n^{-1}d_n^2$, which is infinite by the definition of $\nu_n$,
a contradiction. Thus $g\notin\he$.

%
%
%
%
%
On the other hand, if $\lambda $ is a zero of $g$ (notice that $Z_g \cap \{t_n\} = \emptyset$), then
$\la\in \BR$ and
$$
\sum_n \frac{|g(t_n)|^2}{(t_n-\lambda)^2 |E(t_n)|^2} \,\nu_n =  \sum_n \frac{d_n^2}{(t_n - \lambda)^2 \nu_n} <\infty.
$$
Recall that
$$
\|K_{t_n}\|_E^2=\frac{|B(t_n)A'(t_n)|} \pi\, , \quad \text{ whence}\quad
\widetilde K_{t_n}(z) = \frac{K_{t_n}(z)}{\|K_{t_n}\|_E} =
- \nu_n^{1/2} \frac{\sqrt{\pi}\, \sign B(t_n)\,A(z)}{z-t_n}.
$$
Therefore, the interpolation formula
$$
\frac{g(z)}{(z-\lambda)A(z)} = \sum_{n\in N} \frac{d_n} {t_n-\lambda} \cdot \frac{1}{z-t_n}
$$
(which clearly holds since the difference of the left- and right-hand parts is an
entire function of zero type which tends to zero
along the imaginary axis) can be rewritten as
$$
g(z) = \sum_n \frac{d_n}{\nu_n (t_n-\lambda)} \cdot\frac{A(z)}{z-t_n} =
-\pi^{-1/2} \sum_n
\frac{d_n\sign B(t_n)}{\nu_n^{1/2} (t_n-\lambda)} \cdot \widetilde K_{t_n}(z),
$$
and thus $\frac{g}{z-\lambda} \in \he$ for any $\lambda\in Z_g$.
\medskip

If we put $\phi = g/E$, then $\phi$ and $\Theta$ correspond,
by Theorem \ref{rank-one-model}, to
an almost Hermitian singular rank one perturbation associated with certain
$\mu$, $a$ and $b$.
In particular, condition $({\rm A})$ is satisfied.
Moreover,
\[
\sum_n |a_n|^2\mu_n=
\sum_n |\phi(t_n)|^2\nu_n =
\sum_n \frac {|g(t_n)|^2}{|E(t_n)|^2}\, \nu_n
= +\infty.
\]
On the other hand,
by construction, $\sum_n \nu_n = \sum_n |b_n|^2\mu_n = \infty$, thus $b
\notin \Elltwo(\mu)$ and $({\rm A^*})$ is satisfied. Therefore, the
adjoint operator is correctly defined.
\bigskip
\\
{\bf Proof of (iv).}
First let us show that
the system $\big\{\frac{g(z)}{z-\lambda}\big\}_{\lambda \in Z_g}$ is not complete in $\he$.
Consider the function
$$
T_1(z) = \prod_j\Big(1-\frac{z}{t_{m_j} + \delta_j}\Big),
$$
where $\delta_j>0$ are chosen to be so small that
$$
\sum_j|T_1(t_{m_j})|^2 =\sum_j \nu_{m_j}|T_1(t_{m_j})|^2   <\infty.
$$
 Also, for sufficiently small $\delta_j$
we have $|T_1(t_n)| \le |T_1(t_n+i)| \asymp |T(t_n+i)|$, whence, by~\eqref{arb4}
$$
\sum_{n\notin \mathcal{N}_1\cup \mathcal{N}_2} \nu_n |T_1(t_n)|^2 = \sum_{n\notin \mathcal{N}_1 \cup \mathcal{N}_2} 2^n d_n^2|T_1(t_n)|^2<\infty
$$
and
$$
\sum_k \nu_{n_k} |T_1(\tk)|^2 = \sum_k d_{n_k}^2 |T_1(\tk)|^2 <\infty.
$$
Thus, $\sum_n \nu_n |T_1(t_n)|^2 <\infty$.

Now we show that the function $T_1 g$ admits the following expansion:
\begin{equation}
\label{vybornulei1}
\frac{T_1(z)g(z)}{ A(z)} =  \sum_n \frac{g(t_n)}{B(t_n)}\cdot\frac{c_n}{z-t_n}\,\nu_n^{1/2}
\end{equation}
with some $\{c_n\} \in\ell^2$. Indeed, put
$$
c_n = \frac{T_1(t_n)B(t_n)}{A'(t_n) \nu_n^{1/2}} = -T_1(t_n)\nu_n^{1/2}.
$$
Then $\{c_n \} \in \ell^2$. Since $\frac{g(z)}{z-\lambda} \in \he$ and so  $\big\{\frac{g(t_n)}{(t_n-\lambda)E(t_n)}\nu_n^{1/2}\big\} \in \ell^2$
for any $\lambda \in Z_g$, the series on the right-hand side of \eqref{vybornulei1} converges uniformly on compact sets
and it is easily seen that the residues on the left- and right-hand side
of \eqref{vybornulei1} coincide. Therefore,
$$
H(z) =
\frac{T_1(z)g(z)}{ A(z)} - \sum_n \frac{g(t_n)}{B(t_n)}\cdot\frac{c_n}{z-t_n}\,\nu_n^{1/2}
$$
is an entire function of zero exponential type and $|H(iy)| \to 0$ as $|y| \to 0$ (recall that $y|T_1(iy)| = o(|S(iy)|)$
while $|g(iy)|/ |A(iy)| \le C|y|/|S(iy)|$).
Thus, $H\equiv 0$.

Now put $z=\lambda$, $\lambda\in Z_g$,  in
(\ref{vybornulei1}). We have
$$
\sum_n \frac{g(t_n)}{B(t_n)}\cdot\frac{c_n}{\lambda -t_n}\nu_n^{1/2}= 0, \qquad
\lambda\in Z_g,
$$
which is equivalent to
$$
\Big\langle \frac{g(z)}{z-\lambda}, h\Big\rangle=0, \quad \lambda
\in Z_g, \quad \text{where} \quad h=\sum_n \nu_n^{1/2}
\frac{c_n}{B(t_n)} K_{t_n}.
$$
Since $\nu_n^{1/2} \big|B(t_n)\big|^{-1} K_{t_n}$ coincides with the normalized
kernel $\widetilde K_{t_n}$ up to a constant factor, we have $h\in \he$.
Thus, the orthogonal complement to the system $\frac{g(z)}{z-\lambda}$,
$\lambda\in Z_g$, is nontrivial.

Moreover, the above argument
works as well for the function $T_1/P_m$,
where $P_m(z) = (z-z_1)\dots(z-z_m)$ and $z_1,\dots,z_m$ are the first $m$ zeros of $T_1$.
Clearly, the functions $h_m$, constructed with $T_1/P_m$ in place of $T_1$, are linearly independent. Thus, the orthogonal complement to
the system $\frac{g(z)}{z-\lambda}$, $\lambda\in Z_g$ is infinite-dimensional.
\bigskip
\\
{\bf Proof of (v).} Now we show that $g$ is the generating function of a complete system of
reproducing kernels. If the system
$\{K_\lambda\}_{\lambda\in Z_g}$ is not complete in $\he$, then there exists a nonzero
entire function $U$ such that $Ug\in \he$.
We have $gU/E = h\in H^2(\BC^+)$. Since
$g/A = (B_0/A_0) \cdot \gamma \cdot S^{-1}$, where both $A_0/B_0$ and $\gamma$ are functions
with positive imaginary part
(and, thus, in the Smirnov class) and $S$ is of zero order, we conclude that
$$
U = h\cdot \frac{E}{A} \cdot \frac{A_0}{B_0} \cdot \gamma^{-1} S
$$
is in the Smirnov class in the upper half-plane.
Since $gU^*$ is also in $\he$, the same is true for $U$ in the lower half-plane, and we conclude,
by Krein's theorem, that $U$ is of zero type.

Since $h\in H^2(\BC^+)$, we have $|h(z)| \le C_1(\ima z)^{-1/2}$, $z\in \BC^+$.
Also, by  (\ref{estim}) $|2A(z)|/|E(z)| = |1+\Theta(z)| \ge C_2 (|z|^2+1)^{-1} \ima z$,
$z\in \BC^+$.
Applying \eqref{arb9} to $B_0/A_0$ and to $\gamma$ we conclude that
$$
|U(z)| \le C_3 \frac{(|z|^2+1)^3}{(\ima z)^{7/2}} |S(z)|, \qquad z\notin \RR.
$$

We could have taken $U$ such that $U(0)=1$.
Then, by the Jensen formula,
$$
\begin{aligned}
\int_0^R \frac{n_U(r)}{r}dr & = \frac{1}{2\pi} \int_0^{2\pi}
\log |U(Re^{it})|dt \\
& \le
\log C_3 +
\frac{1}{2\pi} \int_0^{2\pi}
\log |S(Re^{it})|dt + 3\log (R^2+1) -
\frac{7}{4\pi} \int_0^{2\pi}
\log R |\sin t| \,dt \\
& = \int_0^R \frac{n_S(r)}{r}dr  +O(\log R)
=  o\bigg(\int_0^R  \frac{n_0(r)}{r}dr\bigg), \qquad R\to \infty.
\end{aligned}
$$

On the other hand, it follows from the inclusion $gU\in\he$ that
$$
\sum_n \frac{|g(t_n)U(t_n)|^2}{|E(t_n)|^2}\nu_n<\infty.
$$
Since for $n = n_k$ we have
$ |g(t_n)|^2 |E(t_n)|^{-2} \nu_n = 1$
(see the proof of $(ii)$),
we conclude that $\sum_{k} |U(t_{n_k})|^2 <\infty$.
Thus $\{U(t_{n_k})\}$ is bounded, whence, by the choice of $n_0$, $U\equiv const$.
However, $g\notin \he$, and so  $U \equiv 0$.
This contradiction proves  that the system
$\{K_\lambda\}_{\lambda \in Z_g}$ is complete in $\he$.
\end{proof}
\medskip



\section{Rank one perturbations of a compact selfadjoint operator}

In this section we complete the proofs
of the results for usual (bounded) perturbations of
compact selfadjoint operators stated in Introduction.
All the proofs will be based on a reduction to the
equivalent problem for singular rank one perturbations.

Let us fix the notations which will be used throughout this section.
To distinguish between the bounded and singular perturbations,
we use in the latter case the notation $\ta, \tl$, etc.

Let $\A$ be a compact selfadjoint operator with simple
spectrum $\{s_n\}$, $s_n \ne 0$, that is, the
operator of multiplication by the independent variable
in some space $\Elltwo(\mu)$, $\mu = \sum_n \mu_n\delta_{s_n}$.
Denote by $\ta$ the operator of multiplication by $x$
in $\Elltwo(\tilde \mu)$, where $\tilde \mu = \sum_n \mu_n\delta_{t_n}$,
$t_n = s_n^{-1}$. Then, clearly, $\A$ is unitarily  equivalent
to the inverse $\ta^{-1}$ via the obvious unitary identification
$(Ua)(t_n) = a(s_n)$ between the spaces $\Elltwo(\mu)$ and $\Elltwo(\tilde \mu)$.

For $a=(a_n), b=(b_n) \in \Elltwo(\mu)$ and $\deab\in \mathbb{C}$, $\deab \ne 0$,
consider the bounded rank one perturbation $\LL = \A - \deab^{-1} ab^*$ of $\A$.
Let us assume that $\ker \LL = 0$,
which means that either $a\notin x\Elltwo(\mu)$ or $a\in x\Elltwo(\mu)$ and
$\sum_n s_n^{-1} a_n \bar b_n \mu_n \ne \deab$. Put
\beqn
\label{bab1}
\taa_n = s_n^{-1}a_n, \qquad \tb_n= s_n^{-1} b_n
\neqn
and consider $\taa = (\taa_n)$ and $\tb=(\tb_n)$ as elements
of $\Elltwo(\tilde \mu)$ (assuming $\tilde a(t_n) = \tilde a_n$).
In our notations we can write $a = \A U^{-1}\taa$, $b = \A U^{-1}\tb$.
Then $x^{-1}\taa, x^{-1}\tb \in \Elltwo(\tilde \mu)$ and
the data $(\taa, \tb, \deab)$ satisfy condition (${\rm A }$).
Define the singular rank one perturbation
$\tl = \LL(\ta, \taa, \tb, \deab)$ of $\ta$. By Proposition \ref{L inverse},
\beqn
\label{bab2}
\tl = \big(\ta^{-1} - \deab^{-1} (\ta^{-1}\taa) (\ta^{-1} \tb)^*\big)^{-1} =
U\big(\A - \deab^{-1} ab^* \big)^{-1}U^{-1},
\neqn
and so $\tl$ is unitary equivalent to the algebraic inverse of $\LL$.

\subsection{Completeness of rank one perturbations: proofs of Theorem~\ref{genweak0}--\ref{frombb0}}

Theorem \ref{genweak0} follows immediately from
Proposition~\ref{gen_weak_Mats}.
The second condition in \eqref{smooth33} is exactly the required (one-dimensional)
invertibility condition.

\begin{proof}[Proof of Theorem \ref{positive0}]
Define $\ta$, $\taa$ and $\tb$ as above. Then
$\taa_n\overline{\tb}_n \ge0$
for all values of $n$ except a finite number and
$\sum_n |t_n|^{-1} |\taa_n \tb_n| \mu_n =
\sum_n |s_n|^{-1}|a_nb_n|\mu_n =\infty$.
In particular, $\taa, \tb \notin \Elltwo(\tilde \mu)$
and so the perturbation $\tl = \tl(\ta, \taa, \tb, -1)$
and its adjoint $\tl^*$  are well defined.
By Theorem~\ref{positive}, both $\tl$ and $\tl^*$ are complete. Hence,
$\LL =\A +ab^*$ and $\LL^*$ are complete (note that $\LL^*$
is unitary equivalent to $(\tl^*)^{-1}$).
\end{proof}

\begin{proof}[Proof of Theorem \ref{adjoint0}]
Apply Theorem~\ref{noncompleteness2} to the unbounded selfadjoint
operator $\ta$ (multiplication by $x$ in $\Elltwo(\tilde \mu)$) with
discrete spectrum $\{t_n\}$: there exists a singular rank one
perturbation $\tl  = \tl (\ta, \taa, \tb, \deab)$ such that both
$\tl$ and $\tl^*$ are well-defined, $\tl$ has simple real spectrum
and is complete, but $\tl^*$ is not complete (and the orthogonal
complement to the span of its eigenvectors is
infinite-dimensional).

Recall that in the construction of Theorem \ref{noncompleteness2}
both $\tilde a$ and $\tb$ do not belong to
$\Elltwo(\tilde \mu)$ and so conditions (${\rm A }$) and (${\rm A}^*$)
are satisfied. Moreover, by construction,
$\phi(0) \ne 0$ for the corresponding function $\phi$, whence, by
statement (5) in Proposition~\ref{param},
$\deab \ne 0$. Then, by Proposition~\ref{L inverse},
the operator $\tl^{-1}$ is unitary equivalent to the operator
$\LL = \A - \deab^{-1} ab^*$ with $a, b$ related to $\taa, \tb$ by \eqref{bab1}.

Thus, $\LL$ is complete while $\LL^*$ is not.
Also, $a, b\notin x \Elltwo(\mu)$, and so
$\ker \LL =\ker \LL^* = 0$.
\end{proof}

\begin{proof}[Proof of Theorem \ref{frombb0}]
Consider $\taa$ and $\tb$ defined by \eqref{bab1}. Since $a, b \notin x \Elltwo(\mu)$,
the data $(\taa, \tb, -1)$ satisfy (A) and (${\rm A}^*$) and the singular
rank one perturbation $\tl = \tl(\ta, \taa, \tb, -1)$
(note the reverse order of $\taa$ and $\tb$)
is well defined as well as its adjoint $\tl^*$.
By the conditions $(i)$ or $(ii)$ the data $\taa$ and $\tb$
satisfy either $|\taa_n|^2 \mu_n \ge C |t_n|^{-N}>0$
or $|\tb_n \taa_n^{-1}| \le C |t_n|^{N}$ for some $N>0$.
Since $\LL = \A+ab^*$ is complete, its algebraic
inverse (up to the unitary equivalence)
$\tl$ is also complete. Now the operator $\tl^*$ is complete
by Theorem~\ref{frombb}, whence $\LL^*$ is also complete.
\end{proof}


\subsection{Proofs of spectral synthesis results}
\label{pssr}
Here we prove Theorems~\ref{synthesis} and \ref{sp_sint}.

\begin{proof}[Proof of Theorem~\ref{synthesis}]
Let the spectrum $\{s_n\}$ of $\A$ be ordered so that
$s_n$ are positive and decrease for $n\ge 0$, and $s_n$ are
negative and increase for $n<0$. Put $t_n=s_n^{-1}$. Then
$\{t_n\}$ is an increasing sequence satisfying
$t_n\to \infty$, $|n|\to \infty$.

By \cite[Theorem 1.4]{bbb1}, there exists a de Branges space
$\he$ and an entire function $G$ such that:

(i) $\F = \frac{E+E^*}{2}$ vanishes exactly on the set $\{t_n\}$;

(ii) $\Theta \defin E^*/E$ has a representation
\beqn
\label{repr-Tht}
\frac{1-\Theta(z)}{1+\Theta(z)}=\frac 1 i\, \int \Bigl( \frac 1
{t-z} - \frac t {t^2+1} \Bigr)\, d\nu(t),
\neqn
where $\nu=\sum_n \nu_n \delta_{t_n}$ and $\sum_n \nu_n = \infty$;

(iii) $G$ has only simple and real zeros, $Z_G\cap \{t_n\} = \emptyset$,
$G(0) \ne 0$, $G\notin\he$, but the functions
$\frac{G}{z-\lambda}$ are in $\he$ for any $\lambda\in Z_G$;

(iv) The systems $\{K_\lambda\}_{\lambda\in Z_G}$  and
$\big\{ \frac{G(z)}{G'(\lambda)(z-\lambda)} \big\}_{\lambda\in Z_G}$ are complete in
$\he$ and biorthogonal one to another, but are
not hereditarily complete
(as before, here $K_\lambda$ stand for the
reproducing kernel of $\he$)
\footnote{In \cite[Theorem 1.6]{bbb}
a de Branges space $\he$ and an entire function $G$ satisfying (i)--(iv)
were constructed under the additional restriction
$c|t_n|^{-N} \le t_{n+1} - t_n = o(|t_n|)$, $|n|\to\infty$,
for some $c,\, N>0$.};

(v) Moreover, for some partition $Z_G = \Lambda_1 \cup\Lambda_2$,
$\Lambda_1\cap \Lambda_2 = \emptyset$, of the zero set $Z_G$ the orthogonal
complement to the mixed system
$$
\Big\{\frac{G(z)}{z-\lambda}\Big\}_{\lambda\in \Lambda_1}
\cup \{K_\lambda\}_{\lambda\in\Lambda_2}
$$
in $\mathcal{H}(E)$  is infinite-dimensional.

Notice that \eqref{repr-Tht} implies that
$1+\Theta\notin H^2$ (see \eqref{clark-meas} and \eqref{mass}).
On the other hand, one also has $\zeta+\Theta\notin K_\Theta$ (and, thus,
$\zeta+\Theta\notin H^2$) for any $|\zeta|=1$, $\zeta\ne 1$.
Indeed, $\zeta+\Theta\notin \Elltwo(\nu)$, because
$\zeta+\Theta$ is a nonzero constant on the sequence $\{t_n\}$ and
$\nu\big(\mathbb{R}\big)=\infty$.

Put  $\phi = G/E$. By the converse statement in
Theorem \ref{rank-one-model}, $\Theta$ and $\phi$
are functional parameters of the model of some
almost Hermitian singular rank one perturbation
$\tl \defin \tl (\ta, \taa, \tb, \varkappa)$ of the
operator $\ta$ of multiplication by $x$ in some space
$\Elltwo(\tilde \mu)$, $\tilde \mu = \sum_n \mu_n \delta_{t_n}$.
Here $\{t_n\} = \{x: \Theta(x) = -1\}$
and $\nu_n = |\tb_n|^2\mu_n$.

Since $\phi(0) \ne 0$, we conclude that $\varkappa \ne 0$
(see Proposition~\ref{param}, (5)).
Since $\zeta+\Theta \notin H^2$ for any constant $\zeta$,
$|\zeta|=1$, by Theorem~\ref{mincomp},
$\tl (\ta, \taa, \tb, \varkappa)$ has an adjoint, so that
the data $(\taa, \tb, \varkappa)$ satisfy  $(\admone)$
and $(\admonest)$. By Proposition~\ref{L inverse},
$\tl^{-1}$ is a unitary equivalent to a bounded rank one perturbation $\LL$
of the operator $\A$; namely,
$\LL = \A-\deab^{-1} ab^*$, where $a,b$ are related to $\taa, \tb$
by \eqref{bab1} (see \eqref{bab2}).
Conditions $(\admone)$ and $(\admonest)$ are equivalent to $\ker\LL=\ker\LL^*=0$.

The systems $\big\{\frac{\phi}{z-\lambda}\big\}_{\lambda\in Z_G}$
and $\{k_\lambda\}_{\lambda\in Z_G}$
are unitarily equivalent to the systems of eigenfunctions
of $\LL$, $\LL^*$, respectively.
By (iv), the systems $\{k_\lambda\}_{\lambda\in Z_G}$ and
$\big\{\frac{\phi}{z-\lambda}\big\}_{\lambda\in Z_G}$ are minimal
and complete in $K_\Theta$.
Since these systems are not hereditarily complete,
by applying the above-cited theorem by Markus~\cite[Theorem~4.1]{markus70},
we conclude that $\LL$ does not admit spectral synthesis.
\end{proof}

In the proof of Theorem \ref{sp_sint} we will use the following lemma.

\begin{lemma}
\label{distan}
Let $\Theta$ be a meromorphic inner function and let
$\nu = \sum_{n\in I}
\nu_n \delta_{t_n}$ be its Clark measure $\sigma_\alpha$ for some $\alpha$
\textup(we assume that $\{t_n\}$ is an increasing sequence\textup).
The following statements are equivalent\textup:

\begin{enumerate}
\item[(i)] There exist $C, N >0$ such that
$|\Theta'(t)| \le C(|t|+1)^N$, $t\in \RR$\textup;

\item[(ii)] There exist $c, M>0$ such that
\begin{equation}
\label{hypot1}
\nu_n \ge c (|t_n|+1)^{-M}, \qquad
t_{n+1} - t_n \ge c (|t_n|+1)^{-M}.
\end{equation}
\end{enumerate}
\end{lemma}

\begin{proof}[Proof]
The implication (i)$\Longrightarrow$(ii) is obvious. Indeed,
$\nu_n = \pi/|\Theta'(t_n)|$. If we denote by $\psi$ a continuous
increasing branch of the argument of $\Theta$ on $\RR$, then
$\Theta(t) =\exp(i\psi(t))$
and $|\Theta'(t)| = \psi'(t)$, $t \in  \RR$.
We have
$$
\psi(t_{n+1}) - \psi(t_n) = \int_{t_n}^{t_{n+1}} \psi'(t) dt = 2 \,\pi,
$$
whence $t_{n+1} -t_n \ge 2 \pi C^{-1} (|t_n|+1)^{-N}$.
Thus, (ii) holds with $M=N$.

(ii)$\Longrightarrow$(i)
Without loss of generality we assume that $\nu = \sigma_{-1}$. Then, by
(\ref{clark-meas}),
$$
|\Theta'(t)| = \bigg|
i +\sum_n \nu_n \bigg(\frac{1}{t_n -t} -\frac{1}{t_n}\bigg)
\bigg|^{-2} \sum_n \frac{2 \nu_n}{(t_n -t)^2}.
$$
Now the required estimate follows from the fact that
$|\Theta'(t)| \asymp \nu_k^{-1}$ when $|t-t_k|< t_k^{-M}$
with sufficiently large $M$, while for $t$ such that
${\rm dist}\, (t,\{t_n\}) > |t|^{-M}$ we use the rough estimate
$|\Theta'(t)| \lesssim \sum_n |t-t_n|^{-2} \nu_n$
and the fact that $\sum_n (t_n^2+1)^{-1}\nu_n < \infty$.
\end{proof}

\begin{proof}[Proof of Theorem \ref{sp_sint}]
We have $\LL f_j=\la_j f_j$, where $\{\la_j\}$ are distinct and non-zero.
There is a family of eigenvectors $g_j$ of $\LL^*$ with
$\LL^* g_j=\bar \la_j g_j$, which form a biorthogonal family to $\{f_j\}_{j\in J}$.
Let $\M$ be an invariant subspace of $\LL$.
Put $J_1 = \{j: f_j\in \mathcal{M}\}$ and $J_2=J \setminus J_1$. Assume that
$h \in \M$ and $h$ is orthogonal to all $f_j$ with
$j\in J_1$
(that is, $h$ is orthogonal to $\mathcal{E}(\M)$).
The proof of~\cite[Lemma~4.2]{markus70} shows that
whenever $f_j\notin \mathcal{M}$, the corresponding eigenvalue $\lambda_j$
does not belong to $\sigma(\LL |_\mathcal{M})$
and $h$ is orthogonal to $g_j$. This implies that
$h$ is orthogonal to the system $\{f_j\}_{j\in J_1} \cup
\{g_j\}_{j\in J_2}$. We will show that in our situation,
the orthogonal complement to this system is always finite-dimensional
(for any decomposition $J=J_1\cup J_2$).

Let us pass again to the unbounded inverses
(note that $\ker \LL=\ker \LL^*=0$ since $a,\, b\notin x\Elltwo(\mu)$).
Define $\taa, \tb$ by \eqref{bab1} and let
$\tilde \mu=\sum_n \mu_n\delta_{t_n}$, $t_n = s_n^{-1}$.
Then we have $|\taa_n|^2\mu_n \ge C_2 |t_n|^{N+1}$
and also $|t_{n+1} - t_n|\ge C_1|t_n|^{-N_1+1}$.
Since $\taa, \tb \notin \Elltwo(\tilde \mu)$,
the singular perturbation $\tl  = \tl (\ta , a, b, -1)$ of the operator $\ta$
of multiplication by $x$ in $\Elltwo(\tilde \mu)$ is well defined, and $\LL$ is
unitary equivalent to its algebraic inverse.
Theorem~\ref{frombb} implies that $\tl^*$ is also complete.

Now, let $\Theta$ and $\phi$ correspond to $\tl^* = \tl  (\ta , \tb, \taa, -1)$
by Theorem~\ref{rank-one-model}, that is,
$$
\frac{1-\Theta(z)}{1+\Theta(z)} = \frac{1}{i}
\sum_n |a_n|^2\mu_n \bigg(\frac{1}{t_n -z} -\frac{1}{t_n}\bigg).
$$
We use the fact that, by the hypothesis,  $a_n \ne 0$ for any $n$
and so the functional model  applies to
$\tl (\ta , \tb, \taa, -1)$.
By Lemma~\ref{distan} (applied to $\nu_n = |a_n|^2\mu_n$),
$$
|\Theta'(t)| \le C (|t|+1)^K, \qquad t\in\RR.
$$
for some $C,K >0$ with $K$ depending only on $N$ and $N_1$.

Let $E$ be a function in $HB$ such that $\Theta = E^*/E$.
The eigenfunctions of $\tl$ are given by $\{k_\lambda\}_{\lambda\in Z_\phi} \cup
\{\tilde k_\lambda\}_{\lambda\in Z_{\tilde \phi}}$. This system transforms, after
multiplication by $E$, into the system
$\{K_\lambda\}_{\lambda \in Z_\phi\cup\overline{Z}_{\tilde \phi}}$
of the reproducing kernels of the space $\he$. Thus, our problem
reduces to the following: given a complete and minimal systems of reproducing kernels
$\{K_\lambda\}_{\lambda \in \Lambda}$ in $\he$ with
the biorthogonal system $\{G_\lambda\}_{\lambda \in \Lambda}$,
we need to show that for any partition $\Lambda = \Lambda_1\cup\Lambda_2$, the orthogonal
complement to the system $\{K_\lambda\}_{\lambda\in \Lambda_1} \cup
\{G_\lambda\}_{\lambda\in \Lambda_2}$ is finite-dimensional. For the case when
$|\Theta'(t)| = O(|t|^K)$, $|t|\to\infty$, this statement
was proved in~\cite[Theorem 5.3]{bbb}
(with the upper bound for the dimension depending only on~$K$).
\end{proof}

\begin{remark}
{\rm It follows from the results of Gubreev and Tarasenko
\cite[Theorem 2.5]{Gubr-Tar2010} that if $|\phi|^2$ is a
Muchenhoupt $A_2$-weight on $\mathbb{R}$ and $\alpha(\phi)=
\alpha(\tilde \phi) = 0$ \textup(see the notation in Subsection
\ref{abstr}\textup), then any complete rank one perturbation $\A
+ab^*$ admits the spectral synthesis. Moreover, in the case of the
Paley--Wiener space, Belov and Lyubarskii \cite{bl} constructed a
linear summation method for the Fourier series with respect to the
corresponding family of reproducing kernels.}
\end{remark}

\begin{theorem}
\label{newt} Let $\A$ be a compact selfadjoint operator with
simple spectrum $\{s_n\}$, $s_n \ne 0$. Let $\mu = \sum_n \mu_n
\delta_{s_n}$. Suppose that $a\in \Elltwo(\mu)$, but $a\notin
x\Elltwo(\mu)$, and that $a$ is a cyclic vector for $\A$. Then the
following statements are equivalent:
\smallskip

$(i)$ For any $b\in \Elltwo(\mu)$ such that
$\LL=\A + ab^*$ is complete, $\LL$ admits the spectral
\smallskip
synthesis\textup;

$(ii)$ we have
\beqn
\label{gafa} \sum_{k: \; |s_k|>|s_n|}
\frac{|a_k|^2\mu_k}{s_k^2} \lesssim
\frac{|a_n|^2\mu_n}{s_n^2}, \qquad \sum_{k: \; |s_k|\le
|s_n|} |a_k|^2\mu_k  \lesssim |a_n|^2 \mu_n.
\neqn
\end{theorem}

Notice that conditions \eqref{gafa}  imply that
there are some constants $C>0$ and $\ga\in (0,1)$ such that
$|a_n|^2\mu_n\le C \ga^{|n|}$, while $s_n^{-2} |a_n|^2\mu_n \ge C^{-1} \ga^{-|n|}$.
Hence, in particular, $s_n$ tends to $0$ exponentially fast as $|n|\to \infty$.

\begin{proof}
Given a rank one perturbation $\LL = \A+ab^*$
which is complete, define $\taa$ and $\tb$
by \eqref{bab1} and consider the singular rank one perturbation
$\tl = \tl(\ta, \taa, \tb, -1)$ of $\ta$ (where $\ta$ is the multiplication by $x$
in $\Elltwo(\tilde \mu)$) and its adjoint $\tl^* =
\tl(\ta, \tb, \taa, -1)$.
Note that $\ker \LL =0$
since $a\notin x \Elltwo(\mu)$, whereas
the completeness of $\LL$ implies that $\ker \LL^* =0$.
Hence, $\tl$ and $\tl^*$ are well defined.

Thus, the operator $\LL$ admits the spectral synthesis
for any $b$ such that $\LL$ is complete if and only if
each perturbation $\tl$ of $\ta$
(or its adjoint $\tl^*$)
admits the spectral synthesis
whenever $\tl$ is complete. Recall that the spectral synthesis property
holds or not simultaneously for $\tl$ and its adjoint $\tl^*$.

Let $\Theta$ and $\phi$ be the corresponding
functional parameters of the model for $\tl^*$
(we use the fact that $a_n \ne 0$) and let $E\in HB$ be such that
$\Theta = E^*/E$. Then we have, for some $q\in \mathbb{R}$,
\beqn
\label{gafa1}
\frac{1-\Theta(z)}{1+\Theta(z)}=
iq + \frac{1}{i} \sum \Bigl( \frac 1
{t_n-z} - \frac{1}{t_n}\Bigr) |\taa_n|^2\mu_n.
\neqn
Note that $\Theta$ depends only on $a_n$, but not on $b_n$.
The eigenfunctions of $\tl^*$ are of the form
$\{\frac{\phi}{z-\lambda}\}_{\lambda\in \Lambda}$
while its biorthogonal (eigenfunctions of $\tl$) is given, up to normalization, by
$\{k_\lambda\}_{\lambda\in \Lambda}$. Multiplying these systems by
$E$, we obtain biorthogonal systems $\{K_\lambda\}_{\lambda\in \Lambda}$
and $\big\{ \frac{G(z)}{G'(\lambda)(z-\lambda)}\big\}_{\lambda\in \Lambda}$
in the de Branges space $\he$. We assume here that the spectrum of $\LL$
(equivalently, of $\tl$) is simple; the case of higher order root vectors
can be treated analogously.

Recall that by Theorem \ref{mincomp}, any generating function of a complete
and minimal system of reproducing kernels in $K_\Theta$ may be realized
as the function $\phi$ corresponding to some $\tilde b$.
Thus, the problem reduces to the following:
given a de Branges space $\he$ where $\Theta = E^*/E$ is given by \eqref{gafa1}
(with $\sum_n |\taa_n|^2\mu_n =\infty$), when
is it true that {\it any complete system of reproducing kernels
$\{K_\lambda\}_{\lambda\in \Lambda}$ in $\he$
is hereditarily complete}? This question was answered in \cite[Theorem 1.1]{bbb1}: this
property holds if and only if
$$
\sum_{|t_k|<|t_n|} |\taa_k|^2\mu_k \lesssim |\taa_n|^2\mu_n,
\qquad
\sum_{|t_k|\ge |t_n|} \frac{|\taa_k|^2\mu_k|}{t_k^2}  \lesssim
\frac{|\taa_n|^2\mu_n|}{t_n^2}.
$$
Clearly, these inequalities coincide with \eqref{gafa} since
$s_n = t_n^{-1}$ and $\taa_n = a_n/s_n$.
\end{proof}


\subsection{Sharpness of Macaev's theorem}
In this subsection we prove Theorem~\ref{sharp} about
existence of Volterra rank one perturbations.

\begin{proof}[Proof of Theorem~\ref{sharp}]
As above, we pass from the bounded rank one perturbation
$\LL = \A - ab^*$ to the equivalent problem for a singular rank one
perturbation $\tl =\tl(\ta, \taa, \tb, 1)$ of the unbounded
multiplication operator $\ta$
on $\Elltwo(\tilde \mu)$, where $\tilde \mu=\sum_n \mu_n\delta_{t_n}$,
$t_n =s_n^{-1}$, and $\taa$ and $\tb$ are related to $a$ and $b$
by \eqref{bab1}. Thus we need to find $t_n$ with $|t_n| \to \infty$ and
$\tilde \mu$ such that for any
$\alpha_1, \alpha_2 \ge 0$ with $\alpha_1+ \alpha_2 <1 $
there exist $\taa_n, \tb_n$ and $\deab\in \RR$ such that
the data $(\taa, \tb, \deab)$ satisfy conditions $(\admone)$, $(\admonest)$
and
\beqn
\label{shar0}
\begin{aligned}
\sum_n |\taa_n|^2 |t_n|^{2\alpha_1-2}\mu_n & =
\sum_n |a_n|^2 |s_n|^{-2\alpha_1} \mu_n
<\infty, \\
\sum_n |\tb_n|^2 |t_n|^{2\alpha_2-2}\mu_n & =
\sum_n |b_n|^2 |s_n|^{-2\alpha_2} \mu_n <\infty,
\end{aligned}
\neqn
and the function
$$
\phi(z) =
\frac{1+ \Theta(z)}{2} \cdot
\bigg(1 + \sum_n \Big(\frac{1}{t_n-z} -\frac{1}{t_n}\Big)
\taa_n \overline{\tb_n} \mu_n \bigg)
$$
has no zeros in $\BC$.
Here $\Theta$ is the meromorphic function defined
by the formula
\beqn
\label{th8}
\Theta(z) = \frac{\rho(z)-i}{\rho(z)+i},
\qquad
\rho(z)  = \sum_n  \Big(\frac{1}{t_n-z} -\frac{1}{t_n}\Big)
|\tb_n|^2\mu_n.
\neqn

Assume for the moment that such $\{t_n\}$, $\taa$ and $\tb$ are constructed.
Then $\sigma(\tl) = Z_\phi \cup \overline{Z}_{\tilde \phi} = \emptyset$.
Applying Proposition~\ref{L inverse}, we get
a bounded rank one perturbation $\LL= \A -ab^*$ of $\A$ such that
$\ker \LL= \ker \LL^*=0$. Since $\LL^{-1}$ is unitary equivalent to
$\tl$ we conclude that $\LL$ is a Volterra operator.

To construct $\{t_n\}$, $\taa$, $\tb$, take the entire function
$\F(z)= \cos\big(\pi \,\root \of z\big)$.
Its zeros are $t_n=(n-\frac 12)^2$, $n=1,2,\dots$.
It is easy to see that
\beqn
\label{ML frac 1F}
\nuu(z)\defin\frac 1 {\F(z)} =
1 + \sum_{n\in\mathbb{N}}
\Big(\frac{1}{t_n-z} -\frac{1}{t_n}\Big) c_n \mu_n,
\neqn
where $\mu_n=1$,
$$
c_n=
- \frac 1 {\F'(t_n)} = \frac 2\pi (-1)^{n+1} \, \Big(n-\frac 12\Big).
$$
Let $\alpha_1+\alpha_2 = 1-\eps$, where $\eps>0$.
We put $\taa_n=n^{2-2\alpha_1 -\frac 12-\eps}$ and
$$
\tb_n=\frac {c_n}{\taa_n},
\qquad
|\tb_n|\asymp \frac n{n^{2-2\alpha_1 -\frac 12-\eps}}=n^{2-2\alpha_2
-\frac 12-\eps}.
$$
Hence $\taa_n$, $\tb_n$ satisfy \eqref{shar0}.
At the same time, $\taa, \tb \notin \Elltwo(\tilde \mu)$,
and so the data $(\taa, \tb, 1)$ satisfy
$(\admone)$ and $(\admonest)$.
Define $\Theta$ by \eqref{th8} and put $\phi = (1+\Theta)\beta/2$.
Since $c_n \in {\RR}$, we have $\phi/\bar\phi = \Theta$.
By Theorem~\ref{rank-one-model}, $\Theta$ and $\phi$ correspond
to the real type singular perturbation $\tl(\ta, \taa, \tb, 1)$
with the smoothness properties (\ref{shar0}). Also, we have
$$
\phi = \frac{1+\Theta}{2\F}.
$$
Since the zeros $t_n$ of $1 + \Theta$ are exactly the
poles of $1/\F$, $\phi$ has no zeros in $\clos{\BC}^+$.
Hence the spectrum of the singular
perturbation $\tl (\ta, \taa, \tb, 1)$ is empty.

Note that $\A$ is a positive operator with eigenvalues $s_n =
(n-\frac 12)^{-2}$, and so it belongs to $\mathfrak{S}_p$ for any $p>1/2$.
\end{proof}
\bigskip


\section*{Appendix 1: A brief survey of the completeness results}
It should be noted
that the literature on completeness of linear operators is
very extensive, so here we will give only its very brief overview.

There are several abstract result on completeness we do not
mention here, see Dunford and Schwartz' book~\cite{Dunf_Schwarz},
Part 2,  Ch. XI and Part 3, Ch. XIX, \S 6. In particular, in
Theorem~XI.9.29, they give a result close to the Keldy\v{s} and
Macaev's theorems, with no assumption on the triviality of the
kernel. We also refer to the
books~\cite{Gohb_Krein},~\cite{Gohb_Kr_Volterr} by Gohberg and
Krein, in particular, for treatments of the dissipative case and
for theorems on the relationship between the sizes of the real and
the imaginary part of a compact operator. In the work by G.
Gubreev and A. Jerbashian~\cite{Gubreev_Jerb1991}, a completeness
result, generalizing the Keldy\v s theorem, is given by applying
M. Jerbashian's classes and their factorization theory.
See also
~\cite{Bask},~\cite{Bask_Kats} for an abstract result on similarity
of a perturbed operator to a generalized spectral operator and for an application to
integro-differential operators with nonlocal boundary conditions. In the book
~\cite{Aziz_Iokh_book}, classical results for spaces with indefinite metric are presented.
We also mention a more recent monograph~\cite{markus-book} by Markus, where also different approaches to the completeness
properties of operator pencils are treated with detail.
%
%
%
We remark that abstract Banach space results on completeness are also known,
see, for instance, \cite{Markus1966},~\cite{Burgoyne} and~\cite{Zhang2001}.

Much more is known for nonselfadjoint operators corresponding to
boundary value problems for ordinary differential equations or
systems. The literature devoted to this field is very extensive.
We will mention here the works by Malamud (2008),
Shkalikov (1976, 1979, 1982) and Malamud and Oridoroga
~\cite{MalamOrid2012}  (2012) (see the references in
~\cite{MalamOrid2012}, where an up-to-date account of this work is
given), and also the works by Minkin~\cite{Minkin_res} and by
Shubov~\cite{Shubov2011}, Freiling, Rykhlov, Yurko
\cite{FrRykhYur02}, \cite{Rykhlov2009}. Reviews~\cite{Radz82}
and~\cite{RosSolSh89} contain a systematic exposition of this
field.
%
%
In~\cite[\S 6]{Agranovich}, abstract results on completeness
and their relation to partial differential
and pseudodifferential operators on closed manifolds
are reviewed.

In some cases when the completeness holds, the Abel summability
property for eigenfunction can also be proved,
see~\cite{Lidsk62},~\cite{Mats64},~\cite{KostyuShka78},~\cite{FreTroo98},~\cite{Yakubov},~\cite{Boim}
and the review \cite{Agranovich}, \S6.4.

The completeness of finite-dimensional perturbations of Volterra
integral operators and its relationship with
expansions by generalized eigenfunctions of ordinary differential operators
has been studied by Khromov~\cite{Khromov}.

Even stronger property of eigenvectors and generalized eigenvectors
is to form a Riesz basis or
a Riesz basis with parenthesis. There are hundreds of works dedicated
to different aspects of this property.
In relation with differential operators, these properties are discussed
in the above-mentioned reviews, the book~\cite{markus-book} and recent works
\cite{Dzhn94}, \cite{Add-Mityag}, \cite{Makin},
\cite{Shk_2010}, \cite{GesztTk2012},
\cite{LunMal15} and \cite{LunMal16}.
Wyss gives in~\cite{Wyss} both abstract results and
applications to block operator matrices
and to differential operators. Notice that the similarity to a
normal operator with a discrete spectrum is equivalent to the
property of eigenvectors to form a Riesz basis. There are
several papers exploiting the model approach, among which we can
cite~\cite{Kap2},~\cite{Nikolsk-Treil02} and~\cite{VasyuKup}. In
papers~\cite{Veselov},~\cite{Veselov2}, certain criteria for the
completeness and Riesz basis properties of eigenvectors are
obtained in the context of the Naboko's model of nondissipative
nonselfadjoint operators.

Recently, the existence of invariant subspaces and other spectral
properties of finite rank perturbations of diagonalizable normal
operators have been studied in~\cite{Ionascu},~\cite{FangXia2012}
and~\cite{FoiasJungPearcy}. We remark that the similarity of a
compact perturbation of a normal operator with no eigenvalues to
an unperturbed one has also been studied, see~\cite{Y} and
references therein.
\medskip

\textbf{Malamud's example.}
It was pointed out to us by Mark Malamud that, combining
the results of \cite{MalamOrid2012} and \cite{LunMal15},
one can give a simple example of a non-dissipative
Dirac operator $L$ on a finite interval like in Theorem~\ref{adjoint0}, which
is a rank one singular perturbation of a normal operator and is
complete, whereas the adjoint operator $L^*$ is incomplete
and the span of its root vectors is of infinite codimension.
Here we reproduce his example, with his kind permission.

Let $b_1, b_2 \in \BC \setminus\{0\}$,
\[
B= \begin{pmatrix}
b_1^{-1} & 0 \\
0 & b_2^{-1} \\
\end{pmatrix}
\ne B^* \quad \text{and} \quad
Q=
\begin{pmatrix}
Q_{11} & Q_{12} \\
Q_{21} & Q_{22} \\
\end{pmatrix} \in \Elltwo([0,1]; \BC^2).
\]
Consider the associated Dirac operator on $L^2([0,1]; \BC^2)$:
\beqn
\label{Dirac}
Ly=L_{U} y =-iBy'+Q(x)y,
\neqn
whose domain is the set of functions $y(x)$ in the Sobolev space
$W^{1,2} ([0,1]; \BC^2)$
satisfying the boundary conditions
\beqn
\label{Uj}
U_1(y)\defin y_1(0)- h y_2(0)=0, \quad
U_2(y)\defin y_1(1)+ h y_2(0)=0,
\neqn
where $h\in \BC \setminus\{0\}$.
By \cite[Theorem 6.1]{MalamOrid2012}, the system of root functions of
$L_{U}$ is complete and minimal whenever $b_1b_2^{-1}\notin \BR$.

Consider also the
Dirac operator $L_V$ on $L^2([0,1]; \BC^2)$, defined by
the same differential expression as in~\eqref{Dirac}, together
with the antiperiodic boundary conditions
\[
\label{Vj}
V_1(y)\defin y_1(0) + y_1(1)=0, \quad
V_2(y)\defin y_2(0) + y_2(1)=0.
\]

From now on, let us restrict ourselves to the case when $Q\equiv 0$. Then $L_V$
is a direct sum $A_1\oplus A_2$ of two first order differential
operators on $L^2([0,1]$ such that $b_1A_1$ and $b_2A_2$ are
selfadjoint. Therefore $L_V$ is an unbounded normal operator. A
simple calculation gives
\beqn
\label{inv_L_u}
L_U^{-1} f=
\begin{pmatrix}
y_1(x) \\
y_2(x)
\end{pmatrix}
=
\begin{pmatrix}
ib_1 \big[
\int_0^x f_1(t)\, dt - \frac 12 \int_0^1 f_1(t) \big]  \\
ib_2\int_0^x f_2(t)\, dt - \frac {ib_1}{2h} \int_0^1 f_1(t)
\end{pmatrix}\, .
\neqn
Put $e_1= \bf{1} \oplus \bf{0}$, $e_2= \bf{0} \oplus \bf{1}$,
$e_1, e_2\in L^2([0,1]; \BC^2)$.
By applying \eqref{inv_L_u} and a similar expression for
$L_V^{-1} f$, one gets
\[
(L_U^{-1}- L_V^{-1})f=
\frac i 2 \,\langle f,  \bar b_1\bar h^{-1} e_1 - \bar b_2 e_2\rangle \, e_1.
\]
Therefore $L_U^{-1}- L_V^{-1}$ has rank one, in other words,
$L_U$ is a complete rank one singular perturbation of a
normal operator $L_V$.

On the other hand, the adjoint operator $L_U^*$ is given by the differential expression
$L^* = - i B^* \frac d {dx}$ and the boundary conditions
\[
\bar h y_1(0)+
{\bar b_1}{\bar b_2}^{-1} y_2(0) + \bar h y_1(1) = 0, \quad
y_2(1)=0.
\]
Clearly all its eigenvectors satisfy $y_2(x)\equiv 0$, so that
$L_U^*$ is not complete, with infinite defect.

Note that it is essential for completeness of $L_U$
that $b_1b_2^{-1}\notin \BR$, and so the normal operator
$L_V$ is nonselfadjoint. Its spectrum is given by
$\{b_1^{-1}(\pi+2\pi k)\}_{k\in\mathbb{Z}}\cup
\{b_2^{-1}(\pi+2\pi k)\}_{k\in\mathbb{Z}}$, while the spectrum of $L_U$
coincides with its ''half''
$\{b_1^{-1}(\pi+2\pi k)\}_{k\in\mathbb{Z}}$.

\renewcommand{\theequation}{A.\arabic{equation}}
\setcounter{equation}{0}  

\bigskip

\section*{Appendix 2: Proofs of Propositions \ref{Pr1add}
and \ref{Pr-adjoints-n}
\\ on singular balanced rank $n$ perturbations}
\label{s_Appendix}

\begin{proof}[Proof of Proposition~\ref{Pr1add}]
First we prove assertion (2). To this end, assume that $\LL $ is a
balanced rank $n$ singular perturbation of $\A $. We divide our
argument into several
\medskip
steps.

(1) Consider the linear manifold $N$, which is the
projection of $G(\A )\cap G(\LL )$ onto the first component of $H\oplus
H$. Then
$$
N=\{x\in \cDA\cap \cDL: \quad \A x=\LL x\}
$$
and \beqn \label{dim nn} \dim \cDA/N=\dim \cDL/N=n. \neqn
Put $\A^{-1}\LL \defin I+K$, so that $K:\cDL\to H$.
One has
$$
y\in N  \;\Longleftrightarrow\; \big( y\in\cDL \, \;\&\; \,
\A^{-1}\LL y=y \big) \;\Longleftrightarrow\; \big(y\in\cDL  \;\&\; Ky=0\big).
$$
By \eqref{dim nn},  $\rank K=n$.

Next, there exists a factorization
$$
K=\tilde a \wt K,
$$
where \beqn \label{tilde a and K} \wt K:\cDL\to \BC^n, \enspace
\Ran \wt K=\BC^n; \qquad \tilde a: \BC^n\to H,  \enspace \ker
\tilde a=0. \neqn Any $y\in \cDL$ can be represented as \beqn
\label{repr y} y=\A^{-1}\LL y-Ky=y_0+\tilde ac, \quad
\text{where}\enspace y_0\defin \A^{-1}\LL y \in \cDA, \; c\defin -\wt
Ky\in \BC^n. \neqn In particular,
$$
\cDL\subseteq \cDA+\Ran \wt a.
$$
We put
$$
a=\A \tilde a: \BC^n\to \A H.
$$
\medskip

(2) Consider the linear manifold
$$
\big\{ (c,y_0)\in \BC^n \oplus \cDA: \quad \wt ac+y_0\in \cDL,
\enspace \LL \big(\tilde a c+y_0\big)=\A y_0 \big\}.
$$
Since this linear manifold has finite codimension in $\BC^n \oplus
\cDA$, it has a form
$$
\big\{ (c,y_0)\in \BC^n \oplus \cDA: \deab c+b^*y=0 \big\}
$$
for some $b^*:\cDA\to \BC^m$ and some complex matrix $\deab \in
\BC^{m\times n}$ (here $m\ge0$ is an integer). So, for any pair
$(c,y_0)\in \BC^n \oplus \cDA$, one has
\beqn
\label{equiv}
\Big[
\wt ac+y_0\in \cDL \enspace \& \enspace \LL \big(\tilde a c+y_0\big)=\A y_0
\Big]
\;\Longleftrightarrow\; \deab c+b^*y_0=0.
\neqn
By putting here
$c=0$, we get that for $y_0\in \cDA$, the condition $y_0\in N$ is
equivalent to the condition $b^*y_0=0$. Hence \beqn \label{rank b
n} \rank b^*=n.
\neqn
\smallskip

(3) Take any $y\in \cDL$, and define its decomposition
$y=y_0+\tilde ac$ as in \eqref{repr y}. For these $y,y_0$ and $c$,
both conditions of the left-hand side of the equivalence
\eqref{equiv} hold. By \eqref{repr y}, any $c\in \BC^n$ can appear
in this way. Hence $\Ran\deab \subseteq \Ran b^*$.

By \eqref{rank b n}, there is a matrix $\Phi\in \BC^{n\times m}$
such that $\Phi|_{\Ran b^*}$ is an isomorphism. Then the condition
$\deab c+b^*y=0$ is equivalent to $\Phi\deab c+\Phi b^*y=0$,
and we can replace the pair $(\deab ,b)$ with $(\Phi\deab ,
\Phi b)$. In other words, it can be assumed that $m=n$, so that
$\deab \in \BC^{n\times n}$. We assert that $\LL =\LL (\A ,a,b,\deab )$.

Let us check first that $\LL (\A ,a,b,\deab )$ is correctly defined and
$\LL (\A ,a,b,\deab )\subseteq \LL $. To this end, take any
$y=y_0+\A^{-1}ac\in \cD(\LL (\A ,a,b,\deab ))$. Then $\deab c+b^*y_0=0$. By \eqref{equiv}, $y\in \cDL$, and $y_0=\A^{-1}\LL y$.
This shows that $y$ determines $y_0$ uniquely (and therefore
 $(\admn)$  holds), and that $\LL (\A ,a,b,\deab )\subseteq \LL $.

Now we check that $\LL \subseteq \LL (\A ,a,b,\deab )$. Take any $y\in
\cDL$. By
\eqref{equiv}, $c\in \BC^n$, and that $\LL y=\A y_0$. By \eqref{equiv},
$y\in \cD\big(\LL (\A ,a,b,\deab )\big)$ and $\LL (\A ,a,b,\deab ) y=\LL y$.
\medskip

We leave the proof of assertion (1) to the reader. It follows
essentially the same type of arguments.
\end{proof}


%
%

\begin{proof}[Proof of Proposition~\ref{Pr-adjoints-n}]
(1)\,
We conserve the notation of the previous proof. It was shown there that the space
$$
N = \big\{ y_0\in \cDA: b^*y_0=0 \big\}
$$
has codimension $n$ in $\cDA$. If we understand $b$ as an operator
$b:\cDA\to \BC^n$, then
$$
\big\{ f\in \A H: f\perp N \big\} =\Ran b.
$$
It follows that $(\clos_H N)^\perp =(\Ran b)\cap H$.

Since $N\subset \cDL$, it follows that any vector $v\in H$ which
is orthogonal to $\clos\cDL$, should have the form $v=bd$,
some $d\in \BC^n$  such that $v\in H$.
So $\LL $ is densely defined if and only if
the only vector of the form $v=bd\in H$,
$d\in \BC^n$ which is orthogonal to $\cDL$ corresponds to $d=0$.

For any $v=bd$ as above, orthogonal to
$\cDL$, and any $y=y_0+\A^{-1}ac\in \cDL$
$$
0=\langle y, v \rangle
=
\langle b^*y_0, d \rangle +
  \langle c, a^*(\A^{*-1}bd) \rangle
=
\langle c, -\deab^* d+a^*(\A^{*-1}bd) \rangle.
$$
It follows
from the proof of Proposition~\ref{Pr1add} that here $c\in \BC^n$ can be arbitrary. Therefore
$\LL $ is densely defined if and only if $(\admnst)$  holds.


(2)\,
Now assume that both conditions  $(\admn)$ ,  $(\admnst)$  hold. Then both operators
$\LL^*$ and $\LL_*\defin \LL (\A^*, b,a,\deab^*)$ are well-defined.
We have show that $\LL_*=\LL^*$. Notice first that
\beqn
\begin{aligned}
\label{2app-st}
\cD(\LL_*)&\defin \big\{
u=u_0+\A^{*-1}bd: \\
& \qquad \qquad d\in \BC^n,\, u_0\in \cD(\A^*),\, \deab^* d+a^*u_0=0
\big\};                     \\
\LL_* u& \defin \A^*u_0, \quad y\in \cDL.
\end{aligned}
\neqn
Now take any pair of vectors $y=y_0+\A^{-1}ac\in \cDL$, $u=u_0+\A^{*-1}bd\in \cDLst$.
Then $\LL y=\A y_0$, $\LL_*y=\A u_0$.
By \eqref{2app}, \eqref{2app-st},
$$
\langle \LL y, u \rangle
- \langle \A y_0, u_0 \rangle
=\langle \A y_0, \A^{*-1}bd \rangle
=\langle b^*y_0, d \rangle
=- \langle \deab c, d \rangle.
$$
Similarly,
$$
\langle y, \LL_* u \rangle
- \langle y_0, \A^* u_0 \rangle
=- \langle \deab c, d \rangle,
$$
so that
$\langle \LL y, u \rangle = \langle y, \LL_* u \rangle$ for all
vectors $y\in \cDL$, $u\in \cDLst$. Hence $\LL_*\subset \LL^*$.
It is easy to see that if $R,S$ are subspaces of $H$ and
$$
\dim (R\cap S)/S= \dim (R\cap S)/R=n,
$$
then the same relations hold for the orthogonal complements
$R^\perp$, $S^\perp$. Denote by $\Gr(\LL )$ the graph of $\LL $ and
introduce a unitary operator $W$ on $H\oplus H$ by the formula $W(x,y)=(-y,x)$.
It is well-known that  $\big(W\Gr(\LL )\big)^\perp=\Gr(\LL^*)$.
It follows that both $\LL_*$ and $\LL^*$ are rank $n$ balanced
perturbations of $\A^*$. Since
$\Gr(\LL_*)\subset \Gr(\LL^*)$, it follows that
$\LL_* = \LL^*$.
\end{proof}


\end{document}